\documentclass[11pt,preprint]{imsart}

\usepackage{enumerate}

\usepackage[mathscr]{eucal}
\usepackage{graphicx}
\usepackage{latexsym}
\usepackage{amssymb}
\usepackage{psfrag}
\usepackage{epsfig}
\usepackage{enumerate}
\usepackage{amsmath}
\usepackage{amsthm}
\DeclareMathAlphabet{\mathpzc}{OT1}{pzc}{m}{it}

\include{psbox}

\usepackage{amsfonts}
\usepackage{graphicx}
\usepackage{color}
\usepackage[usenames,dvipsnames]{xcolor}
\usepackage[textsize=small]{todonotes}
\usepackage{mathrsfs}
\usepackage{hyperref}

\usepackage[numbers]{natbib}

\usepackage{lipsum}
\usepackage{titlesec}

\titlespacing\section{0pt}{12pt plus 4pt minus 2pt}{0pt plus 2pt minus 2pt}
\titlespacing\subsection{0pt}{12pt plus 4pt minus 2pt}{0pt plus 2pt minus 2pt}
\titlespacing\subsubsection{0pt}{12pt plus 4pt minus 2pt}{0pt plus 2pt minus 2pt}
\begin{document}

\newtheorem{theorem}{\bf Theorem}[section]
\newtheorem{definition}{\bf Definition}[section]
\newtheorem{corollary}[theorem]{\bf Corollary}
\newtheorem{lemma}[theorem]{\bf Lemma}
\newtheorem{assumption}{Assumption}
\newtheorem{condition}{\bf Condition}[section]
\newtheorem{proposition}[theorem]{\bf Proposition}
\newtheorem{definitions}{\bf Definition}[section]
\newtheorem{problem}{\bf Problem}

\theoremstyle{remark}
\newtheorem{example}{\bf Example}[section]
\newtheorem{remark}{\bf Remark}[section]

\newcounter{constant}
\setcounter{constant}{1}
\newcommand{\const}{{\arabic{constant}}}
\newcommand{\constplus}{{\stepcounter{constant}\const}}
\newcommand{\constminus}{{\addtocounter{constant}{-1}\const}}

\newcommand{\beginsec}{
\setcounter{lemma}{0} \setcounter{theorem}{0}
\setcounter{corollary}{0} \setcounter{definition}{0}
\setcounter{example}{0} \setcounter{proposition}{0}
\setcounter{condition}{0} \setcounter{assumption}{0}
\setcounter{remark}{0} }

\numberwithin{equation}{section}
\numberwithin{theorem}{section}

\newcommand{\Erdos}{Erd\H{o}s-R\'enyi }
\newcommand{\Meleard}{M\'el\'eard }
\newcommand{\Frechet}{Fr\'echet }

\newcommand{\one} {{\boldsymbol{1}}}
\newcommand{\half}{{\frac{1}{2}}}
\newcommand{\quarter}{{\frac{1}{4}}}
\newcommand{\skp}{\vspace{\baselineskip}}
\newcommand{\noi}{\noindent}
\newcommand{\lan}{\langle}
\newcommand{\ran}{\rangle}
\newcommand{\lfl}{\lfloor}
\newcommand{\rfl}{\rfloor}
\newcommand{\ti}{\tilde}
\newcommand{\eps}{{\varepsilon}}
\newcommand{\gr}{{\nabla}}  
\newcommand{\Rd}{{\Rmb^d}}
\newcommand{\intR}{\int_\Rmb}
\newcommand{\intRd}{\int_\Rd}

\newcommand{\Amb}{{\mathbb{A}}}
\newcommand{\Bmb}{{\mathbb{B}}}
\newcommand{\Cmb}{{\mathbb{C}}}
\newcommand{\Dmb}{{\mathbb{D}}}
\newcommand{\Emb}{{\mathbb{E}}}
\newcommand{\Fmb}{{\mathbb{F}}}
\newcommand{\Gmb}{{\mathbb{G}}}
\newcommand{\Hmb}{{\mathbb{H}}}
\newcommand{\Imb}{{\mathbb{I}}}
\newcommand{\Jmb}{{\mathbb{J}}}
\newcommand{\Kmb}{{\mathbb{K}}}
\newcommand{\Lmb}{{\mathbb{L}}}
\newcommand{\Mmb}{{\mathbb{M}}}
\newcommand{\Nmb}{{\mathbb{N}}}
\newcommand{\Omb}{{\mathbb{O}}}
\newcommand{\Pmb}{{\mathbb{P}}}
\newcommand{\Qmb}{{\mathbb{Q}}}
\newcommand{\Rmb}{{\mathbb{R}}}
\newcommand{\Smb}{{\mathbb{S}}}
\newcommand{\Tmb}{{\mathbb{T}}}
\newcommand{\Umb}{{\mathbb{U}}}
\newcommand{\Vmb}{{\mathbb{V}}}
\newcommand{\Wmb}{{\mathbb{W}}}
\newcommand{\Xmb}{{\mathbb{X}}}
\newcommand{\Ymb}{{\mathbb{Y}}}
\newcommand{\Zmb}{{\mathbb{Z}}}

\newcommand{\Amc}{{\mathcal{A}}}
\newcommand{\Bmc}{{\mathcal{B}}}
\newcommand{\Cmc}{{\mathcal{C}}}
\newcommand{\Dmc}{{\mathcal{D}}}
\newcommand{\Emc}{{\mathcal{E}}}
\newcommand{\Fmc}{{\mathcal{F}}}
\newcommand{\Gmc}{{\mathcal{G}}}
\newcommand{\Hmc}{{\mathcal{H}}}
\newcommand{\Imc}{{\mathcal{I}}}
\newcommand{\Jmc}{{\mathcal{J}}}
\newcommand{\Kmc}{{\mathcal{K}}}
\newcommand{\lmc}{{\mathcal{l}}}
\newcommand{\Lmc}{{\mathcal{L}}}
\newcommand{\Mmc}{{\mathcal{M}}}
\newcommand{\Nmc}{{\mathcal{N}}}
\newcommand{\Omc}{{\mathcal{O}}}
\newcommand{\Pmc}{{\mathcal{P}}}
\newcommand{\Qmc}{{\mathcal{Q}}}
\newcommand{\Rmc}{{\mathcal{R}}}
\newcommand{\Smc}{{\mathcal{S}}}
\newcommand{\Tmc}{{\mathcal{T}}}
\newcommand{\Umc}{{\mathcal{U}}}
\newcommand{\Vmc}{{\mathcal{V}}}
\newcommand{\Wmc}{{\mathcal{W}}}
\newcommand{\Xmc}{{\mathcal{X}}}
\newcommand{\Ymc}{{\mathcal{Y}}}
\newcommand{\Zmc}{{\mathcal{Z}}}

\newcommand{\abd}{{\boldsymbol{a}}}
\newcommand{\Abd}{{\boldsymbol{A}}}
\newcommand{\bbd}{{\boldsymbol{b}}}
\newcommand{\Bbd}{{\boldsymbol{B}}}
\newcommand{\cbd}{{\boldsymbol{c}}}
\newcommand{\Cbd}{{\boldsymbol{C}}}
\newcommand{\dbd}{{\boldsymbol{d}}}
\newcommand{\Dbd}{{\boldsymbol{D}}}
\newcommand{\ebd}{{\boldsymbol{e}}}
\newcommand{\Ebd}{{\boldsymbol{E}}}
\newcommand{\fbd}{{\boldsymbol{f}}}
\newcommand{\Fbd}{{\boldsymbol{F}}}
\newcommand{\gbd}{{\boldsymbol{g}}}
\newcommand{\Gbd}{{\boldsymbol{G}}}
\newcommand{\hbd}{{\boldsymbol{h}}}
\newcommand{\Hbd}{{\boldsymbol{H}}}
\newcommand{\ibd}{{\boldsymbol{i}}}
\newcommand{\Ibd}{{\boldsymbol{I}}}
\newcommand{\jbd}{{\boldsymbol{j}}}
\newcommand{\Jbd}{{\boldsymbol{J}}}
\newcommand{\kbd}{{\boldsymbol{k}}}
\newcommand{\Kbd}{{\boldsymbol{K}}}
\newcommand{\lbd}{{\boldsymbol{l}}}
\newcommand{\Lbd}{{\boldsymbol{L}}}
\newcommand{\mbd}{{\boldsymbol{m}}}
\newcommand{\Mbd}{{\boldsymbol{M}}}
\newcommand{\mubd}{{\boldsymbol{\mu}}}
\newcommand{\nbd}{{\boldsymbol{n}}}
\newcommand{\Nbd}{{\boldsymbol{N}}}
\newcommand{\nubd}{{\boldsymbol{\nu}}}
\newcommand{\Nalpha}{{\boldsymbol{N_\alpha}}}
\newcommand{\Nbeta}{{\boldsymbol{N_\beta}}}
\newcommand{\Ngamma}{{\boldsymbol{N_\gamma}}}
\newcommand{\obd}{{\boldsymbol{o}}}
\newcommand{\Obd}{{\boldsymbol{O}}}
\newcommand{\pbd}{{\boldsymbol{p}}}
\newcommand{\Pbd}{{\boldsymbol{P}}}
\newcommand{\phibd}{{\boldsymbol{\phi}}}
\newcommand{\phihatbd}{{\boldsymbol{\hat{\phi}}}}
\newcommand{\psibd}{{\boldsymbol{\psi}}}
\newcommand{\psihatbd}{{\boldsymbol{\hat{\psi}}}}
\newcommand{\qbd}{{\boldsymbol{q}}}
\newcommand{\Qbd}{{\boldsymbol{Q}}}
\newcommand{\rbd}{{\boldsymbol{r}}}
\newcommand{\Rbd}{{\boldsymbol{R}}}
\newcommand{\sbd}{{\boldsymbol{s}}}
\newcommand{\Sbd}{{\boldsymbol{S}}}
\newcommand{\tbd}{{\boldsymbol{t}}}
\newcommand{\Tbd}{{\boldsymbol{T}}}
\newcommand{\ubd}{{\boldsymbol{u}}}
\newcommand{\Ubd}{{\boldsymbol{U}}}
\newcommand{\vbd}{{\boldsymbol{v}}}
\newcommand{\Vbd}{{\boldsymbol{V}}}
\newcommand{\wbd}{{\boldsymbol{w}}}
\newcommand{\Wbd}{{\boldsymbol{W}}}
\newcommand{\xbd}{{\boldsymbol{x}}}
\newcommand{\Xbd}{{\boldsymbol{X}}}
\newcommand{\ybd}{{\boldsymbol{y}}}
\newcommand{\Ybd}{{\boldsymbol{Y}}}
\newcommand{\zbd}{{\boldsymbol{z}}}
\newcommand{\Zbd}{{\boldsymbol{Z}}}

\newcommand{\abar}{{\bar{a}}}
\newcommand{\Abar}{{\bar{A}}}
\newcommand{\Amcbar}{{\bar{\Amc}}}
\newcommand{\bbar}{{\bar{b}}}
\newcommand{\Bbar}{{\bar{B}}}
\newcommand{\cbar}{{\bar{c}}}
\newcommand{\Cbar}{{\bar{C}}}
\newcommand{\dbar}{{\bar{d}}}
\newcommand{\Dbar}{{\bar{D}}}
\newcommand{\ebar}{{\bar{e}}}
\newcommand{\Ebar}{{\bar{E}}}
\newcommand{\ebdbar}{{\bar{\ebd}}}
\newcommand{\Embbar}{{\bar{\Emb}}}
\newcommand{\fbar}{{\bar{f}}}
\newcommand{\Fbar}{{\bar{F}}}
\newcommand{\Fmcbar}{{\bar{\Fmc}}}
\newcommand{\gbar}{{\bar{g}}}
\newcommand{\Gbar}{{\bar{G}}}
\newcommand{\Gammabar}{{\bar{\Gamma}}}
\newcommand{\Gmcbar}{{\bar{\Gmc}}}
\newcommand{\Hbar}{{\bar{H}}}
\newcommand{\ibar}{{\bar{i}}}
\newcommand{\Ibar}{{\bar{I}}}
\newcommand{\jbar}{{\bar{j}}}
\newcommand{\Jbar}{{\bar{J}}}
\newcommand{\Jmcbar}{{\bar{\Jmc}}}
\newcommand{\kbar}{{\bar{k}}}
\newcommand{\Kbar}{{\bar{K}}}
\newcommand{\lbar}{{\bar{l}}}
\newcommand{\Lbar}{{\bar{L}}}
\newcommand{\mbar}{{\bar{m}}}
\newcommand{\Mbar}{{\bar{M}}}
\newcommand{\mubar}{{\bar{\mu}}}
\newcommand{\Mmbbar}{{\bar{\Mmb}}}
\newcommand{\nbar}{{\bar{n}}}
\newcommand{\Nbar}{{\bar{N}}}
\newcommand{\nubar}{{\bar{\nu}}}
\newcommand{\obar}{{\bar{o}}}
\newcommand{\Obar}{{\bar{O}}}
\newcommand{\omegabar}{{\bar{\omega}}}
\newcommand{\Omegabar}{{\bar{\Omega}}}
\newcommand{\pbar}{{\bar{p}}}
\newcommand{\Pbar}{{\bar{P}}}
\newcommand{\Phibar}{{\bar{\Phi}}}
\newcommand{\Pmbbar}{{\bar{\Pmb}}}
\newcommand{\Pmcbar}{{\bar{\Pmc}}}
\newcommand{\qbar}{{\bar{q}}}
\newcommand{\Qbar}{{\bar{Q}}}
\newcommand{\rbar}{{\bar{r}}}
\newcommand{\Rbar}{{\bar{R}}}
\newcommand{\Rmcbar}{{\bar{\Rmc}}}
\newcommand{\sbar}{{\bar{s}}}
\newcommand{\Sbar}{{\bar{S}}}
\newcommand{\sigmabar}{{\bar{\sigma}}}
\newcommand{\Smcbar}{{\bar{\Smc}}}
\newcommand{\tbar}{{\bar{t}}}
\newcommand{\Tbar}{{\bar{T}}}
\newcommand{\taubar}{{\bar{\tau}}}
\newcommand{\Thetabar}{{\bar{\Theta}}}
\newcommand{\ubar}{{\bar{u}}}
\newcommand{\Ubar}{{\bar{U}}}
\newcommand{\vbar}{{\bar{v}}}
\newcommand{\Vbar}{{\bar{V}}}
\newcommand{\wbar}{{\bar{w}}}
\newcommand{\Wbar}{{\bar{W}}}
\newcommand{\xbar}{{\bar{x}}}
\newcommand{\Xbar}{{\bar{X}}}
\newcommand{\Xitbar}{{\bar{X}_t^i}}
\newcommand{\Xjtbar}{{\bar{X}_t^j}}
\newcommand{\Xktbar}{{\bar{X}_t^k}}
\newcommand{\Xisbar}{{\bar{X}_s^i}}
\newcommand{\ybar}{{\bar{y}}}
\newcommand{\Ybar}{{\bar{Y}}}
\newcommand{\zbar}{{\bar{z}}}
\newcommand{\Zbar}{{\bar{Z}}}

\newcommand{\Ahat}{{\hat{A}}}
\newcommand{\bhat}{{\hat{b}}}
\newcommand{\etahat}{{\hat{\eta}}}
\newcommand{\fhat}{{\hat{f}}}
\newcommand{\ghat}{{\hat{g}}}
\newcommand{\hhat}{{\hat{h}}}
\newcommand{\Jhat}{{\hat{J}}}
\newcommand{\Jmchat}{{\hat{\Jmc}}}
\newcommand{\lambdahat}{{\hat{\lambda}}}
\newcommand{\lhat}{{\hat{l}}}
\newcommand{\muhat}{{\hat{\mu}}}
\newcommand{\nuhat}{{\hat{\nu}}}
\newcommand{\phihat}{{\hat{\phi}}}
\newcommand{\psihat}{{\hat{\psi}}}
\newcommand{\rhohat}{{\hat{\rho}}}
\newcommand{\Smchat}{{\hat{\Smc}}}
\newcommand{\Tmchat}{{\hat{\Tmc}}}
\newcommand{\Yhat}{{\hat{Y}}}

\newcommand{\atil}{{\tilde{a}}}
\newcommand{\Atil}{{\tilde{A}}}
\newcommand{\btil}{{\tilde{b}}}
\newcommand{\Btil}{{\tilde{B}}}
\newcommand{\ctil}{{\tilde{c}}}
\newcommand{\Ctil}{{\tilde{C}}}
\newcommand{\dtil}{{\tilde{d}}}
\newcommand{\Dtil}{{\tilde{D}}}
\newcommand{\etil}{{\tilde{e}}}
\newcommand{\Etil}{{\tilde{E}}}
\newcommand{\ebdtil}{{\tilde{\ebd}}}
\newcommand{\etatil}{{\tilde{\eta}}}
\newcommand{\Embtil}{{\tilde{\Emb}}}
\newcommand{\ftil}{{\tilde{f}}}
\newcommand{\Ftil}{{\tilde{F}}}
\newcommand{\Fmctil}{{\tilde{\Fmc}}}
\newcommand{\gtil}{{\tilde{g}}}
\newcommand{\Gtil}{{\tilde{G}}}
\newcommand{\gammatil}{{\tilde{\gamma}}}
\newcommand{\Gmctil}{{\tilde{\Gmc}}}
\newcommand{\htil}{{\tilde{h}}}
\newcommand{\Htil}{{\tilde{H}}}
\newcommand{\itil}{{\tilde{i}}}
\newcommand{\Itil}{{\tilde{I}}}
\newcommand{\jtil}{{\tilde{j}}}
\newcommand{\Jtil}{{\tilde{J}}}
\newcommand{\Jmctil}{{\tilde{\Jmc}}}
\newcommand{\ktil}{{\tilde{k}}}
\newcommand{\Ktil}{{\tilde{K}}}
\newcommand{\ltil}{{\tilde{l}}}
\newcommand{\Ltil}{{\tilde{L}}}
\newcommand{\mtil}{{\tilde{m}}}
\newcommand{\Mtil}{{\tilde{M}}}
\newcommand{\mutil}{{\tilde{\mu}}}
\newcommand{\ntil}{{\tilde{n}}}
\newcommand{\Ntil}{{\tilde{N}}}
\newcommand{\nutil}{{\tilde{\nu}}}
\newcommand{\nubdtil}{{\tilde{\nubd}}}
\newcommand{\otil}{{\tilde{o}}}
\newcommand{\Otil}{{\tilde{O}}}
\newcommand{\omegatil}{{\tilde{\omega}}}
\newcommand{\Omegatil}{{\tilde{\Omega}}}
\newcommand{\ptil}{{\tilde{p}}}
\newcommand{\Ptil}{{\tilde{P}}}
\newcommand{\Pmbtil}{{\tilde{\Pmb}}}
\newcommand{\qtil}{{\tilde{q}}}
\newcommand{\Qtil}{{\tilde{Q}}}
\newcommand{\rtil}{{\tilde{r}}}
\newcommand{\Rtil}{{\tilde{R}}}
\newcommand{\rbdtil}{{\tilde{\rbd}}}
\newcommand{\Rmctil}{{\tilde{\Rmc}}}
\newcommand{\stil}{{\tilde{s}}}
\newcommand{\Stil}{{\tilde{S}}}
\newcommand{\sigmatil}{{\tilde{\sigma}}}
\newcommand{\ttil}{{\tilde{t}}}
\newcommand{\Ttil}{{\tilde{T}}}
\newcommand{\Tmctil}{{\tilde{\Tmc}}}
\newcommand{\thetatil}{{\tilde{\theta}}}
\newcommand{\util}{{\tilde{u}}}
\newcommand{\Util}{{\tilde{U}}}
\newcommand{\vtil}{{\tilde{v}}}
\newcommand{\Vtil}{{\tilde{V}}}
\newcommand{\Vmctil}{{\tilde{\Vmc}}}
\newcommand{\wtil}{{\tilde{w}}}
\newcommand{\Wtil}{{\tilde{W}}}
\newcommand{\xtil}{{\tilde{x}}}
\newcommand{\Xtil}{{\tilde{X}}}
\newcommand{\xitil}{{\tilde{\xi}}}
\newcommand{\Xittil}{{\tilde{X}^{i}_t}}
\newcommand{\Xistil}{{\tilde{X}^{i}_s}}
\newcommand{\ytil}{{\tilde{y}}}
\newcommand{\Ytil}{{\tilde{Y}}}
\newcommand{\ztil}{{\tilde{z}}}
\newcommand{\Ztil}{{\tilde{Z}}}
\newcommand{\zetatil}{{\tilde{\zeta}}}
\newcommand{\Zittil}{{\tilde{Z}^{i,N}_t}}
\newcommand{\Zistil}{{\tilde{Z}^{i,N}_s}}

%Macro for MDP jumps
\newcommand{\Xt}{{\Xmb_t}}
\newcommand{\XT}{{\Xmb_T}}
\newcommand{\YT}{{\Ymb_T}}
\newcommand{\mumt}{{\mu^m(t)}}
\newcommand{\mums}{{\mu^m(s)}}
\newcommand{\Gammainf}{{\| \Gamma \|_\infty}}
\newcommand{\intXt}{{\int_\Xt}}
\newcommand{\intXT}{{\int_\XT}}

%Macro for CLT multiple populations
\newcommand{\EmbP}{{\Emb_{\Pmb^N}}}
\newcommand{\EmbPG}{{\Emb_{\Pmb^N,\Gmc}}}
\newcommand{\UiNt}{{U^{i,N}_t}}
\newcommand{\UiNs}{{U^{i,N}_s}}
\newcommand{\Wit}{{W^i_t}}
\newcommand{\Wis}{{W^i_s}}
\newcommand{\Xit}{{X^i_t}}
\newcommand{\Xis}{{X^i_s}}
\newcommand{\Xjt}{{X^j_t}}
\newcommand{\Xjs}{{X^j_s}}
\newcommand{\Xkt}{{X^k_t}}
\newcommand{\Xks}{{X^k_s}}
\newcommand{\XiNt}{{X^{i,N}_t}}
\newcommand{\XiNs}{{X^{i,N}_s}}
\newcommand{\XjNt}{{X^{j,N}_t}}
\newcommand{\XjNs}{{X^{j,N}_s}}
\newcommand{\XkNt}{{X^{k,N}_t}}
\newcommand{\XkNs}{{X^{k,N}_s}}
\newcommand{\YiNt}{{Y^{i,N}_t}}
\newcommand{\YiNs}{{Y^{i,N}_s}}
\newcommand{\Yit}{{Y_t^i}}
\newcommand{\Yis}{{Y_s^i}}
\newcommand{\Yjt}{{Y_t^j}}
\newcommand{\Ykt}{{Y_t^k}}
\newcommand{\ZiNt}{{Z_t^{i,N}}}
\newcommand{\ZjNt}{{Z_t^{j,N}}}
\newcommand{\ZiNs}{{Z_s^{i,N}}}
\newcommand{\ZjNs}{{Z_s^{j,N}}}

%Macro for MDP diffusions
\newcommand{\Xtilims}{{\Xtil_i^m(s)}}
\newcommand{\Xtiljms}{{\Xtil_j^m(s)}}
\newcommand{\Xtilimt}{{\Xtil_i^m(t)}}
\newcommand{\Xtiljmt}{{\Xtil_j^m(t)}}
\newcommand{\Xbaris}{{\Xbar_i(s)}}
\newcommand{\Xbarjs}{{\Xbar_j(s)}}
\newcommand{\Xbarit}{{\Xbar_i(t)}}
\newcommand{\Xbarjt}{{\Xbar_j(t)}}
\newcommand{\mutilms}{{\mutil^m(s)}}
\newcommand{\mutilmt}{{\mutil^m(t)}}
\newcommand{\mubarms}{{\mubar^m(s)}}
\newcommand{\mubarmt}{{\mubar^m(t)}}

\begin{frontmatter}
\title{Moderate Deviation Principles for Weakly Interacting Particle Systems\thanks{Research supported in part by the National Science Foundation (DMS-1016441, DMS-1305120) and the Army Research Office (W911NF-14-1-0331)}}

 \runtitle{MDP for Weakly Interacting Particle Systems}

\begin{aug}
\author{Amarjit Budhiraja and Ruoyu Wu\\ \ \\}
\end{aug}

\today

\skp

\begin{abstract}
Moderate deviation principles for empirical measure processes associated with weakly interacting Markov processes are established.
Two families of models are considered: the first corresponds to a system of interacting diffusions whereas the second describes a collection of pure jump Markov processes with a countable state space.
For both cases the moderate deviation principle is formulated in terms of a large deviation principle (LDP), with an appropriate speed function, for suitably centered and normalized empirical measure processes.
For the first family of models the LDP is established in the path space of an appropriate Schwartz distribution space whereas for the second family the LDP is proved in the space of $l_2$ (the Hilbert space of square summable sequences)-valued paths.
Proofs rely on certain variational representations for exponential functionals of Brownian motions and Poisson random measures.

\noi {\bf AMS 2000 subject classifications:} 60F10, 60K35, 60J75, 60J60.

\noi {\bf Keywords:} moderate deviations, large deviations, Laplace principle, variational representations, weakly interacting jump-diffusions, nonlinear Markov processes, mean field asymptotics, Schwartz distributions, Poisson random measures.
\end{abstract}

\end{frontmatter}

%----------------------------------------------------

\section{Introduction} \label{sec:intro}

For $m \in \Nmb$, consider a collection of stochastic processes $\{X_i^m\}_{i=1}^m$, representing trajectories of $m$ interacting particles, given as the solution of stochastic differential equations (SDE) driven by mutually independent Poisson random measures (PRM) or Brownian motions (BM).
The interaction between particles occurs through the coefficients of the SDE in that these coefficients depend, in addition to the particle's current state, on the empirical distribution of all particles in the collection.
The form of the interaction is such that the stochastic processes $(X_1^m,\dotsc,X_m^m)$ are exchangeable if the initial distribution of the $m$ particles is exchangeable.
Such stochastic systems, commonly referred to as {\it weakly interacting Markov processes}, date back to the classical works of Boltzmann, Vlasov, McKean and others (see~\cite{Sznitman1991,Kolokoltsov2010} and references therein).
Although originally motivated by problems in statistical physics, over the past few decades, such models have arisen in many different application areas, including communication networks, mathematical finance, chemical and biological systems, and social sciences.
For an extensive list of references to such applications see~\cite{BudhirajaDupuisFischerRamanan2015limits,BudhirajaWu2015}.

There have been many works that study law of large numbers (LLN) results, propagation of chaos (POC) property, and central limit theorems (CLT) for such models.
These include McKean~\cite{McKean1966class,McKean1967propagation}, Braun and Hepp~\cite{BraunHepp1977vlasov}, Dawson~\cite{Dawson1983critical}, Tanaka~\cite{Tanaka1984limit}, Oelschal\"{a}ger~\cite{Oelschlager1984martingale}, Sznitman~\cite{Sznitman1984,Sznitman1991}, Graham and M\'{e}l\'{e}ard~\cite{GrahamMeleard1997}, Shiga and Tanaka~\cite{ShigaTanaka1985}, M\'{e}l\'{e}ard~\cite{Meleard1998}.
Many variations of such systems have also been studied.
For example, in~\cite{KurtzXiong1999,KurtzXiong2004,BudhirajaSaha2014,BudhirajaWu2015} limit theorems of the above form are established for a setting where there is a common noise process that influences the dynamics of each particle.
LLN and POC in a setting with $K$-different subpopulations where particle evolution within each subpopulation is exchangeable, has been studied in~\cite{BaladronFaugeras2012} and a corresponding CLT has been established in~\cite{BudhirajaWu2015}.
Other works that have studied mean field results for such heterogeneous populations include~\cite{Collet2014macroscopic,ChongKluppelberg2015partial}.

Large deviation principles (LDP) for weakly interacting particle systems have also been well studied in many works.
A classical reference is~\cite{DawsonGartner1987large} which considers a collection of diffusing particles with non-degenerate diffusion coefficients that interact through the drift terms.
Proofs are based on discretization arguments together with careful exponential probability estimates.
Alternative methods using weak convergence and certain variational representation formulas have recently been introduced in~\cite{BudhirajaDupuisFischer2012}.
Large deviations for pure jump finite state weakly interacting particle systems have been studied in~\cite{Leonard1995large,BorkarSundaresan2012asymptotics,DupuisRamananWu2012sample}.

In the current work, our focus is on the study of deviations from the law of large numbers limit for such weakly interacting systems that are of smaller order than those captured by a large deviation principle.
Results that give asymptotics of such lower order deviations are usually referred to as moderate deviation principles (MDP).
The object of interest in this work is the empirical measure process $\mu^m(t) \doteq \frac{1}{m} \sum_{i=1}^m \delta_{X_i^m(t)}$.
Denoting the state space of particles by $\Smb$, $\mu^m(t)$ is a random measure, with values in $\Pmc(\Smb)$ (the space of probability measures on $\Smb$ equipped with the weak convergence topology).
In order to motivate the problem of interest, we consider as an illustration the setting where the particle distributions are i.i.d.
Let $\{Y_i\}_{i \in \Nmb}$ be an i.i.d.\ sequence of $\Rd$-valued random variables with distribution $\mu$.
Sanov's theorem that gives an LDP for $L_m \doteq \frac{1}{m} \sum_{i=1}^m \delta_{Y_i}$ formally says that for any measurable set $U$ in $\Pmc(\Rd)$,
\begin{equation*}
	\Pmb(L_m \in U) \approx \exp \{ -m \inf_{\nu \in U} I(\nu) \},
\end{equation*}
where for $\nu \in \Pmc(\Rd)$, $I(\nu) = R(\nu \| \mu) \doteq \int_{\Rd} \frac{d\nu}{d\mu} \big( \log \frac{d\nu}{d\mu} \big) d\mu$ is the relative entropy ($R(\nu\|\mu)$ is taken to be $\infty$ if $\nu$ is not absolutely continuous with respect to $\mu$).
Now let $\{a(m)\}_{m \in \Nmb}$ be a positive sequence such that as $m \to \infty$,
\begin{equation} \label{eqjump:a m}
	a(m) \to 0 \text{ and } a(m)\sqrt{m} \to \infty
\end{equation}
(e.g. $a(m) = m^{-\theta}$ for some $\theta \in (0,1/2)$).
A moderate deviation principle for $\{L_m\}$, associated with deviations of order $\frac{1}{a(m)\sqrt{m}}$, gives a result of the following form (see e.g.~\cite{BorovkovMogulskii1980probabilities}): for any measurable set $U$ in $\Pmc_0(\Rd)$ (the space of signed measures on $\Rd$ such that $\nu(\Rd)=0$),
\begin{equation*}
	\Pmb(a(m)\sqrt{m}(L_m-\mu) \in U) \approx \exp \left\{ -\frac{1}{a^2(m)} \inf_{\nu \in U} I^0(\nu) \right\},
\end{equation*}
where for each $\nu \in \Pmc_0(\Rd)$, $I^0(\nu) \doteq \half \int_\Rd \big(\frac{d\nu}{d\mu}\big)^2 \, d\mu$ (once again we take $I^0(\nu) \doteq \infty$ if $\nu$ is not absolutely continuous with respect to $\mu$).
Note that CLT and LDP provide asymptotics for the probabilities on the left side when $a(m) = m^{-\theta}$ with $\theta = 0$ and $\half$ respectively, whereas an MDP studies an asymptotic regime where $\theta \in (0,\half)$ (an MDP also treats more general scale functions $a(m)$).
There is an extensive literature on moderate deviation results in mathematical statistics, including results for i.i.d. sequences and arrays, empirical processes in general topological spaces, weakly dependent sequences, and occupation measures of Markov chains together with general additive functionals of Markov chains (see~\cite{BudhirajaDupuisGanguly2015moderate} for many such references).
MDP for small noise finite and infinite dimensional SDE with jumps have been studied in~\cite{BudhirajaDupuisGanguly2015moderate}.
References to other MDP results for SDE in the context of stochastic averaging and multi-scale systems can be found in~\cite{BudhirajaDupuisGanguly2015moderate}.
For weakly interacting particle systems in discrete time, MDP based on semigroup analysis and projective large deviation methods have been established in~\cite{DelHuWu2015moderate}.

In the current work we will study moderate deviation principles for continuous time weakly interacting Markov processes.
The models we consider will allow for both Brownian and Poisson type noises in the dynamics.
Our approach is based on certain variational representations for exponential functionals of such noise processes developed in~\cite{BoueDupuis1998variational,BudhirajaDupuis2000variational,BudhirajaDupuisMaroulas2008large,BudhirajaDupuisMaroulas2011variational}.
In order to keep the presentation simple we consider two types of models: one that corresponds to pure jump interacting Markov processes and the other that considers interacting Markov processes with continuous sample paths and Brownian noise.
Although not treated here, one can use similar techniques to develop moderate deviation results for settings that have both Brownian and Poisson type noises.

For the diffusion model considered here, we allow for state dependence in both drift and diffusion coefficients and the interaction through the empirical measure appears in both coefficients as well.
Coefficients are assumed to be bounded with suitable smoothness properties but non-degeneracy of the diffusion term is not required.
This is in contrast to the classical results on large deviation principles for such systems (e.g.~\cite{DawsonGartner1987large}) which only allow interaction in the drift and require the diffusion coefficient to be uniformly non-degenerate.
In order to highlight the main ideas we restrict attention to a one dimensional setting, i.e.\ the case where the state space of the particles is $\Rmb$.
The general multidimensional case can be treated in a similar manner.
Specifically, we consider a collection of one dimensional weakly interacting diffusions $\{X^m_i\}_{i=1}^m$ given by the system of equations:
\begin{equation}
	\label{eqdif:Xim}
	X^m_i(t) = x_0 + \int_0^t \sigma(X^m_i(s), \mu^m(s)) \, dW_i(s) + \int_0^t b(X^m_i(s), \mu^m(s)) \, ds, \quad i = 1, \dotsc, m,
\end{equation}
where $\{W_i\}$ is a sequence of i.i.d.\ one-dimensional standard $\{\Fmc_t\}$-Brownian motions given on some filtered probability space $(\Omega,\Fmc,\Pmb,\{\Fmc_t\})$ and $\mu^m(t) \doteq \frac{1}{m}\sum_{i=1}^m \delta_{X^m_i(t)}$.
Here for $\theta \in \Pmc(\Rmb)$, $\sigma(x,\theta) \doteq \int_{\Rmb} \alpha(x,y) \, \theta(dy)$ and $b(x,\theta) \doteq \int_{\Rmb} \beta(x,y) \, \theta(dy)$, where $\alpha$ and $\beta$ are bounded and Lipschitz.
Law of large numbers, large deviation principle and central limit results for $\mu^m$ have been well studied (see for example~\cite{Sznitman1991,DawsonGartner1987large,HitsudaMitoma1986,Meleard1998} and references therein).
A brief summary of these results is as follows.
As $m \to \infty$, $\mu^m$ converges in $\Cmb([0,T]:\Pmc(\Rmb))$ (the space of $\Pmc(\Rmb)$-valued continuous functions, equipped with the usual uniform topology and any metric on $\Pmc(\Rmb)$ metrizing the weak topology and making it a Polish space), in probability, to $\mu$ where $\mu(t)$ is the common probability distribution of the i.i.d.\ collection $\{\Xbarit\}$ governed by the equation
\begin{equation}
	\label{eqdif:Xbari}
	\bar X_i(t) = x_0 + \int_0^t \sigma(\bar X_i(s), \mu(s)) \, dW_i(s) + \int_0^t b(\bar X_i(s), \mu(s)) \, ds, \quad i \in \Nmb.
\end{equation}
Existence and uniqueness of solutions to \eqref{eqdif:Xim} and \eqref{eqdif:Xbari} are classical (see e.g.~\cite{Sznitman1991}).
The paper~\cite{DawsonGartner1987large} establishes an LDP for $\mu^m$ in $\Cmb([0,T]\!:\!\Pmc(\Rmb))$ with a suitable rate function under the condition that $\sigma(u,\theta) \equiv \sigma(u)$ and $\sigma$ is uniformly non-degenerate.
Some of these conditions on the coefficients can be relaxed (see e.g.~\cite{BudhirajaDupuisFischer2012}).
A central limit theorem studying the asymptotics of the fluctuation process of signed measures $S^m(t) \doteq \sqrt{m}(\mu^m(t) - \mu(t))$ has been established in~\cite{HitsudaMitoma1986,Meleard1998}.
As is well understood (cf.~\cite{HitsudaMitoma1986}), $S^m$ is very irregular as a signed measure-valued process as $m$ becomes large and one cannot expect the limit  in general to be a measure-valued process.
A common approach is to regard $S^m(t,\cdot)$ as an element of a suitable distribution space.
For example, it is shown in~\cite{HitsudaMitoma1986} that, under conditions, $S^m$ converges in distribution as a sequence of $\Cmb([0,T]:\Smc')$-valued random variables, where $\Smc'$ is the dual of the Schwartz space $\Smc$, namely the space of rapidly decaying infinitely smooth functions on $\Rmb$.
This space is equipped with the usual topology given in terms of a countable collection of Hilbertian seminorms $\{ \| \cdot \|_n \}_{n \in \Zmb}$ with associated Hilbert spaces $\{\Smc_n\}_{n \in \Zmb}$.
We refer the reader to Section~\ref{secdif:diffusion} for some basic background on the Schwartz space, but for now it suffices to note the following properties of the collection of Hilbert spaces $\{\Smc_n\}_{n \in \Zmb}$: $\Smc_w \subset \Smc_v$ for $w \ge v$, $\Smc_{-n}$ is the dual of $\Smc_n$, $\Smc'=\bigcup_{n \in \Zmb} \Smc_n$ and $\Smc = \bigcap_{n \in \Zmb} \Smc_n$.
The paper~\cite{HitsudaMitoma1986} shows that (see Theorem $1$ therein), under suitable conditions, for some $v \in \Nmb$, $S^m$ converges in $\Cmb([0,T]: \Smc_{-v})$, in distribution, as $m \to \infty$, to $S$ given as the solution of 
\begin{equation}
	\label{eqdif:S}
	S(t) = Z(t) + \int_0^t L^*(s) S(s) \, ds. 
\end{equation}
Here $Z$ is an $\Smc_{-(v+2)}$-valued Gaussian process with an explicit covariance operator (see~$(4.2)$ in~\cite{HitsudaMitoma1986}) and $L^*(s)$ is the adjoint of the operator $L(s)$ defined as (see also~$(1.4)$ in~\cite{HitsudaMitoma1986})
\begin{equation} \label{eqdif:Ls}
	\begin{aligned}
		(L(s)\phi)(x) & \doteq \phi'(x) b(x,\mu(s)) + \half \phi''(x) \sigma^2(x,\mu(s)) \\
		& \quad + \intR \phi'(y) \beta(y,x) \, \mu_s(dy) + \intR \phi''(y) \sigma(y,\mu(s)) \alpha(y,x) \, \mu_s(dy), \quad \phi \in \Smc.
	\end{aligned}
\end{equation}
Under suitable smoothness conditions on the coefficients, $L(s)$ can be regarded as a bounded linear operator from $\Smc_{v+2}$ to $\Smc_v$ and thus $L^*(s)$ is a bounded linear operator from $\Smc_{-v}$ to $\Smc_{-(v+2)}$.
The equation \eqref{eqdif:S} is interpreted as follows:
For all $\phi \in \Smc$,
\begin{equation*}
	\lan S(t), \phi \ran =  \lan Z(t), \phi \ran + \int_0^t \lan S(s), L(s)\phi \ran \, ds.
\end{equation*}
We remark that~\cite{HitsudaMitoma1986} actually considers a modified version of the Schwartz distribution space which allows one to use unbounded test functions as well, however their results hold for the classical Schwartz space as presented above.
Results of~\cite{HitsudaMitoma1986,Meleard1998} study deviations of $\mu^m$ from $\mu$ that are of order $\frac{1}{\sqrt{m}}$.
In this work we will be concerned with deviations of $\mu^m$ from $\mu$ that are of higher order than $\frac{1}{\sqrt{m}}$ (but of lower order than $\frac{1}{m}$).
Let $\{a(m)\}_{m \in \Nmb}$ be as in \eqref{eqjump:a m} and
\begin{equation} \label{eqdif:Ym}
	Y^m(t) \doteq a(m) S^m(t) = a(m)\sqrt{m}(\mu^m(t) - \mu(t)).
\end{equation}
We will show in Theorem~\ref{thmdif:main} that under Conditions~\ref{conddif:diffusion1} and~\ref{conddif:diffusion3}, $Y^m$ satisfies a large deviation principle with speed $a^2(m)$ (see Section~\ref{sec:notaion} for a precise definition) in $\Cmb([0,T]:\Smc_{-\rho})$ with a suitable value of $\rho > v$.
Roughly speaking, this result says that for any Borel set $U$ in $\Cmb([0,T]:\Smc_{-\rho})$ one has
\begin{equation*}
	\Pmb (Y^m \in U) \approx \exp \left\{ -\frac{1}{a^2(m)} \inf_{\eta \in U} I(\eta)\right\},
\end{equation*}
where $I$ is the associated rate function that will be introduced in \eqref{eqdif:rate function}.
Since it provides asymptotics for probabilities of deviations of $\mu^m$ from $\mu$ that are of order $\frac{1}{a(m)\sqrt{m}}$, this LDP for $Y^m$ can be viewed as a moderate deviation principle for the empirical measure process $\mu^m$.

The proof relies on certain variational formulas for exponential functionals of Brownian motions of the form first obtained in~\cite{BoueDupuis1998variational}.
For our purpose it will be convenient to use the somewhat more general form allowing for an arbitrary filtration, given in~\cite{BudhirajaDupuisMaroulas2008large}.
Using the Laplace formulation of the LDP (see Chapter $1$ of~\cite{DupuisEllis2011weak}) these variational formulas reduce the proof of the upper bound in the LDP to proving tightness and characterization of limit points of certain centered and normalized controlled versions of the original weakly interacting particle system.
These centered and normalized controlled empirical measure processes are regarded as random variables with values in a suitable Schwartz distribution path space, and the main work is in obtaining appropriate estimates for tightness in this path space.
%The main work for the proof of the upper bound is in obtaining suitable estimates for tightness of such centered and normalized controlled empirical measures in the path space of a suitable Schwartz distribution space and of the corresponding control processes.
For the proof of the lower bound we construct a sequence of \textit{asymptotically near optimal} controlled weakly interacting diffusions in which the controls are determined by the i.i.d.\ sequence of nonlinear Markov processes $\{\Xbar_i\}$ in \eqref{eqdif:Xbari}.
By asymptotic near optimality we mean that for each $\eps > 0$, controls can be chosen such that the costs associated with the controlled processes (given through the variational representation) are asymptotically at most $\eps$ greater  than  the Laplace upper bound.

The second family of models considered in this work corresponds to weakly interacting Markov processes with a countable state space.
Consider for $m \in \Nmb$, a pure jump Markov process $\{(X^m_1(t), \dotsc, X^m_m(t)), t \in [0,T]\}$ taking values in $\Nmb^m$ with $X^m_i(0) = x^m_i$. 
The evolution of the process is described through the jump intensities that are given as follows:\\
Given $(X^m_1(t-), \dotsc, X^m_m(t-))$ $= (x_1, \dotsc, x_m) \in \Nmb^m$, for $i \in \{1, \dotsc, m\}$ and $y \in \Nmb$, $y \ne x_i$,
\begin{equation} \label{eqjump:jump intensity}
	(x_1, \dotsc, x_{i-1}, x_i, x_{i+1}, \dotsc, x_m) \mapsto (x_1, \dotsc, x_{i-1}, y, x_{i+1}, \dotsc, x_m)
\end{equation}
at rate $\Gamma_{x_i,y}(\mu^m(t-))$, where $\mu^m(t-) =  \frac{1}{m} \sum_{i=1}^m \delta_{x_i} = \frac{1}{m} \sum_{i=1}^m \delta_{X^m_i(t-)} \in \Pmc(\Nmb)$.
All other forms of jump have rate $0$.
Here $\Gamma(q) \doteq (\Gamma_{ij}(q))_{i,j=1}^\infty$ is a rate matrix for each $q \in \Pmc(\Nmb)$, namely $\Gamma_{ij}(q) \ge 0$ for $i \ne j$ and  $\Gamma_{ii}(q) = - \sum_{j \ne i}^\infty \Gamma_{ij}(q) > -\infty$.
We identify $\Pmc(\Nmb)$ with the simplex 
\begin{equation*}
	\Smchat \doteq \left\{ q = (q_1, q_2, \dotsc) \in l_2 \biggm| \sum_{j=1}^\infty q_j = 1, q_j \ge 0 \ \forall j \in \Nmb \right\}
\end{equation*}
in $l_2$ (the Hilbert space of square-summable sequences, equipped with the usual inner product).
With suitable assumptions on the intensity function $\Gamma$ and the initial configuration of the particles, it can be proved (see Theorem~\ref{thmjump:LLN}) that for each $T > 0$, $\mu^m$ converges to $p$ in $\Dmb([0,T]:l_2)$ (the space of $l_2$-valued functions that are right continuous with left limits, equipped with the usual Skorokhod topology) for a continuous function $p$ characterized as the unique solution of an $l_2$-valued ODE (see \eqref{eqjump:p t}).
We are interested in the asymptotics of the centered and scaled quantity $Z^m \doteq a(m)\sqrt{m}(\mu^m-p)$ regarded as a random variable with values in $\Dmb([0,T]:l_2)$, where $a(m)$ is as defined in \eqref{eqjump:a m}. 
In Theorem~\ref{thmjump:MDP} we will establish a moderate deviation principle for $\mu^m$ which is formulated in terms of an LDP for $Z^m$ that formally says that for any Borel set $U$ in $\Dmb([0,T]:l_2)$
\begin{equation*}
	\Pmb (Z^m \in U) \approx \exp \left\{ -\frac{1}{a^2(m)} \inf_{\eta \in U} \Ibar(\eta)\right\},
\end{equation*}
where $\Ibar$ is the associated rate function introduced in \eqref{eqjump:rate function}.
We also give an alternative expression for the rate function in \eqref{eqjump:rate function alter} which is somewhat easier to interpret in terms of the model parameters.

Our proof of the MDP in this pure jump setting once more relies on certain variational representations.
This time the variational representations are for exponential functionals of Poisson random measures that have been studied in~\cite{BudhirajaDupuisMaroulas2011variational}.
It is easy to see that one can describe the evolution of the particle system using Poisson random measures on a suitable point space.
In particular, for the construction we use here it suffices to consider a PRM on $\Ymb_T \doteq [0,T] \times \Rmb_+^3$.
One can represent the associated empirical measure process $\{\mu^m(t), t \in [0,T]\}$ as a solution of an SDE where the driving noise is described in terms of this PRM.
Moderate deviation principles for SDE in finite and infinite dimensions, driven by a small Poisson noise, have recently been studied in~\cite{BudhirajaDupuisGanguly2015moderate}.
The SDE for the empirical measure $\{\mu^m(t), t \in [0,T]\}$  is indeed a small Poisson noise equation in the Hilbert space $l_2$, however, the coefficients in the equation fail to satisfy the sort of Lipschitz conditions that were crucial in the analysis of~\cite{BudhirajaDupuisGanguly2015moderate}, in fact the coefficients even fail to be continuous (see discussion below \eqref{eqjump:Gi}).
Thus we need new estimates to overcome this lack of regularity in the coefficients, and this is one of the challenges in the proof (see remarks above Lemma~\ref{lemjump:property of l} and above Lemma~\ref{lemjump:Lipschitz of G} and also the proof of Proposition~\ref{propjump:verifying condition part b} which is based on Lemmas~\ref{lemjump:moment bounds on G psi}-\ref{lemjump:h}).
The paper~\cite{BudhirajaDupuisGanguly2015moderate} provided a general sufficient condition for MDP to hold for small noise stochastic dynamical systems.
The proof of Theorem~\ref{thmjump:MDP} proceeds by verifying that this sufficient condition holds for the SDE governing the evolution of $\{\mu^m(t), t \in [0,T]\}$.
Verification of this condition requires establishing weak convergence of certain controlled processes, and establishing these convergence properties is the main content of the proof of Theorem~\ref{thmjump:MDP} which is given in Section~\ref{secjump:pf}.

The paper is organized as follows.
In Section~\ref{secdif:diffusion} we begin by describing our model of weakly interacting diffusions.
Centered and normalized empirical measures are regarded as elements of a suitable distribution space.
We introduce this space and note some of its basic properties.
We then introduce the two main conditions (Conditions~\ref{conddif:diffusion1} and~\ref{conddif:diffusion3}) for the MDP which is given in Theorem~\ref{thmdif:main}.
Proof of this theorem is provided in Section~\ref{secdif:pf}.
In Section~\ref{secjump:jump} we give a precise formulation of the weakly interacting pure jump Markov process studied in this work.
In Section~\ref{secjump:model} we give a convenient representation for the associated empirical measure process in terms of a Poisson random measure on a suitable point space.
This section also presents some basic well-posedness results and a law of large numbers result under a natural condition (Condition~\ref{condjump:cond1}).
The MDP for the empirical measure process in this setting is given in Section~\ref{secjump:MDP}.
The main result is Theorem~\ref{thmjump:MDP} which establishes an MDP for $\mu^m$ under Condition~\ref{condjump:cond2}.
Theorem~\ref{thmjump:alternative expression} gives an alternative expression for the rate function.
Proofs of Theorems~\ref{thmjump:MDP} and~\ref{thmjump:alternative expression} are given in Section~\ref{secjump:pf}.
Finally an Appendix collects some auxiliary results.

%----------------------------------------------------------------------

\subsection{Some notations and definitions} \label{sec:notaion}

The following notation will be used.
For a Polish space $\Smb$, denote the corresponding Borel $\sigma$-field by $\Bmc(\Smb)$, and let $\Pmc(\Smb)$ be the space of probability measures on $\Smb$, equipped with the topology of weak convergence.
A convenient metric for this topology is the bounded-Lipschitz metric $d_{BL}$ defined as
\begin{equation*}
d_{BL}(\nu_1,\nu_2) \doteq \sup_{\|f\|_{BL} \le 1} | \langle \nu_1 - \nu_2, f \rangle |, \quad \nu_1, \nu_2 \in \Pmc(\Smb),
\end{equation*}
where $\langle \mu,f \rangle \doteq \int f \, d\mu$ for a signed measure $\mu$ on $\Smb$ and $\mu$-integrable function $f \colon \Smb \to \Rmb$, and $\|\cdot\|_{BL}$ is the bounded Lipschitz norm, i.e.\ for a real bounded Lipschitz function $f$ on $\Smb$,
\begin{equation*}
	\|f\|_{BL} \doteq \max \{\|f\|_\infty, \|f\|_L\}, \quad \|f\|_\infty \doteq \sup_{x \in \Smb} |f(x)|, \quad \|f\|_L \doteq \sup_{x \ne y} \frac{|f(x)-f(y)|}{d(x,y)},
\end{equation*}
where $d$ is the metric on $\Smb$.
Denote by $\Mmb_b(\Smb)$ and $\Cmb_b(\Smb)$ the space of real bounded $\Bmc(\Smb)/\Bmc(\Rmb)$-measurable functions and real bounded and continuous functions, respectively.
For Banach spaces $B_1$ and $B_2$, $L(B_1,B_2)$ will denote the space of bounded linear operators from $B_1$ to $B_2$.
For a measure $\nu$ on $\Smb$ and a Hilbert space $H$, let $L^2(\Smb,\nu,H)$ denote the space of measurable functions $f$ from $\Smb$ to $H$ such that $\int_\Smb \| f(x) \|^2_H \, \nu(dx) < \infty$, where $\| \cdot \|_H$ is the norm on $H$.
When $H = \Rmb$ and $\Smb$ is clear from the context we write $L^2(\nu)$.
Let $\Cmb_b^k(\Rmb^d)$ be the space of functions on $\Rmb^d$, which have continuous and bounded partial derivatives up to the $k$-th order.

Fix $T < \infty$. 
All stochastic processes will be considered over the time horizon $[0,T]$. 
We will use the notations $\{X_t\}$ and $\{X(t)\}$ interchangeably for stochastic processes.
For a Polish space $\Smb$, denote by $\Cmb([0,T]:\Smb)$ and $\Dmb([0,T]:\Smb)$ the space of continuous functions and right continuous functions with left limits from $[0,T]$ to $\Smb$, endowed with the uniform and Skorokhod topology, respectively. 
When $\Smb$ is a normed space with norm $\| \cdot \|$, for a map $f \colon [0,T] \to \Smb$, let $\|f\|_{*,t} \doteq \sup_{0 \le s \le t} \|f(s)\|$, $t \in [0,T]$.
We say a collection $\{ X^m \}$ of $\Smb$-valued random variables is tight if the distributions of $X^m$ are tight in $\Pmc(\Smb)$.
We use the symbol ``$\Rightarrow$" to denote convergence in distribution.
 
%For $k \in \Nmb$, let $\Cmc_k$ denote $\Cmb_{\Rmb^k}[0,T]$.
%Probability law of an $\Smb$-valued random variable $\eta$ will be denoted as $\Lmc(\eta)$. 
%Expected value under some probability law $\Pmb$ will be denoted as $\Emb_{\Pmb}$.
%For a $\sigma$-finite measure $\nu$ on a Polish space $\Smb$, denote by $L^2_{\Rmb^k}(\Smb,\nu)$ as the Hilbert space of $\nu$-square integrable functions from $\Smb$ to $\Rmb^k$. 
%When $k$ and $\Smb$ are not ambiguous, we will merely write $L^2(\nu)$. 
%The norm in this Hilbert space will be denoted as $\| \cdot \|_{L^2(\Smb,\nu)}$.
%Given a sequence of random variables $X_n$ on some probability space $(\Omega_n, \Fmc_n, P_n)$, $n \ge 1$, we say $X_n$ converges to $0$ in $L^2(\Omega_n, P_n)$ if $\int X_n^2 \, dP_n \to 0$ as $n \to \infty$.
We will usually denote by $\kappa, \kappa_1, \kappa_2, \dotsc$, the constants that appear in various estimates within a proof. 
The value of these constants may change from one proof to another.
We use $\Nmb_0$ to denote the set of non-negative integers.
%Cardinality of a finite set $A$ will be denoted as $|A|$.
%$B_r^d$ denotes the closed ball of radius $r$ in $\Rmb^d$, centered at origin, i.e. $B_r^d = \{ x \in \Rmb^d : \| x \| \le r \}$.
%For $x,y \in \Rmb^d$, let $x \cdot y = \sum_{i=1}^d x_i y_i$.

A function $I \colon \Smb \to [0,\infty]$ is called a rate function on $\Smb$ if for each $M < \infty$, the level set $\{ x \in \Smb : I(x) \le M \}$ is a compact subset of $\Smb$.
Given a collection $\{\alpha(m)\}_{m \in \Nmb}$ of positive reals, a collection $\{ X^m \}$ of $\Smb$-valued random variables is said to satisfy the Laplace principle upper bound (respectively, lower bound) on $\Smb$ with speed $\alpha(m)$ and rate function $I$ if for all $h \in \Cmb_b(\Smb)$
\begin{equation*}
	\limsup_{m \to \infty} \alpha(m) \log \Emb \Big \{ \exp \Big[ -\frac{1}{\alpha(m)} h(X^m) \Big] \Big \} \le -\inf_{x \in \Smb} \{h(x) + I(x) \},
\end{equation*}
and, respectively,
\begin{equation*}
	\liminf_{m \to \infty} \alpha(m) \log \Emb \Big \{ \exp \Big[ -\frac{1}{\alpha(m)} h(X^m) \Big] \Big \} \ge -\inf_{x \in \Smb} \{h(x) + I(x) \}.
\end{equation*}
The Laplace principle is said to hold for $\{ X^m \}$ with speed $\alpha(m)$ and rate function $I$ if both the Laplace upper and lower bounds hold.
It is well known that the family $\{ X^m \}$ satisfies the Laplace principle upper (respectively lower) bound with a rate function $I$ on $\Smb$ if and only if $\{ X^m \}$ satisfies the large deviation upper (respectively lower) bound for all closed sets (respectively open sets) with the rate function $I$.
In particular, the Laplace principle holds with rate function $I$ iff the large deviation principle holds with the same rate function.
For proofs of these statements we refer to Section 1.2 of~\cite{DupuisEllis2011weak}.

%----------------------------------------------------

\section{The diffusion case} \label{secdif:diffusion}

In this section we consider the collection of weakly interacting diffusions $\{X_i^m\}_{i=1}^m$ described by \eqref{eqdif:Xim}.
We are interested in the asymptotic behavior of $Y^m$ defined by \eqref{eqdif:Ym}.
As noted in the introduction, we will regard $Y^m$ as a stochastic process with values in a suitable space of distributions.
The natural space to consider is the standard Schwartz distribution space that is described as follows.

Let $\Smc$ denote the space of functions $\phi \colon \Rmb \to \Rmb$ such that $\phi$ is infinitely differentiable and $|x|^m |\phi^{(k)}(x)| \to 0$ as $|x| \to \infty$ for every $m,k \in \Nmb_0$, where $\phi^{(k)}$ denotes the $k$-th derivative of $\phi$.
On $\Smc$, define a sequence of inner product $\lan \cdot,\cdot \ran_n$ and seminorms $\|\cdot\|_n$, $n \in \Nmb_0$, as
\begin{equation}
	\lan \phi,\psi \ran_n \doteq \sum_{0 \le k \le n} \int_\Rmb (1+x^2)^{2n} \phi^{(k)}(x) \psi^{(k)}(x) \, dx, \quad \|\phi\|_n^2 \doteq \lan \phi,\phi \ran_n, \quad \phi,\psi \in \Smc. \label{eqdif:norm}
\end{equation}
This sequence of seminorms introduces a nuclear \Frechet topology on $\Smc$ (see Gel'fand and Vilenkin \cite{GelfandVilenkin1964}).
Let $\Smc_n$ be the completion of $\Smc$ with respect to $\|\cdot\|_n$.
Let $\Smc'$ and $\Smc_n'$ be the dual space of $\Smc$ and $\Smc_n$, respectively.
Then $\Smc' = \bigcup_{n \in \Nmb_0} \Smc_n'$.
Denote by $\|\cdot\|_{-n}$ the dual norm on $\Smc_{-n} \doteq \Smc_n'$, with corresponding inner product $\lan \cdot,\cdot \ran_{-n}$.
The collection $\{\Smc_n\}_{n \in \Zmb}$ defines a sequence of nested Hilbert spaces with $\Smc_w \subset \Smc_v$ if $w \ge v$.
The main result of this section shows that for a suitable $\rho \in \Nmb$, $\{Y^m\}$ satisfies an LDP in $\Cmb([0,T]: \Smc_{-\rho})$ with speed $a^2(m)$ as introduced in \eqref{eqjump:a m}, namely, for all $F \in C_b(\Cmb([0,T]:\Smc_{-\rho}))$
\begin{equation}
	\label{eqdif:lapasy}
	\lim_{m\to \infty}-a^2(m) \log \Emb \exp \left\{-\frac{1}{a^2(m)} F(Y^m)\right\} = \inf_{\zeta \in \Cmb([0,T]: \Smc_{-\rho}) }\{I(\zeta) + F(\zeta)\}
\end{equation}
for a suitable rate function $I$.
The form of the rate function will be identified in \eqref{eqdif:rate function}.

We make the following assumption on the coefficients $\alpha$ and $\beta$.

\begin{condition} \label{conddif:diffusion1}
%	$\alpha(\cdot,\cdot)$ and $\beta(\cdot,\cdot)$ have continuous and bounded partial derivatives up to the second order.
	$\alpha, \beta \in \Cmb_b^2(\Rmb^2)$.
\end{condition}

It is easy to show that under Condition~\ref{conddif:diffusion1} there is a unique pathwise solution to \eqref{eqdif:Xim}.
In fact under this condition one also has unique solvability of certain controlled analogues of \eqref{eqdif:Xim} that will be used in our proofs.
We now introduce these controlled processes.

Let for each $m \in \Nmb$, $\{u_i^m; i=1,\dotsc,m\}$ be a collection of real-valued $\{\Fmc_t\}$-progressively measurable processes such that
\begin{equation*} 
%\label{eqdif:controlbd on u}
	\Emb \sum_{i=1}^m \int_0^T |u_i^m(s)|^2 \, ds < \infty.
\end{equation*}
We will refer to $\{u^m_i\}$ as control processes.
Define for $t \in [0,T]$,
\begin{equation} \label{eqdif:mutilm}
	\mutilmt \doteq \frac{1}{m}\sum_{i=1}^m \delta_{\Xtilimt},
\end{equation}
where
\begin{align} \label{eqdif:Xtilim}
	\begin{aligned}
		\Xtilimt & = x_0 + \int_0^t \sigma(\Xtilims, \mutilms) \, dW_i(s) + \int_0^t b(\Xtilims, \mutilms) \, ds\\
		&\quad + \int_0^t \sigma(\Xtilims, \mutilms) u_i^m(s) \, ds, \quad i = 1, \dotsc, m. 
	\end{aligned}
\end{align}
It is easy to check that under Condition~\ref{conddif:diffusion1} there is a unique pathwise solution to the system of equations in \eqref{eqdif:Xtilim}.
Define for $t \in [0,T]$,
\begin{equation} \label{eqdif:Ytilm}
	\Ytil^m(t) \doteq a(m)\sqrt{m}(\mutilmt - \mu(t)).
\end{equation}
In Section~\ref{secdif:pf} (see Theorem~\ref{thmdif:tightness}), we will show that under Condition~\ref{conddif:diffusion1}, for every control sequence $\{u_i^m; i=1,\dotsc,m\}_{m\in\Nmb}$ such that
\begin{equation*}
	\sup_{m \in \Nmb} a^2(m) \Emb \sum_{i=1}^m \int_0^T |u_i^m(s)|^2 \, ds < \infty,
\end{equation*}
$\{\Ytil^m\}_{m\in\Nmb}$ is tight in $\Cmb([0,T]:\Smc_{-v})$ for some $v>4$.
Specifically, one can take any $v > 4$ for which there is an $r \in \Nmb, 4 < r < v$, such that $\sum_{j=1}^\infty \|\phi_j^v\|_r^2 < \infty$ and $\sum_{j=1}^\infty \|\phi_j^r\|_4^2 < \infty$, where for $n \in \Zmb$, $\{\phi_j^n\}$ is a complete orthonormal system of $\Smc_n$ (see proof of Theorem~\ref{thmdif:tightness}, above \eqref{eqdif:g}).
We remark that the convergence of the above two series is equivalent to the property that the embedding maps $\Smc_{-4} \to \Smc_{-r}$ and $\Smc_{-r} \to \Smc_{-v}$ are Hilbert-Schmidt.
Let $w \doteq v + 2$.

It will be convenient to introduce another system of seminorms
$|\cdot|_n$ on $\Smc$ as 
$$|\phi|_n \doteq \sum_{0 \le k \le n} \sup_{x} |\phi^{(k)}(x)|.$$
It is easy to check that, for each $n \in \Nmb_0$, there is a $\gamma_0(n) \in (0,\infty)$ such that 
\begin{equation} \label{eqdif:twonorms}
	|\phi|_n \le \gamma_0(n) \|\phi\|_{n+1}.
\end{equation}
We make the following additional assumption on the coefficients $\alpha$ and $\beta$.

\begin{condition} \label{conddif:diffusion2}
	$\alpha$, $\beta$ are $w$-times continuously differentiable and,
	
	(a) $\sup_y |\alpha(\cdot,y)|_w < \infty$ and $\sup_y |\beta(\cdot,y)|_w < \infty$.
	
	(b) $\sup_x \|\alpha(x,\cdot)\|_w < \infty$ and $\sup_x \|\beta(x,\cdot)\|_w < \infty$.
\end{condition}

\begin{remark}
	Conditions~\ref{conddif:diffusion1} and~\ref{conddif:diffusion2} are satisfied for $w$-times continuously differentiable functions $\alpha$ and $\beta$ if the functions along with their derivatives decay rapidly at $\infty$.
\end{remark}

We can now state our main MDP result of this section.
We begin by introducing the associated rate function.

Let $\Mmc_T(\Rmb^d \times [0,T])$ be the space of all measures $\nu$ on $(\Rmb^d \times [0,T], \Bmc(\Rmb^d \times [0,T]))$ such that $\nu(\Rmb^d \times [0,t]) = t$ for all $t \in [0,T]$, equipped with the usual weak convergence topology.
%Let
%\begin{equation*}
%	\Mmc_T(\Rmb^2 \times [0,T]) = \{ \nu \in \Mmc(\Rmb^2 \times [0,T]) : \nu(\Rmb^2 \times [0,T]) = T\}.
%\end{equation*}
With $\mu_s \equiv \mu(s)$ as in \eqref{eqdif:Xbari}, define $\nubar \in \Mmc_T(\Rmb \times [0,T])$ as
\begin{equation} \label{eqdif:nubar}
	\nubar(A \times [0,t]) \doteq \int_0^t \mu_s(A) \, ds, \quad A \in \Bmc(\Rmb), \quad t \in [0,T].
\end{equation}
For a measure $\theta \in \Mmc_T(\Rmb^d \times [0,T])$, denote by $\theta_{(i)}$ [resp.\ $\theta_{(i,j)}$] the $i$-th [resp. $(i,j)$-th joint] marginal.
Let
\begin{equation} \label{eqdif:Pinfty}
	\Pmc_{\infty} \doteq \left\{ \nu \in \Mmc_T(\Rmb^2 \times [0,T]) \, \bigg| \, \nu_{(2,3)} = \nubar, \int_{\Rmb^2 \times [0,T]} y^2 \, \nu(dy\, dx\, ds) < \infty \right\}.
\end{equation}
The space $\Pmc_\infty$ will be used to formulate the rate function and will play a key role in our weak convergence analysis.
Roughly speaking, for a $\nu \in \Pmc_\infty$, the first marginal $\nu_{(1)}$ corresponds to the ``control variable", $\nu_{(2)}$ to the ``state variable" and $\nu_{(3)}$ to the ``time variable" (see \eqref{eqdif:char}).

Given $\eta \in \Cmb([0,T]: \Smc_{-v})$, let $\Tmc(\eta)$ be the collection of all $\nu \in \Pmc_{\infty}$
such that, for all $\phi \in \Smc$ and $t \in [0,T]$
\begin{equation}
	\label{eqdif:char}
	\lan \eta(t), \phi \ran = \int_0^t \lan \eta(s), L(s)\phi \ran \, ds + \int_{\Rmb^2 \times [0,t]} \phi'(x) \sigma(x,\mu(s)) y \, \nu(dy\, dx\, ds).
\end{equation}
Note that since $L(s)$ maps $\Smc_w$ to $\Smc_v$ (see Lemma~\ref{lemdif:Ls}) and $\Smc \subset \Smc_n$ for all $n \in \Nmb$, $\lan \eta(s),L(s)\phi \ran$ is well defined for all $s \in [0,T]$.
In Section~\ref{secdif:laplowerbd} (see Lemma~\ref{lemdif:unique}), we will show that under Conditions~\ref{conddif:diffusion1} and~\ref{conddif:diffusion2}, for every $\nu \in \Pmc_\infty$, there exists a unique $\eta \in \Cmb([0,T]: \Smc_{-v})$ that solves \eqref{eqdif:char}.
Define $I \colon \Cmb([0,T]: \Smc_{-v}) \to [0,\infty]$ as
\begin{equation} \label{eqdif:rate function}
	I(\eta) \doteq \inf_{\nu \in \Tmc(\eta)} \left\{ \half \int_{\Rmb^2 \times [0,T]} y^2 \, \nu(dy\, dx\, ds)\right\},
\end{equation}
where the infimum over an empty set is taken to be $\infty$.
In Sections~\ref{secdif:Lapupperbd} and~\ref{secdif:laplowerbd} we will see that under Conditions~\ref{conddif:diffusion1} and~\ref{conddif:diffusion2} for every $\tau \ge v$ Laplace upper and lower bounds (see \eqref{eqdif:Laplaceupperbound} and \eqref{eqdif:Laplacelowerbound}) hold for every $F \in \Cmb_b(\Cmb([0,T]: \Smc_{-\tau}))$ with $I$ defined as above.
Although the Laplace upper and lower bound hold in particular with $\tau=v$, the function $I$ need not have relatively compact level sets in $\Cmb([0,T]: \Smc_{-v})$ 
(see comments in Section \ref{secdif:rate function} below \eqref{eqdif:rate bound}) and one needs to strengthen Condition~\ref{conddif:diffusion2} and enlarge the space in order to obtain the compactness property of $I$.
Specifically, we take $\rho > w$ such that $\sum_{j=1}^\infty \|\phi_j^\rho\|_w^2 < \infty$ and we strengthen Condition~\ref{conddif:diffusion2} as follows.

\begin{condition} \label{conddif:diffusion3}
	$\alpha$, $\beta$ are $(\rho+2)$-times continuously differentiable and,
	
	(a) $\sup_y |\alpha(\cdot,y)|_{\rho+2} < \infty$ and $\sup_y |\beta(\cdot,y)|_{\rho+2} < \infty$.
	
	(b) $\sup_x \|\alpha(x,\cdot)\|_{\rho+2} < \infty$ and $\sup_x \|\beta(x,\cdot)\|_{\rho+2} < \infty$.
\end{condition}

Under Conditions~\ref{conddif:diffusion1} and~\ref{conddif:diffusion3} we will establish an LDP for $Y^m$ in $\Cmb([0,T]: \Smc_{-\rho})$ with rate function $I$.
%This $\rho$ identifies the space $\Smc_{-\rho}$ in which the LDP will be established.
We thus regard $I$ as a function from $\Cmb([0,T]: \Smc_{-\rho})$ to $[0,\infty]$, with the convention that
%\begin{equation} \label{eqdif:rate infinite}
	$I(\eta) \doteq \infty$ for all $\eta \in \Cmb([0,T]: \Smc_{-\rho}) \setminus \Cmb([0,T]: \Smc_{-v})$.
%\end{equation}

The following is the main result of this section.
The proof will be given in Section~\ref{secdif:pf}.
\begin{theorem}
	\label{thmdif:main}
	Under Conditions~\ref{conddif:diffusion1} and~\ref{conddif:diffusion3}, $\{Y^m\}$ satisfies an LDP in $\Cmb([0,T]:\Smc_{-\rho})$ with speed $a^2(m)$ and rate function $I$, where $\rho \in \Nmb$ is as introduced above.
\end{theorem}

\textbf{Outline of the proof:} The proof of Theorem~\ref{thmdif:main} will be completed in three steps.
\begin{itemize}
	\item 
		Laplace principle upper bound: In Section~\ref{secdif:Lapupperbd} we show that under Conditions~\ref{conddif:diffusion1} and~\ref{conddif:diffusion2}, for all $\tau \ge v$ and $F \in \Cmb_b(\Cmb([0,T]: \Smc_{-\tau}))$,
		\begin{equation} \label{eqdif:Laplaceupperbound}
			\liminf_{m\to \infty}-a^2(m) \log \Emb \exp \left\{-\frac{1}{a^2(m)} F(Y^m)\right\} \ge \inf_{\zeta \in \Cmb([0,T]: \Smc_{-v}) }\{I(\zeta) + F(\zeta)\}.
		\end{equation}
	\item
		Laplace principle lower bound: In Section~\ref{secdif:laplowerbd} we show that under Conditions~\ref{conddif:diffusion1} and~\ref{conddif:diffusion2}, for all $\tau \ge v$ and $F \in \Cmb_b(\Cmb([0,T]: \Smc_{-\tau}))$,
		\begin{equation} \label{eqdif:Laplacelowerbound}
			\limsup_{m\to \infty}-a^2(m) \log \Emb \exp \left\{-\frac{1}{a^2(m)} F(Y^m)\right\} \le \inf_{\zeta \in \Cmb([0,T]: \Smc_{-v}) }\{I(\zeta) + F(\zeta)\}.
		\end{equation}
	\item
		$I$ is a rate function on $\Cmb([0,T]: \Smc_{-\rho})$: In Section~\ref{secdif:rate function} we show that under Conditions~\ref{conddif:diffusion1} and~\ref{conddif:diffusion3}, for each $K < \infty$, $\{\eta \in \Cmb([0,T]: \Smc_{-\rho}):I(\eta) \le K\}$ is a compact subset of $\Cmb([0,T]: \Smc_{-\rho})$.
\end{itemize} 
Note that since $I(\eta)=\infty$ for $\eta \notin \Cmb([0,T]:\Smc_{-v})$, we can replace $v$ by any $\tau \ge v$ on the right sides of \eqref{eqdif:Laplaceupperbound} and \eqref{eqdif:Laplacelowerbound}.
%The Laplace principle upper and lower bounds will be proved in Sections~\ref{secdif:Lapupperbd} and~\ref{secdif:laplowerbd}, respectively.
%We remark that for the proof of the Laplace upper and lower bounds one can replace Condition~\ref{conddif:diffusion3} with the weaker Condition~\ref{conddif:diffusion2}.
%In Section~\ref{secdif:rate function}, we will show that $I$ is a rate function.
Theorem~\ref{thmdif:main} follows on combining these results.

\begin{remark}
	The rate function $I$ has the following alternative representation.
Given $\eta \in \Cmb([0,T]:\Smc_{-v})$, let $\Tmc^*(\eta)$ be the collection of  $g \in L^2(\nubar)$ such that for all $\phi \in \Smc$ and $t \in [0,T]$,
	\begin{equation} \label{eqdif:charalter}
		\lan \eta(t), \phi \ran = \int_0^t \lan \eta(s), L(s)\phi \ran \, ds + \int_{\Rmb^2 \times [0,t]} \phi'(x) \sigma(x,\mu(s)) g(x,s) \, \nu(dy\, dx\, ds).
	\end{equation}
	As for \eqref{eqdif:char}, under Conditions~\ref{conddif:diffusion1} and~\ref{conddif:diffusion2}, for every $g \in L^2(\nubar)$, there is a unique $\eta \in \Cmb([0,T]:\Smc_{-v})$ that solves \eqref{eqdif:charalter}.
	We take $\Tmc^*(\eta)$ to be the empty set if $\eta \in \Cmb([0,T]:\Smc_{-\rho}) \setminus \Cmb([0,T]:\Smc_{-v})$.
	Define $I^* \colon \Cmb([0,T]:\Smc_{-\rho}) \to [0,\infty]$ as
	\begin{equation*}
		I^*(\eta) \doteq \inf_{g \in \Tmc^*(\eta)} \left\{ \half \int_{\Rmb \times [0,T]} g^2(x,s) \, \mu_s(dx) \, ds \right\}.
	\end{equation*}
	It is easy to check that $I^* = I$.
	Indeed every $g \in \Tmc^*(\eta)$ corresponds to a $\nu_g \in \Pmc_\infty$ given as $\nu_g(dy\,dx\,ds) \doteq \delta_{g(x,s)}(dy) \, \nubar(dx\,ds)$ and every $\nu \in \Pmc_\infty$ corresponds to a $g_\nu \in L^2(\nubar)$ given as $g_\nu(x,s) \doteq \int_\Rmb y \, \vartheta(x,s,dy)$, where $\vartheta(x,s,dy)$ is obtained by disintegrating $\nu$ as $\nu(dy\,dx\,ds) = \vartheta(x,s,dy)\,\nubar(dx\,ds)$. 
\end{remark}

%----------------------------------------------------

\section{The pure jump case} \label{secjump:jump}

%\subsection{Preliminaries} \label{secjump:preliminaries}

In this section we study the weakly interacting pure jump Markov processes $\{(X_1^m,\dotsc,X_m^m)\}$ taking values in $\Nmb^m$ that were introduced in Section~\ref{sec:intro}.
We begin in Section~\ref{secjump:model} with a precise model description and a law of large numbers result.
Proof of the LLN result follows from standard arguments, however for completeness we provide a sketch in the Appendix.
We then present our main result (Theorem~\ref{thmjump:MDP}) on moderate deviations for the associated empirical measure processes in Section~\ref{secjump:MDP}.
Proof of Theorem~\ref{thmjump:MDP} is given in Section~\ref{secjump:pf}.

%----------------------------------------------------

\subsection{Model and law of large numbers} \label{secjump:model}

Recall the pure jump Markov process $\{(X_1^m(t),\dotsc,X_m^m(t)), t \in [0,T]\}$ governed by intensity function $\Gamma$ that was introduced in Section~\ref{sec:intro} (see above \eqref{eqjump:jump intensity}).
It will be convenient to describe the evolution of the associated empirical measure process $\{\mu^m(t)\}$ through an SDE driven by a Poisson random measure.
We now introduce some notation that will be needed to formulate this evolution equation.

%Let $l_1$ and $l_2$ denote the space of sequences that are absolutely summable and square-summable, respectively, equipped with the usual norm $\| \cdot \|_1$ and $\| \cdot \|_2$.
%Especially, we let $\| \cdot \| \equiv \| \cdot \|_2$ since it is most commonly used in the jump cases.
For a locally compact Polish space $\Smb$, let $\Mmc_{FC}(\Smb)$ be the space of all measures $\nu$ on $(\Smb,\Bmc(\Smb))$ such that $\nu(K) < \infty$ for every compact $K \subset \Smb$.
We equip $\Mmc_{FC}(\Smb)$ with the usual vague topology.
This topology can be metrized such that $\Mmc_{FC}(\Smb)$ is a Polish space (see for example~\cite{BudhirajaDupuisMaroulas2011variational}).
A Poisson random measure (PRM) $\nbd$ on $\Smb$ with mean measure (or intensity measure) $\nu$ is a $\Mmc_{FC}(\Smb)$-valued random variable such that for each $B \in \Bmc(\Smb)$ with $\nu(B) < \infty$, $\nbd(B)$ is Poisson distributed with mean $\nu(B)$ and for disjoint $B_1,\dotsc,B_k \in \Bmc(\Smb)$, $\nbd(B_1),\dotsc,\nbd(B_k)$ are mutually independent random variables (cf.~\cite{IkedaWatanabe1990SDE}).

Let $l_2$ be the Hilbert space of square-summable sequences, equipped with the usual inner product and norm denoted by $\lan \cdot,\cdot \ran$ and $\| \cdot \|$, respectively.
For each $i \in \Nmb$ let $\ebd_i$ be the unit vector in $l_2$ with $1$ for the $i$-th coordinate and $0$ otherwise.

We are interested in characterizing the limit of the empirical measure process $\{ \mumt \}$ in the space $\Dmb([0,T]:l_2)$, and to establish a moderate deviation principle for $\{ \mumt \}$.
We begin by giving an equivalent in law representation of this empirical measure process using a PRM on a suitable point space.
We will follow the notation in~\cite{BudhirajaDupuisMaroulas2011variational}.

Let $\Xmb \doteq \Rmb^2_+$, $\Ymb \doteq \Xmb \times \Rmb_+ = \Rmb_+^3$, $\XT \doteq [0,T] \times \Xmb$ and $\YT \doteq [0,T] \times \Ymb$.
Let $\lambda_T$, $\lambda_\Xmb$ and $\lambda_\infty$ be the Lebesgue measures on $[0,T]$, $\Xmb$ and $\Rmb_+$, respectively. 
Let $N$ be a PRM on $\YT$ with intensity $\lambda_\YT \doteq \lambda_T \otimes \lambda_\Xmb \otimes \lambda_\infty$, defined on some filtered probability space $(\Omega, \Fmc, \Pmb, \{\Fmc_t\})$ with a $\Pmb$-complete right-continuous filtration.
We assume that $N([0,a] \times A)$ is $\Fmc_a$-measurable and $N((a,b] \times A)$ is independent of $\Fmc_a$ for all $0 \le a < b \le T$ and $A \in \Bmc(\Ymb)$.
Given $m \ge 1$, let $N^m$ be a counting process on $\XT$ defined as
\begin{equation*}
	N^m([0,t] \times A) \doteq \int_{[0,t] \times A} \one_{[0,m]}(r) \, N(ds \, dy \, dr), \quad t \in [0,T], A \in \Bmc(\Xmb).
\end{equation*}
We will make the following assumption on $\Gamma$ (this assumption will be restated in Condition~\ref{condjump:cond1})
\begin{equation} \label{eqjump:Gamma norm finite}
	\Gammainf \doteq \sup_{q \in \Smchat} \sup_{i \in \Nmb} | \Gamma_{ii}(q)| < \infty.
\end{equation}
Given $i$, $j \in \Nmb$ with $i \ne j$ and $q \in \Smchat$, let
\begin{equation*}
	A_{ij}(q) \doteq \{ y \in \Xmb : i-1 < y_1 \le i, \: (j-1) \Gammainf < y_2 \le (j-1) \Gammainf + q_i \Gamma_{ij}(q) \}.
\end{equation*} 
Note that for every $q \in \Smchat$,
\begin{equation} \label{eqjump:Aijdisjoint}
	A_{ij}(q) \cap A_{i'j'}(q) = \emptyset, \text{ if } (i,j) \ne (i',j').
\end{equation}
For $q \in \Smchat$ and $y \in \Xmb$, let
\begin{equation} \label{eqjump:G}
	G(q,y) \doteq \sum_{i=1}^\infty \sum_{j=1, j \ne i}^\infty (\ebd_j - \ebd_i) \one_{A_{ij}(q)}(y) = \sum_{i=1}^\infty G_i(q,y) \ebd_i, 
\end{equation}
%	G(q,y) & = \sum_{i=1}^\infty \sum_{j \ne i}^\infty (\ebd_j - \ebd_i) \one_{A_{ij}(q)}(y) = \sum_{i=1}^\infty \Big( \sum_{j \ne i}^\infty \big( \one_{A_{ji}(q)}(y) - \one_{A_{ij}(q)}(y) \big) \Big) \ebd_i = \sum_{i=1}^\infty G_i(q,y) \ebd_i, \label{eqjump:G} \\
where
\begin{equation} \label{eqjump:Gi}
	G_i(q,y) \doteq \sum_{j=1, j \ne i}^\infty \big( \one_{A_{ji}(q)}(y) - \one_{A_{ij}(q)}(y) \big), \quad i \in \Nmb. 
\end{equation}
Note that $\| G(q,y) \| \le \sqrt{2}$ for all $q \in \Smchat$ and $y \in \Xmb$, in particular $G$ is a well defined map from $l_2 \times \Xmb$ to $l_2$.
Define the stochastic process $\{ \mumt, t \in [0,T] \}$ as
\begin{equation} \label{eqjump:mu m t N}
	\mumt = \mu^m(0) + \frac{1}{m} \intXt G(\mu^m(s-),y) \, N^m(ds \, dy),
\end{equation}
where $\mu^m(0) \doteq \frac{1}{m} \sum_{i=1}^m \delta_{x_i^m}$ and $\Xt \doteq [0,t] \times \Xmb$ for each $t \in [0,T]$.
This  describes a pure jump Markov process for which jump at time instant $t$ is $\frac{1}{m} (\ebd_j - \ebd_i)$ at rate $m \lambda_\Xmb(A_{ij}(q)) = m q_i \Gamma_{ij}(q)$ with $q = \mu^m(t-)$ for $i,j \in \Nmb$ with $i \ne j$.
Thus $\{ \mumt \}$ defined by \eqref{eqjump:mu m t N} has the same law as the empirical process $\{ \frac{1}{m} \sum_{i=1}^m \delta_{X^m_i(t)} \}$ introduced in Section~\ref{sec:intro} with jump intensities \eqref{eqjump:jump intensity}.
Throughout this work we will use the representation for $\{\mu^m(t)\}$ given by \eqref{eqjump:mu m t N}.
With this representation $\{\mumt\}$ can be viewed as an Hilbert space ($l_2$)-valued small noise stochastic dynamical system driven by a PRM.
Moderate deviation principles for such small noise processes have been studied in~\cite{BudhirajaDupuisGanguly2015moderate}.
However one key difference between the models in~\cite{BudhirajaDupuisGanguly2015moderate} from that considered here is that unlike in~\cite{BudhirajaDupuisGanguly2015moderate} the map $x \mapsto G(x,y)$ is not Lipschitz (in fact not even continuous).
This lack of regularity is one of the challenges in the large deviation analysis.

Let $\Ntil^m(ds \, dy) \doteq N^m(ds \, dy) - m \lambda(ds \, dy)$, where $\lambda \doteq \lambda_T \otimes \lambda_\Xmb$ is the Lebesgue measure on $\XT$.
Then \eqref{eqjump:mu m t N} can be written as:
\begin{equation} \label{eqjump:mu m t N tilde}
	\mumt = \mu^m(0) + \int_0^t b(\mums) \, ds + \frac{1}{m} \intXt G(\mu^m(s-),y) \, \Ntil^m(ds \, dy),
\end{equation}
where $b \colon \Smchat \to l_2$ is defined as 
\begin{equation} \label{eqjump:b}
	b(q) \doteq \sum_{i=1}^\infty \big( \sum_{j=1}^\infty q_j \Gamma_{ji}(q) \big) \ebd_i, \quad q = (q_1, q_2, \dotsc) \in \Smchat.
\end{equation}
To see that $b(q)$ defined by \eqref{eqjump:b} is in $l_2$, note that for $q = (q_1, q_2, \dotsc) \in \Smchat$,
\begin{gather*}
	\Big| \sum_{j=1}^\infty q_j \Gamma_{ji}(q) \Big| \le \Gammainf < \infty, \\
	\| b(q) \|^2 = \sum_{i=1}^\infty \Big| \sum_{j=1}^\infty q_j \Gamma_{ji}(q) \Big|^2 \le \sum_{i=1}^\infty \sum_{j=1}^\infty q_j \Gamma_{ji}^2(q) \le 2 \|\Gamma\|_\infty^2.
\end{gather*}
Note also that $b(q) = \int_\Xmb G(q,y) \, \lambda_\Xmb(dy)$.

We now introduce an assumption under which a law of large numbers result holds.
Note that part (a) was previously stated in \eqref{eqjump:Gamma norm finite}.

\begin{condition} \label{condjump:cond1}
	(a) $\Gammainf < \infty$.
	
	(b) The map $b \colon \Smchat \to l_2$ defined in \eqref{eqjump:b} is Lipschitz, namely there exists $L_b \in (0,\infty)$ such that for all $q$, $\qtil \in \Smchat$,
	\begin{equation*}
		\| b(q) - b(\qtil) \| \le L_b \| q - \qtil \|.
	\end{equation*}

	(c) $\| \mu^m(0)-p(0) \| \to 0$ as $m \to \infty$ for some probability measure $p(0) \in \Pmc(\Nmb)$.
\end{condition}

\begin{remark} \label{rmkjump:rmk for cond1}
	Condition~\ref{condjump:cond1}(c) is trivially satisfied if $p(0) = \delta_x$ and $x^m_i = x$ for all $m,i \in \Nmb$, for some $x \in \Nmb$.
	An elementary application of Scheff\'e's lemma and the strong law of large numbers shows that it is also satisfied for a.e. $\omega$ if $x^m_i = \xi_i(\omega)$ where $\xi_i$ are i.i.d.\ with common distribution $p(0)$.
\end{remark}

Let $\Mmb \doteq \Mmc_{FC}(\XT)$, namely the space of all measures $\nu$ on $(\XT,\Bmc(\XT))$ such that $\nu(K) < \infty$ for every compact $K \subset \XT$.
%and $\Pmb$ be the unique probability measure on $(\Mmb,\Bmc(\Mmb))$ under which the canonical map, $N_* : \Mmb \to \Mmb, N_*(\mtil) := \mtil$, is a Poisson random measure with intensity measure $\lambda_\XT$.
%The corresponding expectation operator will be denoted by $\Emb$.
The proof of the following result giving unique solvability and law of large numbers is standard.
We provide a sketch in Appendix~\ref{Appendix:proof of theorem of LLN} for completeness.

\begin{theorem} \label{thmjump:LLN}
	Under Condition~\ref{condjump:cond1}, the following conclusions hold.
	
	(a) For each $m \in \Nmb$ there is a measurable map $\Gmcbar^m \colon \Mmb \to \Dmb([0,T]:l_2)$ such that for any probability space $(\Omegatil,\Fmctil,\Pmbtil)$ on which is given a Poisson random measure $\nbd_m$ on $\XT$ with intensity measure $m \lambda$, $\mutil^m \doteq \Gmcbar^m\left(\frac{1}{m} \nbd_m\right)$ is an $\Fmctil_t \doteq \sigma \{ \nbd_m([0,s] \times A), s \le t, A \in \Bmc(\Xmb) \}$-adapted RCLL process that is the unique adapted solution of the stochastic integral equation
	\begin{equation*}
		\mutil^m(t) = \mutil^m(0) + \frac{1}{m} \intXt G(\mutil^m(s-),y) \, \nbd_m(ds \, dy), \quad t \in [0,T].
	\end{equation*}
	In particular $\mu^m \doteq \Gmcbar^m(\frac{1}{m} N^m)$ is the unique $\{\Fmc_t\}$-adapted solution of \eqref{eqjump:mu m t N}.
	
	(b) The process $\mu^m$ converges in probability to $p$ in $\Dmb([0,T]:l_2)$, where $p$ is given as the unique solution of the following integral equation in $l_2$:
	\begin{equation} \label{eqjump:p t}
		p(t) = p(0) + \int_0^t b(p(s)) \, ds, \quad t \in [0,T].
	\end{equation}
\end{theorem}

%----------------------------------------------------

\subsection{Moderate deviation principle} \label{secjump:MDP}

Let $a(m)$ be as in \eqref{eqjump:a m}.
We are interested in the asymptotics of the probabilities of deviations of $\mu^m$ from $p$ that are of order $\frac{1}{a(m)\sqrt{m}}$.
For this we will establish a large deviation principle for 
\begin{equation} \label{eqjump:Z m}
	Z^m \doteq a(m)\sqrt{m}(\mu^m - p).
\end{equation}
We make the following stronger assumption in place of Condition~\ref{condjump:cond1}.
\begin{condition} \label{condjump:cond2}
%	(a) The operator $\Gammabar : \Smchat \to L(l_1,l_1)$ defined in \eqref{eqjump:Gamma} is Lipschitz, namely there exists $L_\Gammabar \in (0,\infty)$ such that for all $q, \qtil \in \Smchat$,
%	\begin{equation*}
%		\| \Gammabar(\qtil) - \Gammabar(q) \|_{op} \le L_\Gammabar \|\qtil-q\|.
%	\end{equation*}
	(a) $\Gammainf < \infty$.
	
	(b) There exist $c_\Gamma$, $L_\Gamma \in (0,\infty)$ such that
	\begin{equation*}
		\sup_{q \in \Smchat} \sup_{j \in \Nmb} \sum_{i=1}^\infty |\Gamma_{ij}(q)| \le c_\Gamma
	\end{equation*} 
	and for all $\qtil$, $q \in \Smchat$,
	\begin{equation*}
		\sup_{i \in \Nmb} \sum_{j = 1, j \ne i}^\infty |\Gamma_{ij}(\qtil) - \Gamma_{ij}(q)| \le L_\Gamma \|\qtil - q\|.
	\end{equation*}

	(c) For the map $b \colon \Smchat \to l_2$ defined in \eqref{eqjump:b}, there exist $c_b \in (0,\infty)$; a map $Db \colon \Smchat \to L(l_2,l_2)$; and $\theta_b \colon \Smchat \times \Smchat \to l_2$ such that for all $q$, $\qtil \in \Smchat$,	
	\begin{equation*}
		b(\qtil) - b(q) = Db(q)[\qtil-q] + \theta_b(q,\qtil)
	\end{equation*}
	and
	\begin{equation*}
		\| \theta_b(q,\qtil) \| \le c_b \|\qtil-q\|^2, \quad \| Db(q) \|_{L(l_2,l_2)} \le c_b.
	\end{equation*}

	(d) $a(m) \sqrt{m} \| \mu^m(0)-p(0) \| \to 0$ as $m \to \infty$ for some probability measure $p(0) \in \Pmc(\Nmb)$.
\end{condition}

\begin{remark} \label{rmkjump:rmk for cond2}
	(i) Condition~\ref{condjump:cond1}(b) is implied by Conditions~\ref{condjump:cond2}(c) while Condition~\ref{condjump:cond1}(c) is implied by Conditions~\ref{condjump:cond2}(d).
	
	(ii) Condition~\ref{condjump:cond2}(b) is satisfied in particular for finite range jump models of the following form: There exists some $K \in (0,\infty)$ such that for all $q \in \Smchat$, $\Gamma_{ij}(q) = 0$ for $|i-j|>K$ and $q \mapsto \Gamma_{ij}(q)$ is Lipschitz continuous with $\|\Gamma_{ij}\|_L \le K$ for $|i-j| \le K$.
	
	(iii) Condition~\ref{condjump:cond2}(c) is an assumption on smoothness of $b$ which says that $b$ is differentiable and the derivative is bounded.
	
	(iv) Condition~\ref{condjump:cond2}(d) is trivially satisfied if $p(0) = \delta_x$ and $x^m_i = x$ for all $m,i \in \Nmb$, for some $x \in \Nmb$.
	It is also satisfied for a.e. $\omega$ if $x^m_i = \xi_i(\omega)$ where $\xi_i$ are i.i.d.\ with common distribution $p(0)$ and $\sum_{m=1}^\infty [a(m)]^{2n} < \infty$ for some $n \in \Nmb$ (for a proof see Appendix~\ref{Appendix:proof of rmk}).
	This summability property is satisfied quite generally, e.g. when $a(m) = O(m^{-\theta})$ for some $\theta \in (0, 1/2)$ or $a(m) = O((\log m)^k m^{-\theta})$ for some $\theta \in (0,1/2]$ and $k > 0$.
\end{remark}

The following theorem is the main result of this section.
The proof will be given in Section~\ref{secjump:pfMDP}.

\begin{theorem} \label{thmjump:MDP}
	Under Condition~\ref{condjump:cond2}, $\{ Z^m \}$ satisfies a large deviation principle in $\Dmb([0,T]:l_2)$ with speed $a^2(m)$ and the rate function given by
	\begin{equation} \label{eqjump:rate function}
		\Ibar(\eta) \doteq \inf_\psi \left\{ \frac{1}{2} \| \psi \|_{L^2(\lambda)}^2 \right\}, \quad \eta \in \Dmb([0,T]:l_2),
	\end{equation}
	where the infimum is taken over all $\psi \in L^2(\lambda)$ such that
	\begin{equation} \label{eqjump:eta and psi}
		\eta(t) = \int_0^t Db(p(s))[\eta(s)] \, ds + \intXt G(p(s),y) \psi(s,y) \, \lambda(ds \, dy), \quad t \in [0,T].
	\end{equation}
%	and $\| \cdot \|_\XT$ denotes the norm on $L^2(\lambda)$.
\end{theorem}

Along the lines of the proof of Theorem~\ref{thmjump:LLN}(b), it is easy to check that under Condition~\ref{condjump:cond2}(c), \eqref{eqjump:eta and psi} has a unique solution in $\Cmb([0,T]:l_2)$ for each $\psi \in L^2(\lambda)$.
In particular, $\Ibar(\eta) \doteq \infty$ for all $\eta \in \Dmb([0,T]:l_2) \backslash \Cmb([0,T]:l_2)$.

The rate function $\Ibar$ introduced in \eqref{eqjump:rate function} is somewhat indirect in that its definition involves the extraneous function $G$ that was introduced for the convenience of representation of $\mu^m$ as a small noise stochastic dynamical system.
The following result gives an alternative representation that is more intrinsic.
For $\eta \in \Dmb([0,T]:l_2)$ let
\begin{equation} \label{eqjump:rate function alter}
	I(\eta) \doteq \inf_{u} \left\{ \half \int_0^T \sum_{i=1}^\infty \sum_{j=1,j \ne i}^\infty u_{ij}^2(s) \, ds \right\},
\end{equation}
where the infimum is taken over all $u \doteq \{u_{ij}\}_{i,j=1}^\infty$ with each $u_{ij} \in L^2([0,T]:\Rmb)$ such that
\begin{equation} \label{eqjump:eta and u}
	\eta(t) = \int_0^t Db(p(s))[\eta(s)] \, ds + \int_0^t \sum_{i=1}^\infty \sum_{j=1,j \ne i}^\infty (\ebd_j-\ebd_i) \sqrt{p_i(s) \Gamma_{ij}(p(s))} u_{ij}(s) \, ds.
\end{equation}
The proof of the following result will be given in Section~\ref{secjump:pfalter}.

\begin{theorem} \label{thmjump:alternative expression}
	Under the conditions of Theorem~\ref{thmjump:MDP}, $I = \Ibar$.
\end{theorem}

%----------------------------------------------------

\section{Proofs for the diffusion case} \label{secdif:pf}

The main ingredient in the proof of Theorem~\ref{thmdif:main} is the following variational representation.
Such a representation was first obtained in~\cite{BoueDupuis1998variational}.
For our purpose it is convenient to use the form of the representation given in~\cite{BudhirajaDupuisMaroulas2008large} that allows for an arbitrary filtration.
Let $(\Omegatil,\Fmctil,\Pmbtil)$ be a probability space with an increasing family of right continuous $\Pmbtil$-complete $\sigma$-fields $\{\Fmctil_t\}$.
Let for $m \in \Nmb$, $\{B(t) \doteq (B_1(t),\dotsc,B_m(t)), 0 \le t \le T \}$ be an $m$-dimensional standard $\{\Fmctil_t\}$-Brownian motions on this probability space.
Let $\ti\Amc$ be the collection of $\Rmb^m$-valued $\{\Fmctil_t\}$-progressively measurable processes $\{u(t)=(u_1(t),\dotsc,u_m(t)), 0 \le t \le T\}$ such that $\Pmbtil \left(\int_0^T \|u(s)\|^2 \, ds < \infty\right)=1$.
Let $f \in \Mmb_b(\Cmb([0,T]:\Rmb^m))$.
The following representation follows from~\cite{BoueDupuis1998variational,BudhirajaDupuisMaroulas2008large}
\begin{equation} \label{eqdif:variationalrep}
	-\log \Embtil \exp \{ f(B) \} = \inf_{u \in \ti\Amc} \Embtil \left( \half \int_0^T \|u(s)\|^2 \, ds + f\left(B + \int_0^\cdot u(s) \, ds \right) \right).
\end{equation}

The proof of the MDP in Theorem~\ref{thmdif:main} will proceed by first establishing the Laplace upper bound \eqref{eqdif:lapupp} and then the Laplace lower bound \eqref{eqdif:laplow}.
The above representation will be a key ingredient in both proofs.
Rest of this section is organized as follows.
We first analyze certain controlled  process in Sections~\ref{secdif:varrep} and~\ref{secdif:tightness}.
Proof of the Laplace upper bound is given in Section~\ref{secdif:Lapupperbd} whereas the lower bound is established in Section~\ref{secdif:laplowerbd}.
In order to argue that $\{Y^m\}$ satisfies an LDP it then remains to establish that $I$ defined in \eqref{eqdif:rate function} is a rate function.
This is proved in Section~\ref{secdif:rate function}.

\subsection{Controlled processes} \label{secdif:varrep}

Throughout this section we assume Condition~\ref{conddif:diffusion1}.
Let $v,w,\rho$ be as introduced in Section~\ref{secdif:diffusion}.
Fix $\tau \ge v$ and let $F \in \Cmb_b(\Cmb([0,T]:\Smc_{-\tau}))$.
Using the variational representation in \eqref{eqdif:variationalrep} on the filtered probability space introduced below \eqref{eqdif:Xim}, we have
\begin{equation*}
%	\label{eqdif:repn1}
	-a^2(m) \log \Emb \exp \left\{-\frac{1}{a^2(m)} F(Y^m)\right\}  = \inf_{u^m = \{u_i^m\}_{i=1}^m} \Emb\left\{ \half a^2(m)\sum_{i=1}^m\int_0^T |u_i^m(s)|^2 \, ds	+ F(\Ytil^m)\right\},
\end{equation*}
where the infimum is taken over all $\{\Fmc_t\}$-progressively measurable $u^m$ such that
\begin{equation*}
	\Emb \sum_{i=1}^m\int_0^T |u_i^m(s)|^2 \, ds < \infty,
\end{equation*}
and $\Ytil^m$ is as in \eqref{eqdif:Ytilm} with $\Xtil^m$ given by \eqref{eqdif:Xtilim}.
We will view $\{u_i^m\}_{i=1}^m$ as a sequence of controls and $\{\Xtil_i^m\}_{i=1}^m$ as the controlled analogue of the original interacting particle system \eqref{eqdif:Xim}.
Letting $\util_i^m \doteq a(m)\sqrt{m} u_i^m$, one can write
\begin{equation}
	\label{eqdif:repn1b}
	-a^2(m) \log \Emb \exp \left\{-\frac{1}{a^2(m)} F(Y^m)\right\} = \inf_{\util^m = \{\util_i^m\}_{i=1}^m} \Emb\left\{ \half \frac{1}{m}\sum_{i=1}^m\int_0^T |\util_i^m(s)|^2 ds	+ F(\Ytil^m)\right\}
\end{equation}
and
\begin{align} \label{eqdif:Xtilim util}
	\begin{aligned}
		\Xtilimt &= x_0 + \int_0^t \sigma(\Xtilims, \mutilms) \, dW_i(s) + \int_0^t b(\Xtilims, \mutilms) \, ds\\
		&\quad + \frac{1}{a(m)\sqrt{m}} \int_0^t \sigma(\Xtilims, \mutilms) \util_i^m(s) \, ds, \quad i = 1, \dotsc, m. 
	\end{aligned}
\end{align}

The following lemma gives an important moment bound that will be used to argue tightness and convergence of controlled processes.
Recall that the processes $\{X_i^m\}$, $\{\Xtil_i^m\}$ and $\{\Xbar_i\}$ are defined on the same probability space with the same sequence of Brownian motions $\{W_i\}$.
We let for $m \in \Nmb$, $\mubar_t^m \doteq \frac{1}{m} \sum_{i=1}^m \delta_{\Xbar_i(t)}$.

\begin{lemma} \label{lemdif:XtilXbar}
	Suppose the control sequence $\{\util_i^m\}_{i=1}^m$ satisfies
	\begin{equation} \label{eqdif:controlbdgeneral}
		\sup_{m \in \Nmb} \Emb \frac{1}{m} \sum_{i=1}^m \int_0^T |\util_i^m(s)|^2 \, ds < \infty.
	\end{equation}
	Then there exists $\gamma_1 \in (0,\infty)$ such that for all $m \in \Nmb$,
	\begin{equation} \label{eqdif:XtilXbar}
		\Emb \frac{1}{m} \sum_{i=1}^m | \Xtil_i^m - \Xbar_i|_{*,T}^2 \le \frac{\gamma_1}{a^2(m)m}.
	\end{equation}
	In particular,
	\begin{equation} \label{eqdif:Xtilbd}
		\sup_{m \in \Nmb} \Emb \frac{1}{m} \sum_{i=1}^m | \Xtil_i^m |_{*,T}^2 < \infty.
	\end{equation}
\end{lemma}

\begin{proof}
	Using the bounded Lipschitz property of the coefficients $\alpha$ and $\beta$,
	\begin{align*}
		& \Emb | \Xtil_i^m - \Xbar_i|_{*,t}^2 \\
		& \quad \le \kappa_1 \int_0^t \Emb \Big( |\sigma(\Xtilims, \mutilms) - \sigma(\Xbaris, \mubarms)|^2 + |\sigma(\Xbaris, \mubarms) - \sigma(\Xbaris, \mu(s))|^2 \\
		& \qquad + |b(\Xtilims, \mutilms) - b(\Xbaris, \mubarms)|^2 + |b(\Xbaris, \mubarms) - b(\Xbaris, \mu(s))|^2 \\
		& \qquad + \frac{1}{a^2(m)m} |\sigma(\Xtilims, \mutilms) \util_i^m(s)|^2 \Big) \, ds \\
		& \quad \le \kappa_2 \int_0^t \Emb \left( |\Xtilims - \Xbaris|^2 + \frac{1}{m} \sum_{j=1}^m |\Xtiljms - \Xbarjs|^2 + \frac{1}{a^2(m)m} |\util_i^m(s)|^2 + \frac{1}{m} \right) \, ds,
	\end{align*}
	where the contribution of $\frac{1}{m}$ is obtained from the second and the fourth terms on the right side upon using the independence of $\Xbar_i$ and $\Xbar_j$ for $i \ne j$.
	Taking the average over $i=1,\dotsc,m$ on both sides of above inequality and using \eqref{eqdif:controlbdgeneral}
	\begin{equation*}
		\Emb \frac{1}{m} \sum_{i=1}^m | \Xtil_i^m - \Xbar_i|_{*,t}^2 \le \kappa_3 \int_0^t \Emb \frac{1}{m} \sum_{j=1}^m | \Xtil_i^m - \Xbar_i|_{*,s}^2 \, ds + \frac{\kappa_3}{a^2(m)m} + \frac{\kappa_3}{m}.
	\end{equation*}
	The estimate in \eqref{eqdif:XtilXbar} is now immediate by Gronwall's lemma.
	The estimate in \eqref{eqdif:Xtilbd} is a consequence of \eqref{eqdif:XtilXbar} and the fact that
	\begin{equation} \label{eqdif:Xbarbd}
		\sup_{m \in \Nmb} \Emb \frac{1}{m} \sum_{i=1}^m | \Xbar_i |_{*,T}^2 = \Emb |\Xbar_1|_{*,T}^2 < \infty. \qedhere
	\end{equation}
\end{proof}

Our main result of this section is the following representation for the controlled processes $\Ytil^m$.
Recall the operator $L(s)$ defined in \eqref{eqdif:Ls}.

\begin{proposition} \label{propdif:representation}
	Suppose that the control sequence $\{\util_i^m\}_{i=1}^m$ satisfies \eqref{eqdif:controlbdgeneral}.
	Then, for every $t \in [0,T]$ and $\phi \in \Smc$,
	\begin{equation}
		\label{eqdif:mainprelim}
		\lan \Ytil^m(t), \phi \ran = \int_0^t \lan \Ytil^m(s), L(s)\phi \ran \, ds + \int_{\Rmb^2 \times [0,t]} \phi'(x) \sigma(x,  \mu(s)) y \, \nutil^m(dy\, dx\, ds) + \Rmc^m(t),
	\end{equation}
	where
	\begin{equation} \label{eqdif:R5}
		\Emb |\Rmc^m|_{*,T} \le \gamma(m) \|\phi\|_4
	\end{equation}
	and $\gamma(m) \to 0$ as $m \to \infty$.
\end{proposition}

Rest of this section is devoted to the proof of Proposition~\ref{propdif:representation} and so we will assume throughout the remaining section that \eqref{eqdif:controlbdgeneral} holds.
Note that by an application of Ito's formula, for $\phi \in \Smc$,
\begin{align*}
	\phi(\Xtilimt) &= \phi(x_0) + \int_0^t \phi'(\Xtilims) \sigma(\Xtilims, \mutilms) \, dW_i(s) \\
	& \quad + \int_0^t \phi'(\Xtilims) b(\Xtilims, \mutilms) \, ds\\
	&\quad + \frac{1}{a(m)\sqrt{m}}\int_0^t \phi'(\Xtilims) \sigma(\Xtilims, \mutilms) \util_i^m(s) \, ds \\
	&\quad + \half \int_0^t \phi{''}(\Xtilims) \sigma^2(\Xtilims, \mutilms) \, ds.
\end{align*}
Similarly
\begin{align*}
	\phi(\bar X_i(t)) &= \phi(x_0) + \int_0^t \phi'(\bar X_i(s)) \sigma(\bar X_i(s),\mu(s)) \, dW_i(s) \\
	&\quad + \int_0^t \phi'(\bar X_i(s)) b(\bar X_i(s),  \mu(s)) \, ds + \half \int_0^t \phi{''}(\bar X_i(s)) \sigma^2(\bar X_i(s),  \mu(s)) \, ds,
\end{align*}
and so taking expectations
\begin{equation*}
	\lan \mu(t), \phi \ran = \phi(x_0) + \int_0^t \lan \mu(s), \phi'(\cdot) b(\cdot,  \mu(s))\ran \, ds + \half \int_0^t \lan \mu(s), \phi{''}(\cdot) \sigma^2(\cdot,  \mu(s))\ran \, ds.
\end{equation*}
Combining the above observations
\begin{align}
	\lan \Ytil^m(t), \phi \ran &= a(m)\sqrt{m} \left(\lan \mutilmt, \phi \ran - \lan \mu(t), \phi \ran \right) \notag \\
	&=  \frac{a(m)}{\sqrt{m}}\sum_{i=1}^m\int_0^t \phi'(\Xtilims) \sigma(\Xtilims, \mutilms) \, dW_i(s) \notag \\
	&\quad + \frac{a(m)}{\sqrt{m}}\sum_{i=1}^m \int_0^t \left(\phi'(\Xtilims) b(\Xtilims, \mutilms) - \lan \mu(s), \phi'(\cdot) b(\cdot,  \mu(s))\ran \right) \, ds \notag \\
	&\quad + \frac{1}{m} \sum_{i=1}^m\int_0^t \phi'(\Xtilims) \sigma(\Xtilims, \mutilms) \util_i^m(s) \, ds \notag \\
	&\quad + \half \frac{a(m)}{\sqrt{m}} \sum_{i=1}^m \int_0^t \left(\phi''(\Xtilims) \sigma^2(\Xtilims, \mutilms) - \lan \mu(s), \phi{''}(\cdot) \sigma^2(\cdot,  \mu(s))\ran\right) \, ds \notag \\
	& \equiv \sum_{k=1}^4 \Tmc_k^m(t). \label{eqdif:splitforprop}
\end{align}

We will now separately estimate each $\Tmc_k^m$, $k=1,2,3,4$.
We begin with $\Tmc_1^m$.

\begin{lemma} \label{lemdif:R1}
	There exists $\gamma_2 \in (0,\infty)$ such that
	\begin{equation*}
		\Emb |\Tmc_1^m|_{*,T} \le \gamma_2 a(m) \|\phi\|_2.
	\end{equation*}
\end{lemma}

\begin{proof}
	Using Doob's maximal inequality, the boundedness of $\alpha$ and \eqref{eqdif:twonorms},
	\begin{align*}
		\Emb |\Tmc_1^m|_{*,T}^2 & \le\frac{4a^2(m)}{m} \sum_{i=1}^m \Emb \int_0^T [\phi'(\Xtilims) \sigma(\Xtilims, \mutilms)]^2 \, ds \\
		& \le \kappa_1 a^2(m) |\phi|_1^2 \le \kappa_2 a^2(m) \|\phi\|_2^2.
	\end{align*}
	The result follows.
\end{proof}

Next we estimate the term $\Tmc_2^m$.
Note that for $t \in [0,T]$
\begin{align}
	\Tmc_2^m(t) & = \frac{a(m)}{\sqrt{m}}\sum_{i=1}^m \int_0^t \left(\phi'(\Xtilims) b(\Xtilims, \mutilms) - \lan \mu(s), \phi'(\cdot) b(\cdot,  \mu(s))\ran \right) \, ds \notag \\
	&= \int_0^t \lan \Ytil^m(s), \phi'(\cdot) b(\cdot,  \mu(s)) \ran \, ds \notag \\
	&\quad + \frac{a(m)}{\sqrt{m}} \sum_{i=1}^m \int_0^t \left(\phi'(\Xtilims) b(\Xtilims, \mutilms) - \phi'(\Xtilims) b(\Xtilims, \mu(s))\right) \, ds \notag \\
	& = \int_0^t \lan \Ytil^m(s), \phi'(\cdot) b(\cdot,\mu(s)) \ran \, ds + \int_0^t \lan \Ytil^m(s), \int_{\Rmb}\phi'(x) \beta(x,  \cdot) \mu_s(dx) \ran \, ds +\Rmc_2^m(t), \label{eqdif:T2}
\end{align}
where $\Rmc_2^m(t) \doteq \int_0^t \Rmctil_{21}^m(s) \, ds$ and
\begin{align*}
	\Rmctil_{21}^m(s) & \doteq \frac{a(m)}{\sqrt{m}} \sum_{i=1}^m \phi'(\Xtilims) [b(\Xtil_i^m(s),\mutil^m(s)) - b(\Xtil_i^m(s),\mu(s))] \\
	& \qquad - a(m)\sqrt{m} \int_\Rmb \phi'(x) [b(x,\mutil^m(s)) - b(x,\mu(s))] \, \mu_s(dx).
\end{align*}
In the following lemma we estimate the remainder term $\Rmc_2^m$.

\begin{lemma} \label{lemdif:R2}
	There exists $\gamma_3 \in (0,\infty)$ such that
	\begin{equation*}
		\Emb |\Rmc_2^m|_{*,T} \le \frac{\gamma_3 \|\phi\|_3}{a(m)\sqrt{m}}.
	\end{equation*}
\end{lemma}

\begin{proof}
	Define $\Rmctil_{22}^m(s)$ by replacing in the definition of $\Rmctil_{21}^m(s)$ the term $\mutil^m(s)$ with $\mubar^m(s)$, namely
	\begin{align}
		\begin{aligned} \label{eqdif:R22}
			\Rmctil_{22}^m(s) & \doteq \frac{a(m)}{\sqrt{m}} \sum_{i=1}^m \phi'(\Xtilims) [b(\Xtil_i^m(s),\mubar^m(s)) - b(\Xtil_i^m(s),\mu(s))] \\
			& \qquad - a(m)\sqrt{m} \int_\Rmb \phi'(x) [b(x,\mubar^m(s)) - b(x,\mu(s))] \, \mu_s(dx).		
		\end{aligned}
	\end{align}
	Using the representation of $b$ in terms of $\beta$ and suitably adding and subtracting terms we see that
	\begin{align}
		& \Rmctil_{21}^m(s) - \Rmctil_{22}^m(s) \notag \\
		& \quad = \frac{a(m)\sqrt{m}}{m^2} \sum_{i=1}^m \sum_{j=1}^m \Bigg\{ [\phi'(\Xtil_i^m(s)) - \phi'(\Xbar_i(s))][\beta(\Xtilims,\Xtiljms) - \beta(\Xtilims,\Xbarjs)] \notag \\
		& \qquad + \phi'(\Xbaris) [\beta(\Xtilims,\Xtiljms) - \beta(\Xtilims,\Xbarjs) - (\Xtiljms - \Xbarjs) \beta_y(\Xtilims,\Xbarjs)] \notag \\
		& \qquad + \phi'(\Xbaris) [\Xtiljms - \Xbarjs] [\beta_y(\Xtilims,\Xbarjs) - \beta_y(\Xbaris,\Xbarjs)] \notag \\
		& \qquad - \intR \phi'(x)[\beta(x,\Xtiljms) - \beta(x,\Xbarjs) - (\Xtiljms - \Xbarjs) \beta_y(x,\Xbarjs)] \, \mu_s(dx) \notag \\
		& \qquad + [\Xtiljms - \Xbarjs] [\phi'(\Xbaris) \beta_y(\Xbaris,\Xbarjs) - \intR \phi'(x) \beta_y(x,\Xbarjs) \, \mu_s(dx)] \Bigg\}. \label{eqdif:R21-R22}
	\end{align}
	We will now compute square of the first absolute moments of the various terms on the right side.
	Using Cauchy-Schwarz inequality and Lipschitz estimates on $\beta$ and $\phi'$, square of the first absolute moment of the first term on the right side of \eqref{eqdif:R21-R22} can be bounded by
	\begin{equation*}
		\kappa_1 a^2(m)m |\phi|_2^2 \left(\Emb \frac{1}{m} \sum_{i=1}^m |\Xtilims - \Xbaris|^2\right) \left(\Emb \frac{1}{m} \sum_{j=1}^m |\Xtiljms - \Xbarjs|^2\right).
	\end{equation*}
	For the second term of the right side, one has the following upper bound on the square of the first absolute moment	
	\begin{equation*}
		\kappa_1 a^2(m)m |\phi|_1^2 \left(\Emb \frac{1}{m} \sum_{j=1}^m |\Xtiljms - \Xbarjs|^2\right)^2.
	\end{equation*}
	This estimate uses Taylor's formula and the fact that $\beta \in \Cmb_b^2(\Rmb^2)$.
	For the third term, we again use Cauchy-Schwarz inequality yielding the bound
	\begin{equation*}
		\kappa_1 a^2(m)m |\phi|_1^2 \left(\Emb \frac{1}{m} \sum_{i=1}^m |\Xtilims - \Xbaris|^2\right) \left(\Emb \frac{1}{m} \sum_{j=1}^m |\Xtiljms - \Xbarjs|^2\right).
	\end{equation*}
	The square of the first absolute moment of the fourth term can be bounded using Taylor's formula as for the second term by the following expression
	\begin{equation*}
		\kappa_1 a^2(m)m |\phi|_1^2 \left(\Emb \frac{1}{m} \sum_{j=1}^m |\Xtiljms - \Xbarjs|^2\right)^2.		
	\end{equation*}
	Finally using Cauchy-Schwarz inequality and the independence of the sequence $\{\Xbar_i\}$ one can estimate the square of the first absolute moment of the last term as
	\begin{align*}
%		& \kappa_1 a^2(m)m \left(\Emb \frac{1}{m} \sum_{j=1}^m |\Xtiljms - \Xbarjs|^2\right) \\
%		& \qquad \cdot \Emb \left( \frac{1}{m^2} \sum_{i=1}^m \sum_{j=1}^m [\phi'(\Xbaris) \beta_y(\Xbaris,\Xbarjs) - \intR \phi'(x) \beta_y(x,\Xbarjs) \, \mu_s(dx)] \right)^2 \\
		\kappa_1 a^2(m)m \left(\Emb \frac{1}{m} \sum_{j=1}^m |\Xtiljms - \Xbarjs|^2\right) \frac{|\phi|_1^2}{m}.
	\end{align*}
	Combing these estimates with Lemma~\ref{lemdif:XtilXbar} we now have
	\begin{equation} \label{eqdif:R21-R22bd}
		(\Emb |\Rmctil_{21}^m(s) - \Rmctil_{22}^m(s)|)^2 \le \kappa_2 |\phi|_2^2 \left\{ \frac{1}{a^2(m)m} + \frac{1}{m} \right\}.
	\end{equation}
	The above estimate allows approximation of $\Rmctil_{21}^m(s)$ by $\Rmctil_{22}^m(s)$.
	Next we will approximate $\Rmctil_{22}^m(s)$ by the term $\Rmctil_{23}^m(s)$ that is obtained by replacing $\Xtil_i^m$ in \eqref{eqdif:R22} with $\Xbar_i$, namely
	\begin{align*}
		\Rmctil_{23}^m(s) & \doteq \frac{a(m)}{\sqrt{m}} \sum_{i=1}^m \phi'(\Xbaris) [b(\Xbaris,\mubarms) - b(\Xbaris,\mu(s))] \\
		& \qquad - a(m)\sqrt{m} \int_\Rmb \phi'(x) [b(x,\mubarms) - b(x,\mu(s))] \, \mu_s(dx).
	\end{align*}
	By similar addition and subtraction of terms as for \eqref{eqdif:R21-R22},
	\begin{align*}
		\Rmctil_{22}^m(s) - \Rmctil_{23}^m(s) & = \frac{a(m)\sqrt{m}}{m^2} \sum_{i=1}^m \sum_{j=1}^m \Bigg( \phi'(\Xtilims) \bigg\{ [\beta(\Xtilims,\Xbarjs) - \beta(\Xbaris,\Xbarjs) \\
		& \quad - (\Xtilims - \Xbaris) \beta_x(\Xbaris,\Xbarjs)] \\
		& \quad - \intR [\beta(\Xtilims,y) - \beta(\Xbaris,y) - (\Xtilims - \Xbaris) \beta_x(\Xbaris,y)] \, \mu_s(dy) \\
		& \quad + [\Xtilims - \Xbaris] [\beta_x(\Xbaris,\Xbarjs) - \intR \beta_x(\Xbaris,y) \, \mu_s(dy)] \bigg\} \\
		& \quad + [\phi'(\Xtilims) - \phi'(\Xbaris)] [\beta(\Xbaris,\Xbarjs) - \intR \beta(\Xbaris,y) \, \mu_s(dy)] \Bigg).
	\end{align*}
	As for the proof of \eqref{eqdif:R21-R22bd}, we have, once more using Cauchy-Schwarz inequality, Taylor series expansion, Lemma~\ref{lemdif:XtilXbar} and the independence of $\{\Xbar_i\}$,
	\begin{equation*}
		(\Emb |\Rmctil_{22}^m(s) - \Rmctil_{23}^m(s)|)^2 \le \kappa_3 |\phi|_2^2 \left(\frac{1}{a^2(m)m} + \frac{1}{m}\right).
	\end{equation*}
	We omit the details.
	Finally writing
	\begin{align*}
		\Rmctil_{23}^m(s) & = \frac{a(m)\sqrt{m}}{m^2} \sum_{i=1}^m \sum_{j=1}^m \Big( \phi'(\Xbaris) [\beta(\Xbaris,\Xbarjs) - b(\Xbaris,\mu(s))] \\
		& \quad - \int_\Rmb \phi'(x) [\beta(x,\Xbarjs) - b(x,\mu(s))] \, \mu_s(dx) \Big)
	\end{align*}	
	and using independence of $\{\Xbar_i\}$, we have
	\begin{align*}
		\Emb |\Rmctil_{23}^m(s)| \le \frac{\kappa_4 a(m) |\phi|_1}{\sqrt{m}}.
	\end{align*}
	Combining the above estimates and using \eqref{eqjump:a m}, \eqref{eqdif:twonorms} gives
	\begin{equation*}
		\Emb |\Rmc_2^m|_{*,T} \le \Emb \int_0^T |\Rmctil_{21}^m(s)| \, ds \le \frac{\kappa_5 |\phi|_2}{a(m)\sqrt{m}} \le \frac{\kappa_6 \|\phi\|_3}{a(m)\sqrt{m}}.
	\end{equation*}
	This completes the proof of the lemma.
\end{proof}

We will now estimate the term $\Tmc_3^m$.
Define $\nutil^m \in \Mmc_T(\Rmb^2 \times [0,T])$ as
\begin{equation}
	\nutil^m(A \times B \times [0,t]) \doteq \frac{1}{m} \sum_{i=1}^m \int_0^t \delta_{(\util_i^m(s), \Xtilims)}(A \times B) \, ds, \quad A, B \in \Bmc(\Rmb), \quad t \in [0,T].
	\label{eqdif:occmzr}
\end{equation}
Then
\begin{align} \label{eqdif:T3}
	\begin{aligned}
		\Tmc^m_3(t) & = \int_{\Rmb^2 \times [0,t]} \phi'(x) \sigma(x, \mutilms) y \, \nutil^m(dy\, dx\, ds)\\
		&= \int_{\Rmb^2 \times [0,t]} \phi'(x) \sigma(x,  \mu(s)) y \, \nutil^m(dy\, dx\, ds) + \Rmc^m_3(t),
	\end{aligned}
\end{align}
where
\begin{equation*}
	\Rmc_3^m(t) \doteq \frac{1}{m} \sum_{i=1}^m \int_0^t \phi'(\Xtilims) [\sigma(\Xtilims,\mutilms) - \sigma(\Xtilims,\mu(s))] \util_i^m(s) \, ds.
\end{equation*}

\begin{lemma} \label{lemdif:R3}
	There exists $\gamma_4 \in (0,\infty)$ such that
	\begin{equation*}
		\Emb |\Rmc_3^m|_{*,T} \le \frac{\gamma_4 \|\phi\|_2}{\big(a^2(m)m\big)^\quarter}.
	\end{equation*}
\end{lemma}

\begin{proof}
	Recall that we assume that \eqref{eqdif:controlbdgeneral} is satisfied.
	Using Cauchy-Schwarz inequality and boundedness of $\alpha$, we have
	\begin{align*}
		\Emb |\Rmc_3^m|_{*,T} & \le \left( \Emb \frac{1}{m} \sum_{i=1}^m \int_0^T \left[\phi'(\Xtilims)\right]^4 ds \right)^\quarter \left( \Emb \frac{1}{m} \sum_{i=1}^m \int_0^T |\util_i^m(s)|^2 \, ds \right)^\half \\
		& \quad \cdot \left( \Emb \frac{1}{m} \sum_{i=1}^m \int_0^T \left[\sigma(\Xtilims,\mutilms) - \sigma(\Xtilims,\mu(s))\right]^4 ds \right)^\quarter \\
		& \le \kappa_1 |\phi|_1 \left( \Emb \frac{1}{m} \sum_{i=1}^m \int_0^T \left[\sigma(\Xtilims,\mutilms) - \sigma(\Xtilims,\mu(s))\right]^2 ds \right)^\quarter \\
		& \le \kappa_2 \|\phi\|_2 \left( \Rmc_{31}^m + \Rmc_{32}^m + \Rmc_{33}^m \right)^\quarter,
	\end{align*}
	where
	\begin{align*}
		\Rmc_{31}^m & \doteq \Emb \frac{1}{m} \sum_{i=1}^m \int_0^T \left[\sigma(\Xtilims,\mutilms) - \sigma(\Xbaris,\mubarms)\right]^2 ds \\
		\Rmc_{32}^m & \doteq \Emb \frac{1}{m} \sum_{i=1}^m \int_0^T \left[\sigma(\Xbaris,\mubarms) - \sigma(\Xbaris,\mu(s))\right]^2 ds \\
		\Rmc_{33}^m & \doteq \Emb \frac{1}{m} \sum_{i=1}^m \int_0^T \left[\sigma(\Xbaris,\mu(s)) - \sigma(\Xtilims,\mu(s))\right]^2 ds.
	\end{align*}
	Using the Lipschitz property of $\alpha$ and Lemma~\ref{lemdif:XtilXbar},
	\begin{equation*}
		\Rmc_{31}^m \le \kappa_3 \Emb \frac{1}{m} \sum_{i=1}^m \int_0^T |\Xtilims - \Xbaris|^2 ds \le \frac{\kappa_4}{a^2(m)m}.
	\end{equation*}
	Similarly,
	\begin{equation*}
		\Rmc_{33}^m \le \kappa_5 \Emb \frac{1}{m} \sum_{i=1}^m \int_0^T |\Xtilims - \Xbaris|^2 \, ds \le \frac{\kappa_6}{a^2(m)m}.
	\end{equation*}
	Finally using independence of $\{\Xbar_i\}$, we have $\Rmc_{32}^m \le \frac{\kappa_7}{m}$. 
	Combining above estimates completes the proof.
\end{proof}

Finally we consider $\Tmc_4^m$.
\begin{align} \label{eqdif:T4}
	\begin{aligned}
		\Tmc^m_4(t) & = \half \frac{a(m)}{\sqrt{m}} \sum_{i=1}^m \int_0^t \left(\phi''(\Xtilims) \sigma^2(\Xtilims, \mutilms) - \lan \mu(s), \phi{''}(\cdot) \sigma^2(\cdot,  \mu(s))\ran\right) \, ds\\
		&= \int_0^t \lan \Ytil^m(s), \half \phi''(\cdot) \sigma^2(\cdot, \mu(s)) \ran \, ds \\
		&\quad + \half \frac{a(m)}{\sqrt{m}} \sum_{i=1}^m \int_0^t \phi''(\Xtilims) \left[\sigma^2(\Xtilims, \mutilms) - \sigma^2(\Xtilims, \mu(s)) \right] \, ds\\
		&= \int_0^t \left\lan \Ytil^m(s), \half \phi''(\cdot) \sigma^2(\cdot, \mu(s)) + \int_{\Rmb} \phi''(x) \sigma(x, \mu(s)) \alpha(x, \cdot)\mu_s(dx) \right\ran \, ds + \Rmc^m_4(t),
	\end{aligned}
\end{align}
where $\Rmc_4^m(t) \doteq \int_0^t \Rmctil_{41}^m(s) \, ds$ and
\begin{align*}
	\Rmctil_{41}^m(s) & \doteq \half \frac{a(m)}{\sqrt{m}} \sum_{i=1}^m \phi''(\Xtilims) \left[\sigma^2(\Xtilims, \mutilms) - \sigma^2(\Xtilims, \mu(s)) \right] \\
	& \qquad - a(m)\sqrt{m} \int_\Rmb \phi''(x) \sigma(x,\mu(s)) \left[\sigma(x,\mutilms) - \sigma(x,\mu(s))\right] \, \mu_s(dx).
\end{align*}
The proof of the following lemma is similar to that of Lemma~\ref{lemdif:R2}, the only difference being that one needs to estimate $\phi''$ rather than $\phi'$.
As a result the bound on the right side contains $\|\phi\|_4$ instead of $\|\phi\|_3$ as in Lemma~\ref{lemdif:R2}.
We omit the proof.

\begin{lemma} \label{lemdif:R4}
	There exists $\gamma_5 \in (0,\infty)$ such that	
	\begin{equation*}
		\Emb |\Rmc_4^m|_{*,T} \le \frac{\gamma_5 \|\phi\|_4}{a(m)\sqrt{m}}.
	\end{equation*}	
\end{lemma}

%\begin{proof}
%	Let
%	\begin{align*}
%		\Rmctil_{42}^m(s) & = \half \frac{a(m)}{\sqrt{m}} \sum_{i=1}^m \phi''(\Xtilims) \left[\sigma^2(\Xtilims, \mubarms) - \sigma^2(\Xtilims, \mu(s)) \right] \\
%		& \qquad - a(m)\sqrt{m} \int_\Rmb \phi''(x) \sigma(x,\mu(s)) \left[\sigma(x,\mubarms) - \sigma(x,\mu(s))\right] \, \mu_s(dx) \\
%		\Rmctil_{43}^m(s) & = \half \frac{a(m)}{\sqrt{m}} \sum_{i=1}^m \phi''(\Xbaris) \left[\sigma^2(\Xbaris,\mubarms) - \sigma^2(\Xbaris,\mu(s)) \right] \\
%		& \qquad - a(m)\sqrt{m} \int_\Rmb \phi''(x) \sigma(x,\mu(s)) \left[\sigma(x,\mubarms) - \sigma(x,\mu(s))\right] \, \mu_s(dx).		
%	\end{align*}
%%	Next note that
%%	\begin{align*}
%%		\Rmc_{41}^m(s) - \Rmc_{42}^m(s) = \frac{a(m)\sqrt{m}}{2m^2} \sum_{i=1}^m \sum_{j=1}^m \Bigg( \left[\phi''(\Xtilims) - \phi'(\Xbaris) \right] \Bigg)		
%%	\end{align*}
%	Similar to the proof of Lemma~\ref{lemdif:R2} but noting that one has to estimate $\phi''$ rather than $\phi'$, we have
%	\begin{align*}
%		\Emb |\Rmctil_{41}^m(s) - \Rmctil_{42}^m(s)| & \le \frac{\kappa_1|\phi|_3}{a(m)\sqrt{m}} \\
%		\Emb |\Rmctil_{42}^m(s) - \Rmctil_{43}^m(s)| & \le \frac{\kappa_2|\phi|_3}{a(m)\sqrt{m}} \\
%		\Emb |\Rmctil_{43}^m(s)| & \le \frac{\kappa_3 a(m) |\phi|_2}{\sqrt{m}}.
%	\end{align*}
%	Hence we have
%	\begin{equation*}
%		\Emb |\Rmc_4^m|_{*,T} \le \Emb \int_0^T |\Rmctil_{41}^m(s)| \, ds \le \frac{\kappa_4 |\phi|_3}{a(m)\sqrt{m}} \le \frac{\kappa_5 \|\phi\|_4}{a(m)\sqrt{m}}.
%	\end{equation*}
%	Proof of the second result is similar and hence omitted.
%\end{proof}

We can now complete the proof of Proposition~\ref{propdif:representation}.

\noindent \textbf{Proof of Proposition~\ref{propdif:representation}:}
Using \eqref{eqdif:splitforprop}, \eqref{eqdif:T2}, \eqref{eqdif:T3} and \eqref{eqdif:T4},
\begin{equation*}
	\lan \Ytil^m(t), \phi \ran = \int_0^t \lan \Ytil^m(s), L(s)\phi \ran \, ds + \int_{\Rmb^2 \times [0,t]} \phi'(x) \sigma(x,  \mu(s)) y \, \nutil^m(dy\, dx\, ds) + \Rmc^m(t),
\end{equation*}
where for $t \in [0,T]$,
\begin{equation*}
	\Rmc^m(t) \doteq \Tmc_1^m(t) + \Rmc_2^m(t) + \Rmc_3^m(t) + \Rmc_4^m(t).
\end{equation*}
The result now follows from Lemmas~\ref{lemdif:R1}-\ref{lemdif:R4}. \qed

%-----------------------------------------------------------------

\subsection{Tightness of $\Ytil^m$.} \label{secdif:tightness}

In this section we will argue the tightness of $\Ytil^m$ in $\Cmb([0,T]:\Smc_{-v})$ and identify the value of $v$.
Note that this tightness implies tightness in $\Cmb([0,T]:\Smc_{-\tau})$ for all $\tau \ge v$.

\begin{theorem}
	\label{thmdif:tightness}
	Suppose that Condition~\ref{conddif:diffusion1} holds.
	Also suppose that the control sequence $\{\util_i^m\}_{i=1}^m$ satisfies \eqref{eqdif:controlbdgeneral}.
	Then the sequence $\{(\Ytil^m, \nutil^m)\}$ is tight in 
	$\Cmb([0,T]:\Smc_{-v}) \times \Mmc_T(\Rmb^2\times [0,T])$ for some $v > 4$.  
\end{theorem}

\begin{proof}
	We first argue the tightness of $\Ytil^m$.
	For this, we will make use of Proposition~\ref{propdif:representation}.
	Let for $t \in [0,T]$ and $\phi \in \Smc$, 
	\begin{align*}
		A^m(t) & \doteq \int_0^t \left| \lan \Ytil^m(s), L(s)\phi \ran \right| ds \\
		& = \int_0^t a(m)\sqrt{m} \left| \lan \mutilms - \mubarms, L(s)\phi \ran + \lan \mubarms - \mu(s), L(s)\phi \ran \right| ds, \\
		B^m(t) & \doteq \int_{\Rmb^2 \times [0,t]} \left| \phi'(x) \sigma(x,  \mu(s)) y \right| \nutil^m(dy\, dx\, ds) \\
		& = \frac{1}{m} \sum_{i=1}^m \int_0^t \left| \phi'(\Xtilims) \sigma(\Xtilims,\mu(s)) \util_i^m(s) \right| ds.
	\end{align*}
	Then for $t_1, t_2 \in [0,T]$,
	\begin{align*}
		& |A^m(t_2) - A^m(t_1)|^2 \\
		& \quad \le 2 a^2(m)m |t_2 - t_1| \int_0^T \left( \lan \mutilms - \mubarms, L(s)\phi \ran^2 + \lan \mubarms - \mu(s), L(s)\phi \ran^2 \right) ds.
	\end{align*}
	Note that
	\begin{align*}
%		& \Emb |A^m(t_2) - A^m(t_1)|^2 \\
		& a^2(m)m \Emb \int_0^T \left( \lan \mutilms - \mubarms, L(s)\phi \ran^2 + \lan \mubarms - \mu(s), L(s)\phi \ran^2 \right) ds \\
		& \quad \le \kappa_1 a^2(m)m \int_0^T \left( |L(s)\phi|_1^2 \Emb \frac{1}{m} \sum_{i=1}^m |\Xtilims - \Xbaris|^2 + \frac{|L(s)\phi|_0^2}{m} \right) ds \\
		& \quad \le \kappa_2 |\phi|_3^2 \left( 1 + a^2(m) \right) \\
		& \quad \le \kappa_3 \|\phi\|_4^2,
	\end{align*}
	where the second inequality uses Lemma~\ref{lemdif:XtilXbar} and the inequality $\sup_{0 \le s \le T} |L(s)\phi|_1 \le \kappa_4 |\phi|_3$.
	This proves the tightness of $A^m$ in $\Cmb([0,T]:\Rmb)$.
	Also for $t_1, t_2 \in [0,T]$
	\begin{equation*}
		|B^m(t_2) - B^m(t_1)|^2 \le \kappa_5 |t_2 - t_1| |\phi|_1^2 \frac{1}{m} \sum_{i=1}^m \int_0^T |\util_i^m(s)|^2 \, ds.
	\end{equation*}
	Combining this with \eqref{eqdif:controlbdgeneral} we now have the tightness of $B^m$ in $\Cmb([0,T]:\Rmb)$.
%	\begin{equation*}
%		\Emb |B^m(t_2) - B^m(t_1)|^2 \le \kappa_5 |t_2 - t_1| |\phi|_1^2 \Emb \frac{1}{m} \sum_{i=1}^m \int_0^T |\util_i^m(s)|^2 \, ds \le \kappa_6 |t_2 - t_1| \|\phi\|_2^2.
%	\end{equation*}
%	From above fluctuation estimates we see that $A^m$ and $B^m$ are tight in $\Cmb([0,T]:\Rmb)$.
	Also from Proposition~\ref{propdif:representation} we have that $\Rmc^m \Rightarrow 0$ in $\Cmb([0,T]:\Rmb)$.
	The tightness of $t \mapsto \lan \Ytil^m(t),\phi \ran$ in $\Cmb([0,T]:\Rmb)$, for each $\phi \in \Smc$, is now immediate.
	From the above estimates on $A^m$, $B^m$ and Proposition~\ref{propdif:representation} we have that, for all $\phi \in \Smc$,
	\begin{equation*}
		\sup_{m \in \Nmb} \Emb \sup_{0 \le t \le T} \left| \lan \Ytil^m(t), \phi \ran \right| \le \kappa_6 \|\phi\|_4.
	\end{equation*}
	This shows that for any $\eps_1 > 0$ and $\eps_2 > 0$, there is a $\delta > 0$ such that
	\begin{equation*}
		\Pmb \left( \sup_{0 \le t \le T} \left| \lan \Ytil^m(t), \phi \ran \right| > \eps_1 \right) \le \eps_2 \quad \text{ if } \|\phi\|_4 \le \delta, \quad m \in \Nmb.
	\end{equation*}
	Thus the induced measures $\Pmb \circ (\Ytil^m)^{-1}$ on $\Cmb([0,T]:\Smc')$ are uniformly $4$-continuous in the sense of~\cite{Mitoma1983}; see Remark (R.$1$) on page $997$ there.
	It then follows from the same remark that the sequence $\{\Ytil^m\}$ is tight in $\Cmb([0,T]:\Smc_{-v})$ for some $v > 4$.
	Specifically, one can take any $v > 4$ for which there is an $r \in \Nmb, 4 < r < v$, such that $\sum_{j=1}^\infty \|\phi_j^v\|_r^2 < \infty$ and $\sum_{j=1}^\infty \|\phi_j^r\|_4^2 < \infty$.
	
	Now we argue tightness of $\nutil^m$.
	We claim that 
	\begin{equation} \label{eqdif:g}
		g(\nu) \doteq \int_{\Rmb^2 \times [0,T]} (x^2 + y^2) \, \nu(dy \, dx \, ds), \quad \nu \in \Mmc_T(\Rmb^2 \times [0,T]) 
	\end{equation}
	is a tightness function on $\Mmc_T(\Rmb^2 \times [0,T])$, namely $g$ is bounded from below and has compact level sets.
	The first property is clear.
	To verify the second property take $c \in (0,\infty)$ and let $M_c \doteq \{ \nu \in \Mmc_T(\Rmb^2 \times [0,T]) : g(\nu) \le c \}$.
	By Markov's inequality, for all $M > 0$,
	\begin{equation*}
		\sup_{\nu \in M_c} \nu \left( \left\{ (y,x,t) \in \Rmb^2 \times [0,T] : x^2 + y^2 > M, t \in [0,T] \right\} \right) \le \frac{c}{M}.
	\end{equation*}
	Hence $M_c$ is relatively compact as a subset of $\Mmc_T(\Rmb^2 \times [0,T])$.
	It remains to show that $M_c$ is closed.
	Let $\{\nu_m\} \subset M_c$ be such that $\nu_m \to \nu$ weakly for some $\nu \in \Mmc_T(\Rmb^2 \times [0,T])$.
	Then by Fatou's lemma,
	\begin{equation*}
		g(\nu) \le \liminf_{m \to \infty} g(\nu_m) \le c,
	\end{equation*}
	and consequently $\nu \in M_c$.
	This proves that $M_c$ is compact for every $c > 0$ and thus it follows that $g$ is a tightness function on $\Mmc_T(\Rmb^2 \times [0,T])$.
	Next, from \eqref{eqdif:Xtilbd} and the assumption that \eqref{eqdif:controlbdgeneral} holds,
	\begin{equation*}
		\sup_{m} \Emb |g(\nutil^m)| = \sup_m \Emb \frac{1}{m} \sum_{i=1}^m \int_0^T \left( |\Xtilims|^2 + |\util_i^m(s)|^2 \right) \, ds < \infty.
	\end{equation*}
	Since $g$ is a tightness function, it follows that $\{\nutil^m\}$ is tight.
\end{proof}

%-----------------------------------------------------------------

\subsection{Laplace Upper Bound.} \label{secdif:Lapupperbd}

In this section we will establish under Conditions \ref{conddif:diffusion1} and \ref{conddif:diffusion2} the following Laplace upper bound
\begin{equation}
	\label{eqdif:lapupp}
	\liminf_{m\to \infty}-a^2(m) \log \Emb \exp \left\{-\frac{1}{a^2(m)} F(Y^m)\right\} \ge \inf_{\zeta \in \Cmb([0,T]: \Smc_{-v}) }\{I(\zeta) + F(\zeta)\},
\end{equation}
where $I(\cdot)$ is as defined in \eqref{eqdif:rate function}, $F \in \Cmb_b(\Cmb([0,T]:\Smc_{-\tau}))$, and $\tau \ge v$.
Fix $\eps \in (0,1)$ and using \eqref{eqdif:repn1b} choose for each $m \in \Nmb$ a sequence $\util^m \doteq \{\util^m_i\}_{i=1}^m$ of controls such that
\begin{equation} \label{eqdif:lapupprep}
	-a^2(m) \log \Emb \exp \left\{-\frac{1}{a^2(m)} F(Y^m)\right\} \ge \Emb \left\{ \half \frac{1}{m}\sum_{i=1}^m\int_0^T |\util_i^m(s)|^2 \, ds	+ F(\Ytil^m)\right\} - \eps,
\end{equation}
where $\Ytil^m$ is as introduced in \eqref{eqdif:Ytilm}, with $\Xtil_i^m$ defined in \eqref{eqdif:Xtilim util} and above choice of $\util^m$.  
Since the left side of \eqref{eqdif:lapupprep} is bounded between $-\|F\|_\infty$ and $\|F\|_\infty$, we can assume $\util^m$ are such that
\begin{equation*}
	\Emb \frac{1}{m}\sum_{i=1}^m\int_0^T |\util_i^m(s)|^2 \, ds \le 4 \|F\|_\infty + 2 \doteq C_F.
\end{equation*}
Let $\nutil^m$ be as introduced in \eqref{eqdif:occmzr}.
Then
\begin{equation}
	\label{eqdif:costbd}
	\Emb \int_{\Rmb^2 \times [0,T]} y^2 \ti\nu^m(dy\, dx\, ds) = \Emb \frac{1}{m} \sum_{i=1}^m \int_0^T |\util^m_i(s)|^2 ds \le C_F.
\end{equation}

From Theorem~\ref{thmdif:tightness} it follows that $\{(\Ytil^m,\nutil^m)\}$ is tight in $\Cmb([0,T]:\Smc_{-v}) \times \Mmc_T(\Rmb^2\times [0,T])$.
The following lemma will enable us to characterize its weak limit points.
Recall that we denote $w \doteq v+2$.
%, where $v$ is as in Theorem~\ref{thmdif:tightness}.

\begin{lemma} \label{lemdif:Ls}
	Suppose Condition~\ref{conddif:diffusion2} holds.
	Then for each $n=1,2,\dotsc,w$, there exists $c_n \in (0,\infty)$ such that for all $s \in [0,T]$ and $\phi \in \Smc_{n+2}$, $\|L(s)\phi\|_{n} \le c_n \|\phi\|_{n+2}$.
	If Condition~\ref{conddif:diffusion3} holds then $w$ can be replaced by $\rho+2$.
\end{lemma}

\begin{proof}
	Note that from Condition~\ref{conddif:diffusion2}(a), for $\phi \in \Smc$,
	\begin{align*}
		\left\|\phi'(\cdot) b(\cdot,\mu(s))\right\|_n^2 & = \sum_{0 \le k \le n} \intR (1+x^2)^{2n} \left( \left[ \phi'(x)b(x,\mu(s)) \right]^{(k)} \right)^2 \, dx \\
		& \le \kappa_1 \sum_{0 \le k \le n} \intR (1+x^2)^{2n+2} |\phi^{(k+1)}(x)|^2 \, dx \\
		& \le \kappa_2 \|\phi\|_{n+1}^2.	
	\end{align*}
	Similarly,
	\begin{equation*}
		\left\|\half \phi''(\cdot) \sigma^2(\cdot,\mu(s))\right\|_n^2 \le \kappa_3 \|\phi\|_{n+2}^2.
	\end{equation*}
	Also using Condition~\ref{conddif:diffusion2}(b),
	\begin{align*}
		\left\|\intR \phi'(y) \beta(y,\cdot) \, \mu_s(dy)\right\|_n^2 & = \sum_{0 \le k \le n} \intR (1+x^2)^{2n} \left( \left[ \intR \phi'(y) \beta(y,x) \, \mu_s(dy) \right]^{(k)} \right)^2 \, dx \\
		& \le \intR \|\beta(y,\cdot)\|_n^2 |\phi|_1^2 \, \mu_s(dy) \\
		& \le \kappa_4 \|\phi\|_2^2,
	\end{align*}
	and similarly
	\begin{equation*}
		\left\|\intR \phi''(y) a(y,\mu(s)) \alpha(y,\cdot) \, \mu_s(dy)\right\|_n^2 \le \kappa_5 \|\phi\|_3^2.
	\end{equation*}
	The result follows on combining the above estimates.
	Proof of the second statement in the lemma is similar and hence omitted.
\end{proof}

We can now establish the following characterization of the weak limit points of $\{(\Ytil^m,\nutil^m)\}$.

\begin{theorem}
	\label{thmdif:maincharac}
	Suppose Condition~\ref{conddif:diffusion2} holds, the sequence of controls satisfies \eqref{eqdif:controlbdgeneral},	and $\{(\Ytil^m, \nutil^m)\}$ converges weakly along a subsequence to $(\Ytil, \nutil)$ in $\Cmb([0,T]:\Smc_{-v}) \times \Mmc_T(\Rmb^2\times [0,T])$.
	Then $\nutil \in \Pmc_{\infty}$ a.s. and $\Ytil$ solves the following equation a.s.:
	For all $\phi \in \Smc$,
	\begin{equation} \label{eqdif:maincharac}
		\lan \Ytil(t), \phi \ran = \int_0^t \lan \Ytil(s), L(s)\phi \ran \, ds	
		+ \int_{\Rmb^2 \times [0,t]} \phi'(x) \sigma(x,\mu(s)) y \, \nutil(dy\, dx\, ds).
	\end{equation}
\end{theorem}

\begin{proof}
	We assume without loss of generality that $(\Ytil^m, \nutil^m) \to (\Ytil, \nutil)$ weakly along the full sequence.
	We first verify $\nutil \in \Pmc_\infty$.
	Let $\nubar^m \in \Mmc_T(\Rmb \times [0,T])$ be defined as
	\begin{equation*}
		\nubar^m(B \times [0,t]) \doteq \int_0^t \frac{1}{m} \sum_{i=1}^m \delta_{\Xbaris}(B) \, ds, \quad B \in \Bmc(\Rmb), t \in [0,T].
	\end{equation*}
	It follows from \eqref{eqdif:XtilXbar} that
	\begin{align*}
		\Emb d_{BL}^2(\nutil_{(2,3)}^m,\nubar^m) & = \Emb \sup_{\|f\|_{BL} \le 1} \left| \lan \nutil_{(2,3)}^m,f \ran - \lan \nubar^m,f \ran \right|^2 \\
		& = \Emb \sup_{\|f\|_{BL} \le 1} \left| \int_0^T \frac{1}{m} \sum_{i=1}^m f(\Xtilims,s) \, ds - \int_0^T \frac{1}{m} \sum_{i=1}^m f(\Xbaris,s) \, ds \right|^2 \\
		& \le \Emb \frac{T}{m} \sum_{i=1}^m \int_0^T |\Xtilims - \Xbaris|^2 \, ds \\
		& \le \frac{\kappa_1}{a^2(m)m} \to 0.
	\end{align*}
	Hence $d_{BL}(\nutil_{(2,3)}^m,\nubar^m) \to 0$ in probability.
	Also for each $f \in \Cmb_b(\Rmb \times [0,T])$, we have, with $\nubar$ as in \eqref{eqdif:nubar},
	\begin{align*}
		\Emb \left| \lan \nubar^m,f \ran - \lan \nubar,f \ran \right|^2 & = \Emb \left| \int_0^T \frac{1}{m} \sum_{i=1}^m f(\Xbaris,s) \, ds - \int_0^T \lan f(\cdot,s),\mu(s) \ran \, ds \right|^2 \\
		& \le \int_0^T \Emb \left( \frac{1}{m} \sum_{i=1}^m \left( f(\Xbaris,s) - \lan f(\cdot,s),\mu(s) \ran \right) \right)^2 ds \to 0.
	\end{align*}
	Combining the above two convergence properties with the fact that $\nutil_{(2,3)}^m \to \nutil_{(2,3)}$ weakly implies that $\nutil_{(2,3)} = \nubar$ a.s.
	Furthermore, it follows from Fatou's lemma and \eqref{eqdif:costbd} that
	\begin{equation*}
		\Emb \int_{\Rmb^2 \times [0,T]} y^2 \, \nutil(dy\,dx\,ds) \le \liminf_{m \to \infty} \Emb \int_{\Rmb^2 \times [0,T]} y^2 \, \nutil^m(dy\,dx\,ds) \le C_F.
%		& = \liminf_{m \to \infty} \Emb \frac{1}{m} \sum_{i=1}^m \int_0^T |\util_i^m(s)|^2 \, ds < \infty.
	\end{equation*}
	Thus we have shown that $\nutil \in \Pmc_\infty$ a.s.
		
	Now we argue that $\Ytil$ solves \eqref{eqdif:maincharac} a.s.
	Using Skorokhod's representation theorem, we can assume that $(\Ytil^m, \nutil^m) \to (\Ytil, \nutil)$ a.s.\ in $\Cmb([0,T]:\Smc_{-v}) \times \Mmc_T(\Rmb^2\times [0,T])$.
%	Since $\Ytil^m \to \Ytil$ in $\Cmb([0,T]:\Smc_{-v})$, we have 
	Then for each $\phi \in \Smc_v$,
	\begin{equation} \label{eqdif:limcvg1}
		\lan \Ytil^m(t),\phi \ran \to \lan \Ytil(t),\phi \ran, \quad \forall t \in [0,T].
	\end{equation}
	It follows from Lemma~\ref{lemdif:Ls} that for each $\phi \in \Smc_w$, 
	\begin{align*}
		\sup_{m \in \Nmb} \sup_{s \in [0,T]} |\lan \Ytil^m(s),L(s)\phi \ran| & \le \sup_{m \in \Nmb} \sup_{s \in [0,T]} \| \Ytil^m(s) \|_{-v} \| L(s)\phi \|_v \\
		& \le \kappa_2 \sup_{m \in \Nmb} \sup_{s \in [0,T]} \| \Ytil^m(s) \|_{-v} \| \phi \|_w < \infty,
	\end{align*}
	and hence by bounded convergence theorem, for every $t \in [0,T]$,
	\begin{equation} \label{eqdif:limcvg2}
		\int_0^t \lan \Ytil^m(s),L(s)\phi \ran \, ds \to \int_0^t \lan \Ytil(s),L(s)\phi \ran \, ds.
	\end{equation}
	In view of \eqref{eqdif:limcvg1}, \eqref{eqdif:limcvg2} and Proposition~\ref{propdif:representation}, to finish the proof, it suffices to show for each $\phi \in \Smc$,
	\begin{equation} \label{eqdif:limcvg3}
		\int_{\Rmb^2 \times [0,t]} \phi'(x) \sigma(x,\mu(s)) y \, \nutil^m(dy\, dx\, ds) \to \int_{\Rmb^2 \times [0,t]} \phi'(x) \sigma(x,\mu(s)) y \, \nutil(dy\, dx\, ds)
	\end{equation}
	in probability.
	To verify above convergence, first note that by convergence of $\nutil^m$ to $\nutil$, for each $K \in (0,\infty)$,
	\begin{equation*}
		\int_{\Rmb^2 \times [0,t]} \phi'(x) \sigma(x,\mu(s)) h_K(y) \, \nutil^m(dy\, dx\, ds) \to \int_{\Rmb^2 \times [0,t]} \phi'(x) \sigma(x,\mu(s)) h_K(y) \, \nutil(dy\, dx\, ds)
	\end{equation*}
	a.s., as $m \to \infty$, where $h_K(y) \doteq y \one_{\{|y| \le K \}} + K \one_{\{y > K \}} -K \one_{\{y < -K \}}$. 
	Also it follows from \eqref{eqdif:costbd} that
	\begin{align*}
		& \sup_m \Emb \left| \int_{\Rmb^2 \times [0,t]} \phi'(x) \sigma(x,\mu(s)) (y-h_K(y)) \, \nutil^m(dy\, dx\, ds) \right| \\
		& \qquad \le |\phi|_1 \|\alpha\|_\infty \sup_m \Emb \int_{\Rmb^2 \times [0,T]} \frac{y^2}{K} \, \nutil^m(dy\, dx\, ds) \\
		& \qquad \le \frac{|\phi|_1 \|\alpha\|_\infty C_F}{K} \to 0
	\end{align*}
	as $K \to \infty$. 
	Similarly, using Fatou's lemma,
	\begin{equation*}
		\Emb \left| \int_{\Rmb^2 \times [0,t]} \phi'(x) \sigma(x,\mu(s)) (y-h_K(y)) \, \nutil(dy\, dx\, ds) \right| \to 0 \quad \text{ as } K \to \infty.
	\end{equation*}
	Combining the above convergence properties we have \eqref{eqdif:limcvg3}, which completes the proof.
\end{proof}

We can now complete the proof of the Laplace upper bound under  Conditions \ref{conddif:diffusion1} and \ref{conddif:diffusion2}.

\noindent \textbf{Proof of the Laplace upper bound:}
Recall that $\{(\Ytil^m,\nutil^m)\}$ is tight.
By a standard subsequential argument we can assume without loss of generality that $(\Ytil^m, \nutil^m)$ converges in distribution to a limit $(\Ytil, \nutil)$ in $\Cmb([0,T]:\Smc_{-v}) \times \Mmc_T(\Rmb^2 \times [0,T])$.  
It follows from Theorem~\ref{thmdif:maincharac} that $\nutil \in \Tmc(\Ytil)$ a.s. 
Also, from \eqref{eqdif:lapupprep}
\begin{align*}
	& \liminf_{m\to \infty}-a^2(m) \log \Emb \exp \left\{-\frac{1}{a^2(m)} F(Y^m)\right\} \\
	& \quad \ge \liminf_{m\to \infty} \Emb\left[\half\int_{\Rmb^2 \times [0,T]} y^2 \, \nutil^m(dy\, dx\, ds) + F(\Ytil^m) \right] - \eps \\
	& \quad \ge \Emb\left [\half\int_{\Rmb^2 \times [0,T]} y^2 \, \nutil(dy\, dx\, ds) + F(\Ytil) \right] - \eps\\
	& \quad \ge \inf_{\zeta \in \Cmb([0,T]: \Smc_{-v})} \left[\inf_{\nu \in \Tmc(\zeta)} \left\{ \half\int_{\Rmb^2 \times [0,T]} y^2 \, \nu(dy\, dx\, ds) \right\} + F(\zeta) \right] - \eps\\
	& \quad = \inf_{\zeta \in \Cmb([0,T]: \Smc_{-v})}\{I(\zeta) + F(\zeta)\} -\eps,
\end{align*}
where the second inequality uses Fatou's lemma and weak convergence of $(\Ytil^m,\nutil^m)$ to $(\Ytil,\nutil)$ in $\Cmb([0,T]:\Smc_{-v}) \times \Mmc_T(\Rmb^2 \times [0,T])$.
Since $\eps>0$ is arbitrary, the desired Laplace upper bound follows.
\qed

%----------------------------------------------------------------------

\subsection{Laplace Lower bound.} \label{secdif:laplowerbd}

In this section we show that under Conditions \ref{conddif:diffusion1} and \ref{conddif:diffusion2}, for every $\tau \ge v$ and $F \in \Cmb_b(\Cmb([0,T]:\Smc_{-\tau}))$,
\begin{equation}
	\label{eqdif:laplow}
	\limsup_{m\to \infty}-a^2(m) \log \Emb \exp \left\{-\frac{1}{a^2(m)} F(Y^m)\right\} \le \inf_{\zeta \in \Cmb([0,T]: \Smc_{-v}) }\{I(\zeta) + F(\zeta)\}.
\end{equation}
Fix $\eps \in (0,1)$. Let $\zeta^*\in \Cmb([0,T]: \Smc_{-v})$ be such that
\begin{equation*}
	I(\zeta^*) + F(\zeta^*) \le \inf_{\zeta \in \Cmb([0,T]: \Smc_{-v}) }\{I(\zeta) + F(\zeta)\} + \eps.
%	 = \inf_{\zeta \in \Cmb([0,T]: \Smc_{-\tau}) }\{I(\zeta) + F(\zeta)\} + \eps.
\end{equation*}
Recalling the definition of $I$ in \eqref{eqdif:rate function}, choose $\nu^* \in \Tmc(\zeta^*)$ such that
\begin{equation*}
	\half\int_{\Rmb^2 \times [0,T]} y^2 \, \nu^*(dy\, dx\, ds) \le I(\zeta^*) + \eps.
\end{equation*}
Recalling that $\nu_{(2,3)}^* = \nubar$, we can disintegrate $\nu^*$ as
\begin{equation*} 
%\label{eqdif:nu*}
	\nu^*(A \times B \times [0,t]) = \int_{B \times [0,t]} \vartheta(s,x,A) \, \mu_s(dx) \, ds, \quad A, B \in \Bmc(\Rmb), \quad t \in [0,T],
\end{equation*}
for some $\vartheta : [0,T] \times \Rmb \times \Bmc(\Rmb) \to [0,1]$ such that for each $A \in \Bmc(\Rmb)$, $\vartheta(\cdot,\cdot,A)$ is a measurable map and for each $(s,x) \in [0,T] \times \Rmb$, $\vartheta(s,x,\cdot) \in \Pmc(\Rmb)$.
Define 
\begin{equation*}
	u(s,x) \doteq \int_{\Rmb} y \, \vartheta(s,x,dy), \quad (s,x) \in [0,T]\times\Rmb.
\end{equation*}
Note that this is finite ($\mu_s(dx)ds$-a.e.) since 
\begin{equation} \label{eqdif:ubdlower}
	\int_{\Rmb\times [0,T]} \left(\int_{\Rmb} |y| \, \vartheta(s,x, dy)\right)^2 \, \mu_s(dx) \, ds \le \int_{\Rmb^2\times [0,T]} y^2 \, \nu^*(dy\, dx\, ds) < \infty.
\end{equation}
Recall the sequence $\{\Xbar_i\}$ defined through \eqref{eqdif:Xbari} in terms of an i.i.d.\ sequence of real Brownian motions $\{W_i\}$ on the filtered probability space $(\Omega,\Fmc,\Pmb,\{\Fmc_t\})$.
Using the same sequence of Brownian motions, let $\{\ti X^m_i\}$ be the solution of the system of SDE in \eqref{eqdif:Xtilim util},
%\begin{align} \label{eqdif:Xtillower}
%	\begin{aligned}
%		\Xtilimt &= x_0 + \int_0^t \sigma(\Xtilims, \mutilms) \, dW_i(s) + \int_0^t b(\Xtilims, \mutilms) \, ds\\
%		&\quad + \frac{1}{a(m)\sqrt{m}}\int_0^t \sigma(\Xtilims, \mutilms) \util^m_i(s) \, ds, \quad i = 1, \dotsc, m,
%	\end{aligned}
%\end{align}
where
\begin{equation*}
	\mutilmt \doteq \frac{1}{m}\sum_{i=1}^m \delta_{\Xtilimt} , \quad \util^m_i(s) \doteq u(s, \Xbaris).
\end{equation*}
It follows from \eqref{eqdif:ubdlower} that
\begin{equation} \label{eqdif:controlbdlower}
	\Emb \frac{1}{m} \sum_{i=1}^m\int_0^T |\util^m_i(s)|^2 \, ds =  \int_{\Rmb\times [0,T]} |u(s,x)|^2 \, \mu_s(dx) \, ds \le \int_{\Rmb^2\times [0,T]} y^2 \, \nu^*(dy\, dx\, ds) < \infty.
\end{equation}
We note that the controls $\util^m$ are defined using the given processes $\{\Xbar_i\}$ and hence $\{\Fmc_t\}$-progressively measurable.
In particular, $\{\util_i^m\}_{i=1}^m$ is a controlled sequence of the form on which infimum is taken in \eqref{eqdif:repn1b}.
Consequently,
\begin{equation} \label{eqdif:repnbdlower}
	-a^2(m) \log \Emb \exp \left\{-\frac{1}{a^2(m)} F(Y^m)\right\}  \le  \Emb\left\{ \half \frac{1}{m}\sum_{i=1}^m\int_0^T |\util_i^m(s)|^2 \, ds	+ F(\Ytil^m)\right\},
\end{equation}
where $\Ytil^m$ is defined as in \eqref{eqdif:Ytilm} with $\Xtil_i^m$ given through \eqref{eqdif:Xtilim util}.

%Let $\Ytil^m$ and $\nutil^m$ be defined using the $\ti X^m_i$ above.
%Note that $|\util^m_i(s)| \le (a(m)\sqrt{m})^\half$.
%Then analogous to the proof of \eqref{eqdif:XtilXbar}, one can have
%\begin{align*}
%	& \Emb | \Xtil_i^m - \Xbar_i|_{*,t}^2 \\
%	& \quad \le \kappa_1 \int_0^t \Emb \Big( |\Xtilims - \Xbaris|^2 + \frac{1}{m} \sum_{j=1}^m |\Xtiljms - \Xbarjs|^2 + \frac{1}{a^2(m)m} |\util_i^m(s)|^2 + \frac{1}{m} \Big) \, ds \\
%	& \quad \le \kappa_2 \int_0^t \Emb | \Xtil_i^m - \Xbar_i|_{*,s}^2 \, ds + \frac{\kappa_2}{(a(m)\sqrt{m})^{\frac{3}{2}}} + \frac{\kappa_2}{m},
%\end{align*}
%where the last inequality follows from exchangeability of $\{\Xtil_i^m\}$.
%Hence
%\begin{equation} \label{eqdif:XtilXbarlower}
%	\Emb |\Xtil_i^m - \Xbar_i|_{*,T}^2 \le \frac{\kappa_3}{(a(m)\sqrt{m})^{\frac{3}{2}}}.
%\end{equation}
We now claim that as $m \to \infty$, for each $\phi \in \Smc$ and $t \in [0,T]$
%\begin{equation} \label{eqdif:claimupper}
%	\begin{aligned}
%		& \int_{\Rmb^2\times [0,t]} \phi'(x) \sigma(x, \mu(s)) y \, \nutil^m(dy\, dx\, ds) = \frac{1}{m} \sum_{i=1}^m \int_0^t \phi'(\Xtilims) \sigma(\Xtilims, \mu(s)) \util^m_i(s) \, ds \\
%		& \to \int_{\Rmb\times [0,t]} \phi'(x) \sigma(x, \mu(s)) u(s,x) \, \mu_s(dx) \, ds = \int_{\Rmb^2\times [0,t]} \phi'(x) \sigma(x, \mu(s)) y \, \nu^*(dy\, dx\, ds).
%	\end{aligned}
%\end{equation}
\begin{equation} \label{eqdif:claimlower}
	\int_{\Rmb^2\times [0,t]} \phi'(x) \sigma(x, \mu(s)) y \, \nutil^m(dy\, dx\, ds) \to \int_{\Rmb^2\times [0,t]} \phi'(x) \sigma(x, \mu(s)) y \, \nu^*(dy\, dx\, ds)
\end{equation}
in probability.
To verify this convergence, let
\begin{align*}
	\ti A^m(t) & \doteq \frac{1}{m} \sum_{i=1}^m\int_0^t \phi'(\Xtilims) \sigma(\Xtilims, \mu(s)) \util^m_i(s) \, ds \\
	\ti B^m(t) & \doteq \frac{1}{m} \sum_{i=1}^m\int_0^t \phi'(\Xbaris) \sigma(\Xbaris, \mu(s)) \util^m_i(s) \, ds \\
%	\ti C^m(t) & = \int_{\Rmb\times [0,t]} \phi'(x) \sigma(x, \mu(s)) u(s,x) \one_{\left\{|u(s,x)| \le \left(a(m)\sqrt{m}\right)^\half\right\}} \, \mu_s(dx) \, ds \\
	\ti C(t) & \doteq \int_{\Rmb\times [0,t]} \phi'(x) \sigma(x, \mu(s)) u(s,x) \, \mu_s(dx) \, ds.
\end{align*}
Using Cauchy-Schwarz inequality and the boundedness and Lipschiz property of $\phi'$ and $\sigma$, we have
\begin{equation*}
	\Emb |\ti A^m - \ti B^m|_{*,T}^2 \le \kappa_1 \left( \Emb \frac{1}{m} \sum_{i=1}^m |\Xtil_i^m - \Xbar_i|_{*,T}^2 \right) \left( \Emb \frac{1}{m} \sum_{i=1}^m \int_0^T |\util_i^m(s)|^2 \, ds \right) \le \frac{\kappa_2}{a^2(m)m} \to 0,
\end{equation*}
where the second inequality follows from \eqref{eqdif:controlbdlower} and Lemma~\ref{lemdif:XtilXbar}.
Also, since $\{\Xbar_i\}$ are i.i.d., 
\begin{equation*}
	\Emb |\ti B^m - \ti C|_{*,T}^2 \le \frac{\kappa_3}{m} \int_{\Rmb \times [0,T]} |u(s,x)|^2 \, \mu_s(dx) \, ds \to 0.
\end{equation*} 
%Finally an application of dominated convergence theorem shows that $|\ti C^m - \ti D|_{*,T}^2 \to 0$.
Thus we have shown that $\Atil^m(t) \to \Ctil(t)$ in probability for each $t \in [0,T]$, which proves the claim \eqref{eqdif:claimlower}.

Next, from Theorem~\ref{thmdif:tightness} and \eqref{eqdif:controlbdlower} we have that $\{\Ytil^m\}$ is tight in $\Cmb([0,T]:\Smc_{-v})$.
Finally using Proposition~\ref{propdif:representation}, the convergence in \eqref{eqdif:claimlower}, and an analogous application of bounded convergence theorem used below \eqref{eqdif:limcvg1}, we see that any limit point $\Ytil$ of $\Ytil^m$ solves the equation
\begin{equation*}
	\lan \Ytil(t), \phi \ran = \int_0^t \lan \Ytil(s), L(s)\phi \ran \, ds + \int_{\Rmb^2 \times [0,t]} \phi'(x) \sigma(x,\mu(s)) y \, \nu^*(dy\, dx\, ds), \quad \phi \in \Smc.
\end{equation*}
In particular, since $\{\Ytil^m\}$ is tight the above equation admits at least one solution.
The following lemma shows that the equation admits only one solution, in particular, since $\nu^* \in \Tmc(\zeta^*)$, any limit point $\Ytil$ satisfies $\Ytil = \zeta^*$ a.s.

\begin{lemma} \label{lemdif:unique}
	Suppose Conditions~\ref{conddif:diffusion1} and~\ref{conddif:diffusion2} hold.
	Then for each $\nu \in \Pmc_\infty$, there exists a unique solution of \eqref{eqdif:char}  in $\Cmb([0,T]:\Smc_{-v})$.  If in addition Condition \ref{conddif:diffusion3} is satisfied, uniqueness holds in
	$\Cmb([0,T]:\Smc_{-\rho})$.
	
\end{lemma}

\begin{proof}
	We only prove the first statement in the lemma; the second statement is proved in a similar manner.   Existence of solutions was argued above; 
	 we now argue uniqueness.
	Suppose $\eta$ and $\etatil$ are two solutions of \eqref{eqdif:char} in $\Cmb([0,T]:\Smc_{-v})$.
	Let $\xi \doteq \eta - \etatil$.
	Then $\xi$ satisfies
	\begin{equation} \label{eqdif:unique}
		\lan \xi(t),\phi \ran = \int_0^t \lan \xi(s),L(s)\phi \ran \, ds.
	\end{equation}
	It suffices to show $\xi = 0$.
	We adapt arguments of Kurtz and Xiong (see Lemma $4.2$ and Appendix in~\cite{KurtzXiong2004}).
	By an analogous argument to Lemma A.$6$ in~\cite{KurtzXiong2004}, using Condition~\ref{conddif:diffusion2}, we have for all $f \in \Smc_{-v}$,
	\begin{equation} \label{eqdif:adjointbd}
		\sup_{0 \le s \le T} \lan f,L^*(s)f \ran_{-w} \le \kappa_1 \|f\|_{-w}^2 \, ,
	\end{equation}
	where $L^*(s) \colon \Smc_{-v} \to \Smc_{-w}$ is the adjoint of $L(s) \colon \Smc_{w} \to \Smc_v$.
	Recall that $\{ \phi_j^w \}$ is an orthonormal basis for $\Smc_w$.
	We can choose this basis such that for each $j \in \Nmb$, $\phi_j^w \in \Smc$.
	It follows from \eqref{eqdif:unique} that
	\begin{equation*}
		\lan \xi(t),\phi_j^w \ran^2 = 2 \int_0^t \lan \xi(s),\phi_j^w \ran \, d\lan \xi(s),\phi_j^w \ran = 2 \int_0^t \lan \xi(s),\phi_j^w \ran \lan \xi(s),L(s)\phi_j^w \ran \, ds.
	\end{equation*}
	Therefore,
	\begin{equation} \label{eqdif:uniquebd}
		\|\xi(t)\|_{-w}^2 = 2 \int_0^t \lan \xi(s),L^*(s)\xi(s) \ran_{-w} \, ds \le \kappa_2 \int_0^t \|\xi(s)\|_{-w}^2 \, ds,
	\end{equation}
	where the last inequality follows from \eqref{eqdif:adjointbd}.
	Thus  $\xi(t) = 0 \:\forall\; t \in [0,T]$ and uniqueness follows.
\end{proof}

We can now complete the proof of the Laplace lower bound.

\noindent \textbf{Proof of the Laplace lower bound:}
The above lemma shows that $\Ytil^m \Rightarrow \zeta^*$ in $\Cmb([0,T]:\Smc_{-v})$.
Combining this with \eqref{eqdif:repnbdlower} and \eqref{eqdif:controlbdlower} gives us
\begin{align*}
	\limsup_{m\to \infty} -a^2(m) \log \Emb \exp \left\{-\frac{1}{a^2(m)} F(Y^m)\right\} & \le
	 \limsup_{m\to \infty} \Emb\left\{ \half \frac{1}{m}\sum_{i=1}^m\int_0^T |\util_i^m(s)|^2 \, ds + F(\Ytil^m)\right\}\\
	& \quad \le \half \int_{\Rmb^2 \times [0,T]} y^2 \, \nu^*(dy\, dx\, ds) + F(\zeta^*)\\
	& \quad \le I(\zeta^*) + F(\zeta^*) + \eps\\
	& \quad \le \inf_{\zeta \in \Cmb([0,T]: \Smc_{-v}) }\{I(\zeta) + F(\zeta)\} + 2\eps.
\end{align*}
%where the second inequality uses Fatou's lemma and the convergence of $\Ytil^m \Rightarrow \zeta^*$ in $\Cmb([0,T]:\Smc_{-v})$.
Since $\eps > 0$ is arbitrary, we have the desired lower bound.
\qed

%-----------------------------------------------------------------

\subsection{$I$ is a rate function.} \label{secdif:rate function}

In this section we prove that under Conditions~\ref{conddif:diffusion1} and~\ref{conddif:diffusion3}, $I$ defined in \eqref{eqdif:rate function} regarded as a map from $\Cmb([0,T]:\Smc_{-\rho})$ to $[0,\infty]$ has compact level sets and is therefore a rate function on $\Cmb([0,T]:\Smc_{-\rho})$.

Fix $K \in (0,\infty)$ and let $\Theta_K \doteq \{ \eta \in \Cmb([0,T]:\Smc_{-\rho}) : I(\eta) \le K \}$.
Let $\{ \eta^m \}_{m \in \Nmb} \subset \Theta_K$.
Then for each $m \in \Nmb$ there exists $\nu^m \in \Tmc(\eta^m)$ such that
\begin{equation} \label{eqdif:rate bound}
	\half \int_{\Rmb^2 \times [0,T]} y^2 \, \nu^m(dy\,dx\,ds) \le K + \frac{1}{m}.
\end{equation}
It follows from \eqref{eqdif:Pinfty} and \eqref{eqdif:Xbarbd} that
\begin{equation*}
	\sup_{m \in \Nmb} \int_{\Rmb^2 \times [0,T]} x^2 \, \nu^m(dy\,dx\,ds) = \int_{\Rmb \times [0,T]} x^2 \, \mu_s(dx) \, ds < \infty.
\end{equation*}
So we have $\sup_{m \in \Nmb} |g(\nu^m))| < \infty$, where $g$ is the tightness function on $\Mmc_T(\Rmb^2 \times [0,T])$ defined in \eqref{eqdif:g}.
Hence $\{ \nu^m \}$ is pre-compact.
Let $\nu^m$ converge along a subsequence (labeled once more as $\{m\}$) to $\nuhat$.
It follows from Fatou's lemma and \eqref{eqdif:rate bound} that $\nuhat \in \Pmc_\infty$.
%Analogous to the argument in Section~\ref{secdif:laplowerbd}, define $u, \{\Xtil_i^m\}, \mutil^m, \{\util_i^m\}, \Ytil^m, \nutil^m$ by replacing $\nu^*$ there with $\nuhat$, we have $\Ytil^m \to \Ytil$ in $\Cmb([0,T]:\Smc_{-v})$, where $\Ytil$ solves the equation
%\begin{equation*}
%	\lan \Ytil(t), \phi \ran =   \int_0^t \lan \Ytil(s), L(s)\phi \ran \, ds + \int_{\Rmb^2 \times [0,t]} \phi'(x) \sigma(x,\mu(s)) y \, \nuhat(dy\, dx\, ds), \quad \phi \in \Smc.
%\end{equation*}
%This together with Lemma~\ref{lemdif:unique} shows the existence and uniqueness of solution to \eqref{eqdif:char}.
Now let $\etahat$ be defined as in \eqref{eqdif:char} with $\nu$ replaced by $\nuhat$.
Note that from Lemma~\ref{lemdif:unique} there is a unique such $\etahat \in \Cmb([0,T]:\Smc_{-v})$.
We claim that $\eta^m \to \etahat$ in $\Cmb([0,T]:\Smc_{-\rho})$.
Once the claim is verified, it will follow from \eqref{eqdif:rate bound} and Fatou's lemma that $I(\etahat) \le K$, which will prove the desired compact level set property. 
Note that both $\eta^m$ and $\etahat$ are in $\Cmb([0,T]:\Smc_{-v})$ and if one could show that $\eta^m \to \etahat$ in $\Cmb([0,T]:\Smc_{-v})$, we would have that $I$ is a rate function on $\Cmb([0,T]:\Smc_{-v})$.
However, that convergence is not immediately obvious.

Now we prove the above claim. 
Disintegrate $\nu^m$ as
\begin{equation*}
	\nu^m(A \times B \times [0,t]) = \int_0^t \nu_s^m(A \times B) \, ds, \quad A,B \in \Bmc(\Rmb), t \in [0,T].
%	\quad \nuhat(A \times B \times [0,t]) = \int_0^t \nuhat_s(A \times B) \, ds
\end{equation*}
%for $A,B \in \Bmc(\Rmb)$ and $t \in [0,T]$.
Since $\nu^m \in \Pmc_\infty$, we have $\nu_s^m \in \Pmc(\Rmb^2)$ for a.e.\ $s \in [0,T]$.
Define for $s \in [0,T]$, the function $J^m(s) \colon \Smc_w \to \Rmb$ as follows: 
\begin{equation*}
	\lan J^m(s),\phi \ran \doteq \int_{\Rmb^2} \phi'(x) \sigma(x,\mu(s)) y \, \nu_s^m(dy\,dx), \quad \phi \in \Smc_w.
%	\quad \lan \Jhat(s),\phi \ran = \int_{\Rmb^2} \phi'(x) \sigma(x,\mu(s)) y \, \nuhat_s(dy\,dx).
\end{equation*}
%
%\begin{align*}
%	\sup_{\|\phi\|_v \le 1}
%\end{align*}
%
It is easy to see that $J^m(s) \in \Smc_{-w}$ for a.e.\ $s \in [0,T]$, in fact it follows from \eqref{eqdif:twonorms} and \eqref{eqdif:rate bound} that
\begin{align}
	\begin{aligned} \label{eqdif:Jmbd}
		\sup_{m \in \Nmb} \int_0^T \| J^m(s) \|_{-w}^2 \, ds & = \sup_{m \in \Nmb} \int_0^T \sup_{\| \phi \|_w = 1} \left( \int_{\Rmb^2} \phi'(x) \sigma(x,\mu(s)) y \, \nu_s^m(dy\,dx) \right)^2 \, ds \\
		& \le \kappa_1 \sup_{m \in \Nmb} \int_{\Rmb^2 \times [0,T]} y^2 \, \nu^m(dy\,dx\,ds) < \infty.
	\end{aligned}
\end{align}
Since $\nu^m \in \Tmc(\eta^m)$, it follows that
\begin{equation*}
	\lan \eta^m(t),\phi \ran = \int_0^t \lan \eta^m(s),L(s)\phi \ran \, ds + \int_0^t \lan J^m(s),\phi \ran \, ds, \quad \forall\; \phi \in \Smc.
\end{equation*}
Analogous to the proof of Lemma~\ref{lemdif:unique} we have
\begin{equation*}
	\lan \eta^m(t),\phi \ran^2 = 2 \int_0^t \lan \eta^m(s),\phi \ran \lan \eta^m(s),L(s)\phi \ran \, ds + 2 \int_0^t \lan \eta^m(s),\phi \ran \lan J^m(s),\phi \ran \, ds.
\end{equation*}
Also note that since $I(\eta^m) \le K < \infty$, we must have that $\eta^m \in \Cmb([0,T]:\Smc_{-v}) \subset \Cmb([0,T]:\Smc_{-w})$.
Thus, as for the proof of \eqref{eqdif:uniquebd},
\begin{align*}
	\| \eta^m(t) \|_{-w}^2 & = 2 \int_0^t \lan \eta^m(s),L^*(s) \eta^m(s) \ran_{-w} \, ds + 2 \int_0^t \lan \eta^m(s), J^m(s) \ran_{-w} \, ds \\
	& \le \kappa_2 \int_0^t \| \eta^m(s) \|_{-w}^2 \, ds + \kappa_2 \int_0^t \| \eta^m(s) \|_{-w} \| J^m(s) \|_{-w} \, ds \\
	& \le \kappa_3 \int_0^t \| \eta^m(s) \|_{-w}^2 \, ds + \kappa_3 \int_0^t \| J^m(s) \|_{-w}^2 \, ds.
\end{align*}
Applying Gronwall's lemma and using \eqref{eqdif:Jmbd}, we have
\begin{equation} \label{eqdif:etambd}
	\sup_{m \in \Nmb} \sup_{t \in [0,T]} \| \eta^m(t) \|_{-w}^2 < \infty.
\end{equation}
Next note that for $t_1,t_2 \in [0,T]$ and $\phi \in \Smc$, by Cauchy-Schwarz inequality
\begin{align*}
	& | \lan \eta^m(t_2),\phi \ran - \lan \eta^m(t_1),\phi \ran |^2 \\
	& \quad \le 2|t_2-t_1| \left( \int_0^T | \lan \eta^m(s),L(s)\phi \ran |^2 \, ds + \int_0^T | \lan J^m(s),\phi \ran |^2 \, ds \right) \\
	& \quad \le 2 |t_2-t_1| \left( \int_0^T \| \eta^m(s) \|_{-w}^2 \|L(s)\phi\|_w^2 \, ds + \int_0^T \|J^m(s)\|_{-w}^2 \|\phi\|_w^2 \, ds \right) \\
	& \quad \le \kappa_4 |t_2-t_1| \|\phi\|_{w+2}^2,
\end{align*}
where the last inequality follows from \eqref{eqdif:etambd}, Lemma~\ref{lemdif:Ls} and \eqref{eqdif:Jmbd}.
%So similar to the tightness argument above \eqref{eqdif:g} in the proof of Theorem~\ref{thmdif:tightness}, $\{\eta^m\}$ is pre-compact in $\Cmb([0,T]:\Smc_{-\rho})$, where $\rho$ is as introduced below \eqref{eqdif:rate function}.
This together with \eqref{eqdif:etambd} implies that $\{\eta^m\}$ is pre-compact in $\Cmb([0,T]:\Smc_{-\rho})$, where $\rho$ is as introduced below \eqref{eqdif:rate function} (see e.g.\ Theorem $2.5.2$ in~\cite{KallianpurXiong1995stochastic}).
Suppose now that $\eta^m$ converges in $\Cmb([0,T]:\Smc_{-\rho})$ along a subsequence (labeled once more as $\{m\}$) to $\etatil$.
Under Condition~\ref{conddif:diffusion3}, for every $\phi \in \Smc$ and $s \in [0,T]$, $L(s)\phi \in \Smc_\rho$.
Thus $\lan \eta^m(s), L(s)\phi \ran \to \lan \etatil(s), L(s)\phi \ran \; \forall \; s \in [0,T]$.
Finally, using \eqref{eqdif:rate bound}, \eqref{eqdif:etambd}, the convergence of $\nu^m$ to $\nuhat$, and an estimate analogous to the one below \eqref{eqdif:limcvg1} (with $v$ replaced by $\rho$ and $w$ by $\rho+2$), we see that
%Finally, one can use \eqref{eqdif:etambd}, $\nu^m \to \nuhat$, \eqref{eqdif:rate bound} and analogous argument below \eqref{eqdif:limcvg1} (with $v$ replaced by $\rho$ and $w$ by $\rho+2$) to show that 
any limit point $\etatil$ of $\{\eta^m\}$ solves \eqref{eqdif:char} with $\nu$ replaced by $\nuhat$.
From the second statement in Lemma~\ref{lemdif:unique} this equation has a unique solution in $\Cmb([0,T]:\Smc_{-\rho})$ and so we must have that $\etatil = \etahat$.
This proves the desired compactness of $\Theta_K$.
\qed

%---------------------------------------------------------

\section{Proofs for the pure jump case} \label{secjump:pf}

In this section we will prove Theorems~\ref{thmjump:MDP} and~\ref{thmjump:alternative expression}.
Throughout the section we assume that Condition~\ref{condjump:cond2} holds.

\subsection{Proof of Theorem~\ref{thmjump:MDP}} \label{secjump:pfMDP}

The basic idea is to make use of a sufficient condition for MDP presented in~\cite{BudhirajaDupuisGanguly2015moderate}.
We begin with some notation.

%Let $\Mmbbar = \Mmc_{FC}(\YT)$ and let $\Pmbbar$ be the unique probability measure on $(\Mmbbar,\Bmc(\Mmbbar))$ under which the canonical map, $\Nbar_* : \Mmbbar \to \Mmbbar$, $\Nbar(\mtil) := \mtil$, is a Poisson random measure with intensity measure $\lambda_\YT$.
%The corresponding expectation operator will be denoted by $\Embbar$.
%Let $\Fmc_t = \sigma \{ \Nbar((0,s] \times A) : 0 \le s \le t, A \in \Bmc(\Ymb) \}$ be the $\sigma$-algebra generated by $\Nbar$, and let $\Fmcbar_t$ denote the completion under $\Pmbbar$.
%We denote by $\Pmcbar$ the prediction $\sigma$-field on $[0,T] \times \Mmbbar$ with the filtration $\{ \Fmcbar_t : 0 \le t \le T \}$ on $(\Mmbbar, \Bmc(\Mmbbar))$.
Recall the PRM $N$ introduced in Section~\ref{secjump:model}.
We denote by $\Pmcbar$ the $\{\Fmc_t\}$-predictable $\sigma$-field on $\Omega \times [0,T]$.
Let $\Amcbar_+$ [resp. $\Amcbar$] be the class of all $(\Pmcbar \otimes \Bmc(\Xmb)) / \Bmc([0,\infty))$ [resp. $(\Pmcbar \otimes \Bmc(\Xmb)) / \Bmc(\Rmb)$]-measurable maps from $\Omega \times \XT$ to $[0,\infty)$ [resp. $\Rmb$].
%Let $\Amcbar_+$ be the class of all $(\Pmcbar \otimes \Bmc(\Xmb)) / \Bmc([0,\infty))$-measurable maps from $\Omega \times \XT$ to $[0,\infty)$.
For $\varphi \in \Amcbar_+$, define a counting process $N^\varphi$ on $\XT$ by
\begin{equation} \label{eqjump:Nphi}
	N^\varphi([0,t] \times U ) \doteq \int_{[0,t] \times U \times [0,\infty)} \one_{[0,\varphi(s,y)]}(r) \, N(ds \, dy \, dr), \quad t \in [0,T], U \in \Bmc(\Xmb).
\end{equation}
We think of $N^\varphi$ as a controlled random measure, with $\varphi$ the control process that produces a thinning of the point process $N$ in a random but non-anticipative manner.

Define $\ell \colon [0,\infty) \to [0,\infty)$ by
\begin{equation*}
	\ell(r) \doteq r \log r -r + 1, \quad r \in [0,\infty).
\end{equation*}
For any $\varphi \in \Amcbar_+$ and $t \in [0,T]$ the quantity
\begin{equation*}
	L_t(\varphi) \doteq \intXt \ell(\varphi(s,y)) \, \lambda(ds \, dy)
\end{equation*}
is well defined as a $[0,\infty]$-valued random variable.
This quantity will appear as a cost term in the representation presented below.
It will be convenient to restrict to the following smaller collection of controls.
For each $n \in \Nmb$ let
\begin{align*}
	\Amcbar_{b,n} \doteq \{ \varphi \in \Amcbar_+ : \text{ for all } (\omega,t) \in \Omega \times [0,T], \frac{1}{n} \le \varphi(\omega,t,y) \le n & \text{ if } y \in [0,n]^2 \\
	\text{and } \varphi(\omega,t,y) = 1 & \text{ if } y \notin [0,n]^2 \}
\end{align*}
and let $\Amcbar_b \doteq \bigcup_{n=1}^\infty \Amcbar_{b,n}$. 

For $m \in \Nmb$ and $M \in (0,\infty)$, consider the spaces
\begin{align}
	\Smc^M_{+,m} & \doteq \left\{ g : \XT \to \Rmb_+ \Bigm| L_T(g) \le \frac{M}{a^2(m)m} \right\}, \label{eqjump:Smc+}\\
	\Smc^M_m & \doteq \left\{ f : \XT \to \Rmb \Bigm| 1 + \frac{1}{a(m)\sqrt{m}}f \doteq g \in \Smc^M_{+,m} \right\}, \label{eqjump:Smc} \\
	\Umc^M_{+,m} & \doteq \left\{ \varphi \in \Amcbar_b \mid \varphi(\omega,\cdot,\cdot) \in \Smc^M_{+,m}, \Pmb \text{-a.s.} \right\}. \label{eqjump:Umc}
\end{align}
%We also let
%\begin{align} \label{eqjump:Umc}
%	\begin{aligned}
%		\Umc^M_{+,m} & = \{ \varphi \in \Amcbar_b : \varphi(\omega,\cdot,\cdot) \in \Smc^M_{+,m}, \Pmb \text{-a.s.} \} \\
%		\Umc^M_m & = \{ \psi \in \Amcbar : \psi(\omega,\cdot,\cdot) \in \Smc^M_m, \Pmb \text{-a.s.} \}.
%	\end{aligned}
%\end{align}
Given $M \in (0,\infty)$, denote by $B_2(M)$ the ball of radius $M$ in $L^2(\lambda)$.
A set $\{ \psi^m \} \subset \Amcbar$ with the property that $\sup_{m \in \Nmb} \| \psi^m \|_{L^2(\lambda)} \le M$ a.s. for some $M < \infty$ will be regarded as a collection of $B_2(M)$-valued random variables, where $B_2(M)$ is equipped with the weak topology on the Hilbert space $L^2(\lambda)$.
Since $B_2(M)$ is weakly compact, such a collection of random variables is automatically tight.
Throughout this section $B_2(M)$ will be regarded as the compact metric space obtained by equipping it with the weak topology on $L^2(\lambda)$.

It follows from~\cite{BudhirajaDupuisGanguly2015moderate} (see Lemma~\ref{lemjump:moment bounds of varphi and psi} below) that if $g \in \Smc_{+,m}^M$ then, with $f \doteq a(m)\sqrt{m}(g-1)$, $f \one_{\{|f| \le a(m)\sqrt{m}\}} \in B_2(C_M)$, where $C_M \doteq \sqrt{M\gammatil_2(1)}$ and $\gammatil_2(1) \in (0,\infty)$ is as in Lemma~\ref{lemjump:property of l} below.
Let $\Smb$ be a Polish space.
The following condition on a sequence $\{\Gmc^m\}$ of measurable maps from $\Mmb$ to $\Smb$ and a measurable map $\Gmc_0 \colon L^2(\lambda) \to \Smb$ was introduced in~\cite{BudhirajaDupuisGanguly2015moderate} (see Condition $2.2$ therein).

\begin{condition} \label{condjump:cond3 to be checked}
	(a) Given $M \in (0,\infty)$, suppose that $g^m$, $g \in B_2(M)$ and $g^m \to g$.
	Then 
	\begin{equation*}
		\Gmc_0(g^m) \to \Gmc_0(g).
	\end{equation*}
	(b) Given $M \in (0,\infty)$, let $\{ \varphi^m \}_{m \in \Nmb}$ be such that for every $m \in \Nmb$, $\varphi^m \in \Umc^M_{+,m}$ and for some $\beta \in (0,1]$, $\psi^m \one_{\{ | \psi^m | \le \beta a(m)\sqrt{m} \}} \Rightarrow \psi$ in $B_2(C_M)$ where $\psi^m \doteq a(m)\sqrt{m}(\varphi^m-1)$.
	Then
	\begin{equation*}
		\Gmc^m\left(\frac{1}{m} N^{m\varphi^m}\right) \Rightarrow \Gmc_0(\psi).
	\end{equation*}
\end{condition}

Theorem $2.3$ of~\cite{BudhirajaDupuisGanguly2015moderate} says that if a collection of $\Gmc^m$ satisfies Condition~\ref{condjump:cond3 to be checked}, then $\{ \Gmc^m(\frac{1}{m} N^{m\varphi^m}) \}_{m \in \Nmb}$ satisfies an LDP on $\Smb$ with speed $a^2(m)$ and rate function $\Imc$ given by
%\begin{equation*}
%	\Imc(\eta) = \inf_{\psi \in \Smb_\eta[\Gmc_0]} \left\{ \half \|\psi\|_{L^2(\lambda)}^2 \right\}, \quad \Smb_\eta[\Gmc_0] = \left\{ \psi \in L^2(\lambda) : \eta = \Gmc_0(\psi) \right\}, \quad \eta \in \Dmb([0,T]:l_2).
%\end{equation*}
\begin{equation} \label{eqjump:rate reference}
	\Imc(\eta) \doteq \inf_{\psi \in L^2(\lambda) : \eta = \Gmc_0(\psi)} \left\{ \half \|\psi\|_{L^2(\lambda)}^2 \right\}, \quad \eta \in \Smb.
\end{equation}
We will now use this theorem to establish an MDP for $\mu^m$.

From Theorem~\ref{thmjump:LLN} we have that there exists a measurable map $\Gmcbar^m \colon \Mmb \to \Dmb([0,T]:l_2)$ such that $\mu^m = \Gmcbar^m(\frac{1}{m} N^m)$, and hence there is a map $\Gmc^m \colon \Mmb \to \Dmb([0,T] : l_2)$ such that with $Z^m$ defined as in \eqref{eqjump:Z m}, $Z^m = \Gmc^m(\frac{1}{m} N^m)$.
Define $\Gmc_0 \colon L^2(\lambda) \to \Dmb([0,T]:l_2)$ by 
\begin{equation} \label{eqjump:G 0}
	\Gmc_0(\psi) \doteq \eta \text{ if for } \psi \in L^2(\lambda), \eta \text{ solves } \eqref{eqjump:eta and psi}.
\end{equation}
Note that the map is well defined since for each $\psi \in L^2(\lambda)$ there is a unique $\eta \in \Cmb([0,T]:l_2)$ solving \eqref{eqjump:eta and psi}.
%In order to prove the theorem we will verify that the following condition from~\cite{BudhirajaDupuisGanguly2015moderate} (see Condition $2.2$ therein) holds with these choices of $\Gmc^m$ and $\Gmc_0$.
It is easy to check that with the above choice of $\Gmc_0$, $\Imc$ defined in \eqref{eqjump:rate reference} (with $\Smb = \Dmb([0,T]:l_2)$) is same as the function $\Ibar$ introduced in \eqref{eqjump:rate function}.
Thus in order to prove Theorem~\ref{thmjump:MDP} it suffices to check that Condition~\ref{condjump:cond3 to be checked} holds with $\Smb = \Dmb([0,T]:l_2)$ and the above choice of $\{\Gmc^m\}$ and $\Gmc_0$.
Rest of the section is devoted to the verification of this condition.

All statements except the last one in Lemma~\ref{lemjump:property of l}(a) below have been established in~\cite{BudhirajaDupuisGanguly2015moderate} (see Lemma $3.1$ therein).
The last statement in Lemma~\ref{lemjump:property of l}(a) is crucially used in our proofs and is a key ingredient in overcoming the lack of regularity of $G$ (see proof of Proposition~\ref{propjump:verifying condition part b}).

\begin{lemma} \label{lemjump:property of l}
	(a) For each $\beta > 0$, there exist $\gammatil_1(\beta), \gammatil_1'(\beta) \in (0,\infty)$ such that
	\begin{equation*}
		|x-1| \le \gammatil_1(\beta) \ell(x) \text{ for } |x-1| \ge \beta, x \ge 0, \text{ and } x \le \gammatil'_1(\beta) \ell(x) \text{ for } x \ge \beta > 1.
	\end{equation*}
	Furthermore, $\gammatil_1$ can be selected to be such that for $\beta \in (0,\frac{1}{2})$, $\gammatil_1(\beta) \le \frac{4}{\beta}$.
	
	\noi (b) For each $\beta > 0$, there exist $\gammatil_2(\beta) \in (0,\infty)$ such that
	\begin{equation*}
		|x-1|^2 \le \gammatil_2(\beta) \ell(x) \text{ for } |x-1| \le \beta, x \ge 0.
	\end{equation*}
\end{lemma}

\begin{proof}
	We only need to prove the last statement in part (a).
	Note that we can set
	\begin{equation*}
		\gammatil_1(\beta) = \sup_{|x-1| \ge \beta, x \ge 0} \frac{|x-1|}{\ell(x)}.
	\end{equation*}	
	For $\beta \in (0,\frac{1}{2})$, $x \ge 0$ and $|x-1| \ge \beta$, consider function $f(x) \doteq \frac{x-1}{\ell(x)}$.
	Since $\log (1+u) \le u$ for $u \ge -1$, we have
	\begin{equation*}
		f'(x) = \frac{\log x - (x-1)}{\ell^2(x)} < 0 \text{ for } x \ge 0, |x-1| \ge \beta.
	\end{equation*}
	Since $f(0)=-1$, $\lim_{x \to 1-} f(x) = -\infty$, $\lim_{x \to 1+} f(x) = \infty$ and $\lim_{x \to \infty} f(x)=0$, we have
	\begin{equation*}
		\gammatil_1(\beta) = \sup_{|x-1| \ge \beta, x \ge 0} \frac{|x-1|}{\ell(x)} = \sup_{|x-1| \ge \beta, x \ge 0} |f(x)| = \max \{|f(1+\beta)|, |f(1-\beta)|\}.
	\end{equation*}
	For $\beta \in (0,\frac{1}{2})$, let $h_1(\beta) \doteq \beta |f(1+\beta)|$ and $h_2(\beta) \doteq \beta |f(1-\beta)|$.
	It suffices to show $h_i(\beta) \le 4$ for $i=1,2$.
	Since $\log (1+u) \ge u - \frac{u^2}{2}$ for $u \ge 0$, we have
	\begin{equation*}
		h_1(\beta) = \frac{\beta^2}{(1+\beta) \log (1+\beta) - \beta} \le \frac{\beta^2}{(1+\beta)(\beta-\frac{\beta^2}{2}) - \beta} = \frac{1}{\frac{1}{2} - \frac{\beta}{2}} \le 4.
	\end{equation*}
	Also, since $\log (1+u) \le u - \frac{u^2}{2} + \frac{u^3}{3}$ for $u \ge 0$, we have
	\begin{align*}
		h_2(\beta) & = \frac{\beta^2}{(1-\beta) \log (1-\beta) + \beta}  = \frac{\beta^2}{-(1-\beta) \log (1 + \frac{\beta}{1-\beta}) + \beta} \\
		& \le \frac{\beta^2}{-(1-\beta)(\frac{\beta}{1-\beta} - \frac{1}{2}(\frac{\beta}{1-\beta})^2 + \frac{1}{3}(\frac{\beta}{1-\beta})^3) + \beta} = \frac{6(1-\beta)^2}{3 - 5\beta} \le 4.
	\end{align*}
	Thus we have shown that $h_i(\beta) \le 4$ for $i=1,2$, $\forall \beta \in (0,\half)$.
	The result follows.
\end{proof}

The following lemma is taken from~\cite{BudhirajaDupuisGanguly2015moderate} (see Lemma $3.2$ therein).

\begin{lemma} \label{lemjump:moment bounds of varphi and psi}
	Suppose $g \in \Smc^M_{+,m}$ for some $M \in (0,\infty)$.
	Let $f \doteq a(m)\sqrt{m} (g-1) \in \Smc^M_m$.
	Then
	\begin{flalign*}
		& (a) \intXT |f| \one_{\{ |f| \ge \beta a(m)\sqrt{m} \}} \, d\lambda \le \frac{M\gammatil_1(\beta)}{a(m)\sqrt{m}} \text{ for } \beta > 0 &\\
		& (b) \intXT g \one_{\{ g \ge \beta \}} \, d\lambda \le \frac{M\gammatil_1'(\beta)}{a^2(m) m} \text{ for } \beta > 1 \\
		& (c) \intXT |f|^2 \one_{\{ |f| \le \beta a(m)\sqrt{m} \}} \, d\lambda \le M\gammatil_2(\beta) \text{ for } \beta > 0,
	\end{flalign*}
	where $\gammatil_1, \gammatil'_1$ and $\gammatil_2$ are as in Lemma~\ref{lemjump:property of l}.
\end{lemma}

We will now proceed to the verification of Condition~\ref{condjump:cond3 to be checked}.
We begin with verifying part (a) of the condition.
The following moment bounds on $G$ will be useful.

\begin{lemma} \label{lemjump:moment bounds on G}
	For each $k \in \Nmb$, we have
	\begin{equation*}
		\sup_{q \in \Smchat} \int_\Xmb \| G(q,y) \|^k \, \lambda_\Xmb(dy) \le 2^{k/2} \Gammainf.
	\end{equation*}	
%	hence for each measurable $I \subset [0,T]$ and measurable $\Smchat$-valued process $\{ q(t), t \in [0,T] \}$, we have
%	\begin{equation*}
%		\int_{I \times \Xmb} \| G(q(s),y) \|^k \, \lambda(ds \, dy) \le 2^{k/2} \Gammainf |I|,
%	\end{equation*}
%	where $|I| = \lambda_T(I)$ is the Lebesgue measure of $I$.
\end{lemma}

\begin{proof}
%	Given $i$, $j \in \Nmb$, let
%%	\begin{equation*}
%		$B_{ij} = \{ y \in \Xmb : i-1 < y_1 \le i, \: (j-1) \Gammainf < y_2 \le j \Gammainf \}$.
%%	\end{equation*}
%	Note that for all $q \in \Smchat$ and $y \in \Xmb$, $G(q,y) \one_{B_{ii}}(y) = 0$. 
%	Also $B_{ij} \supset A_{ij}(q)$ are disjoint subsets of $\Xmb$ for different $(i,j)$ pairs.
%	So we have
	Recalling the definition of $G$ in \eqref{eqjump:G}, we have
	\begin{align*}
		\int_\Xmb \| G(q,y) \|^k \, \lambda_\Xmb(dy) & = \sum_{i=1}^\infty \sum_{j \ne i}^\infty \int_{A_{ij}(q)} \| G(q,y) \|^k \, \lambda_\Xmb(dy) = \sum_{i=1}^\infty \sum_{j \ne i}^\infty 2^{k/2} \lambda_\Xmb(A_{ij}(q)) \\
		& = \sum_{i=1}^\infty \sum_{j \ne i}^\infty 2^{k/2} q_i \Gamma_{ij}(q) \le \sum_{i=1}^\infty 2^{k/2} \Gammainf q_i = 2^{k/2} \Gammainf,
	\end{align*}
	where the first two equalities use the property \eqref{eqjump:Aijdisjoint} of the sets $\{A_{ij}(q)\}$.
	The result follows.
\end{proof}

The following lemma provides a key convergence property.

\begin{lemma} \label{lemjump:DCT}
	Fix $M \in (0,\infty)$.
	Suppose that $g^m$, $g \in B_2(M)$ and $g^m \to g$.
	Then
	\begin{equation*}
		\int_{[0,\cdot] \times \Xmb} g^m(s,y) G(p(s),y) \, \lambda(ds \, dy) \to \int_{[0,\cdot] \times \Xmb} g(s,y) G(p(s),y) \, \lambda(ds \, dy) \text{ in } \Cmb([0,T]:l_2).
	\end{equation*}
\end{lemma}

\begin{proof}
	It follows from Lemma~\ref{lemjump:moment bounds on G} that $(s,y) \mapsto G_i(p(s),y)$ is in $L_2(\lambda)$ for each $i \in \Nmb$.
	Thus, since $g^m \to g$ in $B_2(M)$, we have for every $t \in [0,T]$ and $i \in \Nmb$,
	\begin{equation*} 
		\int_\Xt g^m(s,y) G_i(p(s),y) \, \lambda(ds \, dy) \to \int_\Xt g(s,y) G_i(p(s),y) \, \lambda(ds \, dy).
	\end{equation*}
	Note that for each $i \in \Nmb$,
	\begin{align*}
		\left| \int_\Xt g^m(s,y) G_i(p(s),y) \, \lambda(ds \, dy) \right| & \le \left( \int_\Xt |g^m(s,y)|^2 \, \lambda(ds \, dy) \right)^{1/2} \left( \int_\Xt G_i^2(p(s),y) \, \lambda(ds \, dy) \right)^{1/2} \\
		& \le M \left( \int_\Xt G_i^2(p(s),y) \, \lambda(ds \, dy) \right)^{1/2} \doteq \alpha_i.
	\end{align*}
	From Lemma~\ref{lemjump:moment bounds on G} we see that $\sum_{i=1}^\infty \alpha_i^2 < \infty$ and so by dominated convergence theorem for each fixed $t \in [0,T]$,
	\begin{equation} \label{eqjump:DCT}
		\int_\Xt g^m(s,y) G(p(s),y) \, \lambda(ds \, dy) \to \int_\Xt g(s,y) G(p(s),y) \, \lambda(ds \, dy).
	\end{equation}
	To argue that the convergence is in fact uniform in $t$, note that by Lemma~\ref{lemjump:moment bounds on G} once again, for $0 \le s \le t \le T$
	\begin{align*}
		& \quad \left\| \int_{[s,t] \times \Xmb} g^m(s,y) G(p(s),y) \, \lambda(ds \, dy) \right\|^2 \\
		& \le \int_{[0,T] \times \Xmb} | g^m(s,y) |^2 \, \lambda(ds \, dy) \int_{[s,t] \times \Xmb} \| G(p(s),y) \|^2 \, \lambda(ds \, dy) \\
		& \le 2 \Gammainf M^2 | t - s |.
	\end{align*}
	This implies equicontinuity, which shows that the convergence in \eqref{eqjump:DCT} is in fact uniform.	
\end{proof}

Now we are able to verify part (a) of Condition~\ref{condjump:cond3 to be checked}.

\begin{proposition} \label{propjump:verifying condition part a}
%	Suppose Condition~\ref{condjump:cond2} holds.
	Fix $M \in (0,\infty)$.
	Suppose that $g^m$, $g \in B_2(M)$ and $g^m \to g$.
	Let $\Gmc_0$ be as defined in \eqref{eqjump:G 0}.
	Then $\Gmc_0(g^m) \to \Gmc_0(g)$.
\end{proposition} 

\begin{proof}
	Let $\eta^m \doteq \Gmc_0(g^m)$ and $\eta \doteq \Gmc_0(g)$.
	Then
	\begin{equation*}
		\eta^m(t) - \eta(t) = \int_0^t Db(p(s))[\eta^m(s) - \eta(s)] \, ds + \intXt \big(g^m(s,y)-g(s,y)\big) G(p(s),y) \, \lambda(ds \, dy).
	\end{equation*}
	The result now follows on applying Gronwall's lemma together with Condition~\ref{condjump:cond2}(c) and Lemma~\ref{lemjump:DCT}.
\end{proof}

In order to verify part (b) of Condition~\ref{condjump:cond3 to be checked}, we first prove some estimates.
Recall spaces $\Smc^M_{+,m}$ and $\Smc^M_{m}$ introduced in \eqref{eqjump:Smc+} and \eqref{eqjump:Smc}.

\begin{lemma} \label{lemjump:moment bounds on G square varphi}
	Let $M \in (0,\infty)$.
	Then there exists $\gammatil_3 \in (0,\infty)$ such that for all measurable maps $q \colon [0,T] \to \Smchat$,
	\begin{equation*}
		\sup_{m \in \Nmb} \sup_{g \in \Smc^M_{+,m}} \intXT \| G(q(s),y) \|^2 g(s,y) \, \lambda(ds \, dy) \le \gammatil_3.
	\end{equation*}
\end{lemma}

\begin{proof}
	Fix $g \in \Smc^M_{+,m}$.
	Then
	\begin{align*}
		& \quad \intXT \| G(q(s),y) \|^2 g(s,y) \, \lambda(ds \, dy) \\
		& = \int_{\{g \ge 2\}} \| G(q(s),y) \|^2 g(s,y) \, \lambda(ds \, dy) + \int_{\{g < 2\}} \| G(q(s),y) \|^2 g(s,y) \, \lambda(ds \, dy) \\
		& \le 2 \int_{\{g \ge 2\}} g(s,y) \, \lambda(ds \, dy) + 2 \intXT \| G(q(s),y) \|^2 \, \lambda(ds \, dy) \\
		& \le \frac{2M\gammatil_1'(2)}{a^2(m) m} + 4 \Gammainf T,
	\end{align*}
	where the last inequality follows from Lemmas~\ref{lemjump:moment bounds of varphi and psi}(b) and~\ref{lemjump:moment bounds on G}.
	Since $a^2(m)m \to \infty$ as $m \to \infty$, we have the result.
\end{proof}

The following lemma will be needed in the proof of the estimate \eqref{eqjump:Cmphi} in Lemma~\ref{lemjump:moment bounds on Z bar} and \eqref{eqjump:Emc2} in Proposition~\ref{propjump:verifying condition part b}.

\begin{lemma} \label{lemjump:moment bounds on G psi}
	Let $M \in (0,\infty)$.
	Then there exists a map $\gammatil_4 \colon (0,\infty) \to (0,\infty)$ such that for all $m \in \Nmb$, $\beta \in (0,\infty)$, measurable $I \subset [0,T]$ and measurable maps $q \colon [0,T] \to \Smchat$,
	\begin{align*}
		\sup_{f \in \Smc^M_m} & \int_{I \times \Xmb}  \left\| G(q(s),y) \right\| |f(s,y)| \one_{\{|f| \ge \beta a(m)\sqrt{m}\}} \, \lambda(ds \, dy) \le \frac{\gammatil_4(\beta)}{a(m) \sqrt{m}}, \\
		\sup_{f \in \Smc^M_m} & \Big\| \int_{I \times \Xmb}  G(q(s),y) f(s,y) \, \lambda(ds \, dy) \Big\| \le \gammatil_4(\beta) \Big( \frac{1}{a(m) \sqrt{m}} + \sqrt{|I|} \Big).
	\end{align*}
\end{lemma}

\begin{proof}
	Fix $f \in \Smc^M_m$.
	Note that
	\begin{align*}
		\left\| \int_{I \times \Xmb}  G(q(s),y) f(s,y) \, \lambda(ds \, dy) \right\|
%		& \le \Big\| \int_{I \times \Xmb}  G(q(s),y) f \one_{\{|f| \ge \beta a(m)\sqrt{m}\}} \, \lambda(ds \, dy) \Big\| + \Big\| \int_{I \times \Xmb}  G(q(s),y) f \one_{\{|f| < \beta a(m)\sqrt{m}\}} \, \lambda(ds \, dy) \Big\| \\
		& \le \int_{I \times \Xmb} \| G(q(s),y) \| |f(s,y)| \one_{\{|f(s,y)| \ge \beta a(m)\sqrt{m}\}} \, \lambda(ds \, dy) \\
		& \quad + \int_{I \times \Xmb} \| G(q(s),y) \| |f(s,y)| \one_{\{|f(s,y)| < \beta a(m)\sqrt{m}\}} \, \lambda(ds \, dy).
	\end{align*}
	It follows from Lemma~\ref{lemjump:moment bounds of varphi and psi}(a) that
	\begin{align*}
		& \int_{I \times \Xmb} \| G(q(s),y) \| |f(s,y)| \one_{\{|f(s,y)| \ge \beta a(m)\sqrt{m}\}} \, \lambda(ds \, dy) \\
		& \qquad \le \sqrt{2} \int_{I \times \Xmb} |f(s,y)| \one_{\{|f(s,y)| \ge \beta a(m)\sqrt{m}\}} \, \lambda(ds \, dy) \le \frac{\sqrt{2} M \gammatil_1(\beta)}{a(m) \sqrt{m}}.
	\end{align*}
	This proves the first inequality in the lemma.
	From Cauchy-Schwarz inequality it follows that
	\begin{align*}
		& \int_{I \times \Xmb} \| G(q(s),y) \| |f(s,y)| \one_{\{|f(s,y)| < \beta a(m)\sqrt{m}\}} \, \lambda(ds \, dy) \\
		& \qquad \le \Big( \int_{I \times \Xmb} \|G(q(s),y)\|^2 \, \lambda(ds \, dy) \Big)^{1/2} \Big( \int_{I \times \Xmb} |f(s,y)|^2 \one_{\{|f(s,y)| < \beta a(m)\sqrt{m}\}} \, \lambda(ds \, dy) \Big)^{1/2} \\
		& \qquad \le \sqrt{2 \Gammainf M \gammatil_2(\beta) |I|},
	\end{align*}
	where the last inequality follows from Lemmas~\ref{lemjump:moment bounds of varphi and psi}(c) and~\ref{lemjump:moment bounds on G}.
	The second inequality in the lemma now follows by combining the above two displays.
\end{proof}

Recall the space $\Umc^M_{+,m}$ and the map $\Gmc^m$ introduced in \eqref{eqjump:Umc} and above \eqref{eqjump:G 0}, respectively.
Let for $\varphi \in \Umc^M_{+,m}$, $\Zbar^{m,\varphi} \doteq \Gmc^m(\frac{1}{m} N^{m\varphi})$, where $N^{m\varphi}$ is as defined in \eqref{eqjump:Nphi}.
Then it follows from an application of Girsanov's theorem that (see for example the arguments above Lemma $4.4$ in~\cite{BudhirajaDupuisGanguly2015moderate})
\begin{equation} \label{eqjump:Zbar m varphi}
	\Zbar^{m,\varphi} = a(m)\sqrt{m} (\mubar^{m,\varphi} - p),
\end{equation}
where $\mubar^{m,\varphi}$ is the unique pathwise solution of 
\begin{equation*}
	\mubar^{m,\varphi}(t) = \mu^m(0) + \frac{1}{m} \intXt G(\mubar^{m,\varphi}(s-),y) \, N^{m\varphi}(ds \, dy). 
\end{equation*}
%We will like to show that for every $\varphi \in \Umc^M_{+,m}$ the integral equation
%\begin{equation}
%	\mubar^{m,\varphi}(t) = \mu^m(0) + \frac{1}{m} \intXt G(\mubar^m(s-),y) \, N^{m\varphi}(ds \, dy) 
%\end{equation}
%has a unique pathwise solution, the proof of which is going to use Girsanov's theorem and is our proposed future work.
%Define $\Zbar^{m,\varphi} \equiv \Gmc^m(\frac{1}{m} N^{m\varphi})$, and note that this is equivalent to
%\begin{equation} \label{eqjump:Z bar}
%	\Zbar^{m,\varphi} = a(m)\sqrt{m} (\mubar^{m,\varphi} - p).
%\end{equation}
The following moment bounds on $\Zbar^{m,\varphi}$ will be useful for our analysis.

\begin{lemma} \label{lemjump:moment bounds on Z bar}  For every $M \in (0, \infty)$,
	\begin{equation*}
		\sup_{m \in \Nmb} \sup_{\varphi \in \Umc^M_{+,m}} \Emb \left\| \Zbar^{m,\varphi} \right\|_{*,T}^2 < \infty .
	\end{equation*}
\end{lemma} 

\begin{proof}
	Given $\varphi \in \Umc^M_{+,m}$, let $\Ntil^{m\varphi}(ds \, dy) \doteq N^{m\varphi}(ds \, dy) - m \varphi(s,y) \lambda(ds \, dy)$ and $\psi \doteq a(m)\sqrt{m} (\varphi-1)$.
	Then recalling \eqref{eqjump:p t} and that $b(q) = \int_\Xmb G(q,y) \, \lambda_\Xmb(dy)$,
	\begin{align*}
		\mubar^{m,\varphi}(t) - p(t) & = \mu^m(0) - p(0) + \frac{1}{m} \intXt G(\mubar^{m,\varphi}(s-),y) \, \Ntil^{m\varphi}(ds \, dy) \\
		& \quad + \intXt \Big( G(\mubar^{m,\varphi}(s),y) - G(p(s),y) \Big) \lambda(ds \, dy) \\
		& \quad + \intXt G(\mubar^{m,\varphi}(s),y) \big( \varphi(s,y) - 1 \big) \lambda(ds \, dy). 
	\end{align*}
	Using this and \eqref{eqjump:Zbar m varphi}, write
	\begin{equation} \label{eqjump:decomposition of Zbar}
	\Zbar^{m,\varphi} = A^m + M^{m,\varphi} + B^{m,\varphi} + C^{m,\varphi},
	\end{equation}
	where
	\begin{align*}
		A^m & \doteq a(m)\sqrt{m} (\mu^m(0) - p(0)), \\
		M^{m,\varphi}(t) & \doteq \frac{a(m)}{\sqrt{m}} \intXt G(\mubar^{m,\varphi}(s-),y) \, \Ntil^{m\varphi}(ds \, dy), \\
		B^{m,\varphi}(t) & \doteq a(m)\sqrt{m} \intXt \Big( G(\mubar^{m,\varphi}(s),y) - G(p(s),y) \Big) \lambda(ds \, dy) \\
		& = a(m)\sqrt{m} \int_0^t \Big( b(\mubar^{m,\varphi}(s)) - b(p(s)) \Big) ds \\
		C^{m,\varphi}(t) & \doteq \intXt G(\mubar^{m,\varphi}(s),y) \psi(s,y) \, \lambda(ds \, dy).
	\end{align*}
	Noting that $M^{m,\varphi}$ is a martingale, 
	Doob's inequality gives us
	\begin{equation*}
		\Emb \left\| M^{m,\varphi}\right\|_{*,T}^2 \le \frac{4a^2(m)}{m} \Emb \intXT \| G(\mubar^{m,\varphi}(s),y) \|^2 m \varphi(s,y) \, \lambda(ds \, dy).
	\end{equation*}
	It then follows from Lemma~\ref{lemjump:moment bounds on G square varphi} that
	\begin{equation} \label{eqjump:moment bounds on M}
		\sup_{\varphi \in \Umc^M_{+,m}} \Emb \left\| M^{m,\varphi}\right\|_{*,T}^2 \le \kappa_1 a^2(m).
	\end{equation}
%	Since $X^m_i(0)$ are drawn independently from $p(0)$, we have
%	\begin{equation} \label{eqjump:moment bounds on mu m 0 - p 0}
%		\Emb \| \mu^m(0) - p(0) \|^2 = \sum_{i=1}^\infty \Emb \| \mu^m_i(0) - p_i(0) \|^2 = \sum_{i=1}^\infty \frac{1}{m} p_i(0) \big( 1 - p_i(0) \big) \le \frac{1}{m} \, .
%	\end{equation}
%	So
%	\begin{equation} \label{eqjump:moment bounds on A m}
%		\sup_{\varphi \in \Umc^M_{+,m}} \Emb \Big[ \sup_{t \le T} \| A^{m,\varphi}(t) \|^2 \Big] \le a^2(m).
%	\end{equation}
	Using Cauchy-Schwarz inequality and Condition~\ref{condjump:cond1}(b) we have for all $\varphi \in \Umc_{+,m}^M$,
	\begin{equation*}
		\left\| B^{m,\varphi} \right\|_{*,t}^2 \le a^2(m) m T \int_0^t \| b(\mubar^{m,\varphi}(s)) - b(p(s)) \|^2 \, ds \le TL_b^2 \int_0^t \left\| \Zbar^{m,\varphi} \right\|_{*,s}^2 \, ds.
	\end{equation*}
	Since $\psi \in \Smc_m^M$ a.s., it follows from Lemma~\ref{lemjump:moment bounds on G psi} that
	\begin{equation} \label{eqjump:Cmphi}
		\sup_{\varphi \in \Umc^M_{+,m}} \left\| C^{m,\varphi}\right\|_{*,T}^2 \le \kappa_2 \Big( \frac{1}{a(m) \sqrt{m}} + \sqrt{T} \Big)^2 \le 2\kappa_2 \Big( \frac{1}{a^2(m) m} + T \Big).
	\end{equation}
	Collecting these estimates we have for some $\kappa_3 \in (0,\infty)$ and all $\varphi \in \Umc^M_{+,m}$, $t \in [0,T]$,
	\begin{equation*}
		\Emb \left\| \Zbar^{m,\varphi} \right\|_{*,t}^2 \le \kappa_3 \left( \|A^m\|^2 + a^2(m) + \frac{1}{a^2(m) m} + 1 + \int_0^t \Emb \left\| \Zbar^{m,\varphi} \right\|_{*,s}^2 \, ds \right).
	\end{equation*}
	The result now follows from Gronwall's inequality, \eqref{eqjump:a m} and Condition~\ref{condjump:cond2}(d).
\end{proof}

Although $G(q,y)$ is not a continuous map, using the specific form of $G$ and properties of $\Gamma$, we can establish the following Lipschitz property.

\begin{lemma} \label{lemjump:Lipschitz of G}
	There exits $\gammatil_5 \in (0,\infty)$ such that for all $g \in \Mmb_b(\Xmb)$ and all $q$, $\qtil \in \Smchat$,
	\begin{equation*}
		\left\| \int_\Xmb \Big( G(\qtil,y) - G(q,y) \Big) g(y) \, \lambda_\Xmb(dy) \right\| \le \gammatil_5 \|g\|_\infty \|\qtil - q\|.
	\end{equation*}
\end{lemma}

\begin{proof}
	Observing that $\lambda_\Xmb(A_{ij}(\qtil) \bigtriangleup A_{ij}(q)) = |\qtil_i \Gamma_{ij}(\qtil) - q_i \Gamma_{ij}(q)|$ for $i \ne j$, where ``$\bigtriangleup$" denotes the symmetric difference, we see
	\begin{align*}
		& \Big\| \int_\Xmb \Big( G(\qtil,y) - G(q,y) \Big) g(y) \, \lambda_\Xmb(dy) \Big\|^2 \\
%		& \quad = \sum_{i=1}^\infty \Big[ \int_\Xmb \Big( G_i(\qtil,y) - G_i(q,y) \Big) g(y) \, \lambda_\Xmb(dy) \Big]^2 \\
		& \quad \le \|g\|_\infty^2 \sum_{i=1}^\infty \Big( \int_\Xmb \Big| G_i(\qtil,y) - G_i(q,y) \Big| \, \lambda_\Xmb(dy) \Big)^2 \\
		& \quad \le \|g\|_\infty^2 \sum_{i=1}^\infty \Big( \sum_{j \ne i}^\infty | \qtil_i \Gamma_{ij}(\qtil) - q_i \Gamma_{ij}(q) | + \sum_{j \ne i}^\infty | \qtil_j \Gamma_{ji}(\qtil) - q_j \Gamma_{ji}(q) | \Big)^2 \\
%		& \quad \le 2 \|g\|_\infty^2 \Big[ \sum_{i=1}^\infty \Big( \sum_{j \ne i}^\infty | \qtil_i \Gamma_{ij}(\qtil) - q_i \Gamma_{ij}(q) | \Big)^2 + \sum_{i=1}^\infty \Big( \sum_{j \ne i}^\infty | \qtil_j \Gamma_{ji}(\qtil) - q_j \Gamma_{ji}(q) | \Big)^2 \Big] \\
		& \quad \le 4 \|g\|_\infty^2 \Big[ \sum_{i=1}^\infty \Big( \sum_{j \ne i}^\infty | \qtil_i \Gamma_{ij}(\qtil) - q_i \Gamma_{ij}(\qtil) | \Big)^2 + \sum_{i=1}^\infty \Big( \sum_{j \ne i}^\infty | q_i \Gamma_{ij}(\qtil) - q_i \Gamma_{ij}(q) | \Big)^2 \\
		& \qquad + \sum_{i=1}^\infty \Big( \sum_{j \ne i}^\infty | \qtil_j \Gamma_{ji}(\qtil) - q_j \Gamma_{ji}(\qtil) | \Big)^2 + \sum_{i=1}^\infty \Big( \sum_{j \ne i}^\infty | q_j \Gamma_{ji}(\qtil) - q_j \Gamma_{ji}(q) | \Big)^2 \Big] \\
		& \quad \equiv 4 \|g\|_\infty^2 \sum_{k=1}^4 \Tmc_k.
	\end{align*}
	The terms $\Tmc_k$ for $k = 1,2,3,4$, can be estimated as follows.
	\begin{align*}
		\Tmc_1 = \sum_{i=1}^\infty \Big[ (\qtil_i - q_i)^2 \Big( \sum_{j \ne i}^\infty \Gamma_{ij}(\qtil) \Big)^2 \Big] \le \|\Gamma\|_\infty^2 \|\qtil - q\|^2.
	\end{align*}
	Also, from Condition~\ref{condjump:cond2}(b),
%	\begin{align*}
%		\Tmc_2 \le \sum_{i=1}^\infty q_i^2 \Big\| \Gammabar(\qtil)[\ebd_i] - \Gammabar(q)[\ebd_i] \Big\|_1^2 \le L_\Gammabar^2 \|\qtil - q\|^2 \sum_{i=1}^\infty q_i^2 \le L_\Gammabar^2 \|\qtil - q\|^2.
%	\end{align*}
	\begin{align*}
		\Tmc_2 & \le \sum_{i=1}^\infty q_i^2 L_\Gamma^2 \|\qtil - q\|^2 \le L_\Gamma^2 \|\qtil - q\|^2, \\
		\Tmc_3 & = \sum_{i=1}^\infty \Big( \sum_{j \ne i}^\infty | \qtil_j - q_j | \Gamma_{ji}(\qtil) \Big)^2 \le \sum_{i=1}^\infty \Big[ \sum_{j \ne i}^\infty \Big( | \qtil_j - q_j |^2 \Gamma_{ji}(\qtil) \Big) \sum_{j \ne i}^\infty \Gamma_{ji}(\qtil) \Big] \\
		& \le c_\Gamma \sum_{i=1}^\infty \sum_{j \ne i}^\infty | \qtil_j - q_j |^2 \Gamma_{ji}(\qtil) \le c_\Gamma \Gammainf \|\qtil - q\|^2, \\
		\Tmc_4 & \le \Big( \sum_{i=1}^\infty \sum_{j \ne i}^\infty q_j | \Gamma_{ji}(\qtil) - \Gamma_{ji}(q) | \Big)^2 \le \Big( \sum_{j=1}^\infty q_j L_\Gamma \|\qtil - q\| \Big)^2 \le L_\Gamma^2 \|\qtil - q\|^2.
	\end{align*}
%	It follows from Cauchy-Schwarz inequality and Condition~\ref{condjump:cond2}(b) that
%	\begin{align*}
%		\Tmc_4 & \le \sum_{i=1}^\infty \sum_{j \ne i}^\infty q_j | \Gamma_{ji}(\qtil) - \Gamma_{ji}(q) |^2 \le 2 \Gammainf \sum_{j=1}^\infty \Big( q_j \sum_{i \ne j}^\infty | \Gamma_{ji}(\qtil) - \Gamma_{ji}(q) | \Big) \\
%		& \le 2 \Gammainf \sum_{j=1}^\infty \Big\|\Gammabar(\qtil)[q_j \ebd_j] - \Gammabar(q)[q_j \ebd_j] \Big\|_1 \le 2 \Gammainf L_\Gammabar \|\qtil - q\| \sum_{j=1}^\infty q_j = 2 \Gammainf L_\Gammabar \|\qtil - q\|.
%	\end{align*}
	The result follows by combining above estimates.
\end{proof}

The following lemma will allow us to apply the continuous mapping theorem to deduce the key weak convergence property in the proof of Proposition~\ref{propjump:verifying condition part b}.

\begin{lemma} \label{lemjump:h}
	Let $M \in (0,\infty)$.
	Given $\varepsilon \in \Dmb([0,T] : l_2)$ and $f \in B_2(M)$, there exists a unique $z \in \Dmb([0,T] : l_2)$ solving the following equation:
	\begin{equation} \label{eqjump:h}
		z(t) = \varepsilon(t) + \int_0^t Db(p(s))[z(s)] \, ds + \intXt G(p(s),y) f(s,y) \, \lambda(ds \, dy), \quad t \in [0,T], 
	\end{equation}
	namely there exists a measurable map $h \colon \Dmb([0,T] : l_2) \times B_2(M) \to \Dmb([0,T] : l_2)$ such that the solution to \eqref{eqjump:h} can be written as $z = h(\varepsilon,f)$.
	Moreover, $h$ is continuous at $(0,f)$ for every $f \in B_2(M)$.
\end{lemma}

\begin{proof}
	The existence and uniqueness of solutions of \eqref{eqjump:h} and the measurability of the solution map are easy to check using Condition~\ref{condjump:cond2}(c) in a manner similar to the proof of Theorem~\ref{thmjump:LLN}(b).
	To see the continuity at $(0,f)$ for $f \in B_2(M)$, first note that \eqref{eqjump:h} can be written as
	\begin{equation*}
		z(t) - \varepsilon(t) = \int_0^t Db(p(s))[z(s)-\varepsilon(s)] \, ds + \int_0^t Db(p(s))[\varepsilon(s)] \, ds + \intXt G(p(s),y) f(s,y) \, \lambda(ds \, dy).
	\end{equation*}
	Suppose $(\varepsilon^m,f^m) \to (0,f)$ in $\Dmb([0,T] : l_2) \times B_2(M)$ as $m \to \infty$.
	Let $z^m \doteq h(\varepsilon^m,f^m)$ and $z \doteq h(0,f)$.
	Since $\varepsilon^m \to 0$ in $\Dmb([0,T] : l_2)$ we see using Condition~\ref{condjump:cond2}(c) that
%	$\sup_{t \le T} \|\varepsilon^m(t)\| \to 0$ and hence
	\begin{equation*}
		\int_0^\cdot Db(p(s))[\varepsilon^m(s)] \, ds \to 0 \text{ in } \Cmb([0,T] : l_2).
	\end{equation*}
	It follows from Lemma~\ref{lemjump:DCT} that
	\begin{equation*}
		\int_{[0,\cdot] \times \Xmb} G(p(s),y) f^m(s,y) \, \lambda(ds \, dy) \to \int_{[0,\cdot] \times \Xmb} G(p(s),y) f(s,y) \, \lambda(ds \, dy) \text{ in } \Cmb([0,T] : l_2).
	\end{equation*}
	Combining above results and applying Gronwall's lemma gives us $z^m - \varepsilon^m \to z - 0$ in $\Cmb([0,T] : l_2)$.
	Since $\varepsilon^m \to 0$, we have that $z^m \to z$ in $\Dmb([0,T] : l_2)$ and the result follows.
\end{proof}

We can now verify part (b) of Condition~\ref{condjump:cond3 to be checked}.
Recall that for $M \in (0,\infty)$, $C_M = \sqrt{M\gammatil_2(1)}$.

\begin{proposition} \label{propjump:verifying condition part b}
	Fix $M \in (0,\infty)$.
	Let $\{ \varphi^m \}_{m \in \Nmb}$ be such that for every $m \in \Nmb$, $\varphi^m \in \Umc^M_{+,m}$.
	Let $\psi^m \one_{\{ | \psi^m | \le \beta a(m)\sqrt{m} \}} \Rightarrow \psi$ in $B_2(C_M)$ for some $\beta \in (0,1]$, where $\psi^m \doteq a(m)\sqrt{m}(\varphi^m-1)$.
	Let $\Gmc_0$ and $\Gmc^m$ be as defined in and above \eqref{eqjump:G 0}, respectively.
	Then $\Gmc^m(\frac{1}{m} N^{m\varphi^m}) \Rightarrow \Gmc_0(\psi)$.
\end{proposition}

\begin{proof}
	We will use the notation from the proof of Lemma~\ref{lemjump:moment bounds on Z bar}.
	From \eqref{eqjump:a m}, \eqref{eqjump:moment bounds on M} and Condition~\ref{condjump:cond2}(d) we have that
	\begin{equation*}
		\Emb \left\| M^{m,\varphi^m} \right\|_{*,T}^2 \to 0 \text{ and } \| A^m \|^2 \to 0
	\end{equation*}
	as $m \to \infty$.
	It follows from Condition~\ref{condjump:cond2}(c) that
	\begin{equation*}
		b(\mubar^{m,\varphi^m}(s)) - b(p(s)) = \frac{1}{a(m)\sqrt{m}} Db(p(s))[\Zbar^{m,\varphi^m}(s)] + R^{m,\varphi^m}(s),
	\end{equation*}
	where
	\begin{equation} \label{eqjump:R m}
		\| R^{m,\varphi^m}(s) \| \doteq \| \theta_b(p(s),\mubar^{m,\varphi^m}(s)) \| \le \frac{c_b}{a^2(m) m} \| \Zbar^{m,\varphi^m}(s) \|^2.
	\end{equation}
	Hence $B^{m,\varphi^m} = \Btil^{m,\varphi^m} + \Emc_1^{m,\varphi^m}$, where
	\begin{align}
		\Btil^{m,\varphi^m}(t) & \doteq \int_0^t Db(p(s))[\Zbar^{m,\varphi^m}(s)] \, ds, \label{eqjump:B til} \\
		\Emc_1^{m,\varphi^m}(t) & \doteq a(m)\sqrt{m} \int_0^t R^{m,\varphi^m}(s) \, ds. \notag
	\end{align}
	From \eqref{eqjump:R m} and Lemma~\ref{lemjump:moment bounds on Z bar} we see that
	\begin{equation*}
		\Emb \left\| \Emc_1^{m,\varphi^m}\right\|_{*,T} \le a(m)\sqrt{m} \Emb \int_0^T \| R^{m,\varphi^m}(s) \| \, ds \le \frac{c_b T}{a(m)\sqrt{m}} \Emb \left\| \Zbar^{m,\varphi^m}\right\|_{*,T}^2 \to 0
	\end{equation*}
	as $m \to \infty$.
	Write $C^{m,\varphi^m} = \Ctil^{m,\varphi^m} + \Emc_2^{m,\varphi^m} + \Emc_3^{m,\varphi^m} + \Emc_4^{m,\varphi^m}$, where
	\begin{align}
		\Ctil^{m,\varphi^m}(t) & \doteq \intXt G(p(s),y) \psi^m(s,y) \one_{\{|\psi^m(s,y)| \le \beta a(m) \sqrt{m} \}} \, \lambda(ds \, dy), \label{eqjump:C til} \\
		\Emc_2^{m,\varphi^m}(t) & \doteq \intXt G(p(s),y) \psi^m(s,y) \one_{\{|\psi^m(s,y)| > \beta a(m) \sqrt{m} \}} \, \lambda(ds \, dy), \notag \\
		\Emc_3^{m,\varphi^m}(t) & \doteq \intXt \Big( G(\mubar^{m,\varphi^m}(s),y) - G(p(s),y) \Big) \psi^m(s,y) \one_{\{|\psi^m(s,y)| > \delta_m \}} \, \lambda(ds \, dy), \notag \\
		\Emc_4^{m,\varphi^m}(t) & \doteq \intXt \Big( G(\mubar^{m,\varphi^m}(s),y) - G(p(s),y) \Big) \psi^m(s,y) \one_{\{|\psi^m(s,y)| \le \delta_m \}} \, \lambda(ds \, dy), \notag
	\end{align}
	and $\delta_m \doteq (a(m)\sqrt{m})^{1/2} \to \infty$ as $m \to \infty$.
	Then using Lemma~\ref{lemjump:moment bounds on G psi} we see that
	\begin{equation} \label{eqjump:Emc2}
		\left\| \Emc_2^{m,\varphi^m}\right\|_{*,T} \le \frac{\gammatil_4(\beta)}{a(m) \sqrt{m}} \to 0
	\end{equation}
	as $m \to \infty$.	
	Also applying Lemma~\ref{lemjump:moment bounds of varphi and psi}(a) with $\beta = \frac{\delta_m}{a(m)\sqrt{m}}$, we see that as $m \to \infty$,
	\begin{align*}
		\left\| \Emc_3^{m,\varphi^m}\right\|_{*,T} & \le \intXT \left\| G(\mubar^{m,\varphi^m}(s),y) - G(p(s),y) \right\| |\psi^m(s,y)| \one_{\{|\psi^m(s,y)| > \delta_m\}} \, \lambda(ds \, dy) \\
		& \le 2\sqrt{2} \intXT |\psi^m(s,y)| \one_{\{|\psi^m(s,y)| > \delta_m\}} \, \lambda(ds \, dy) \\
		& \le \frac{2 \sqrt{2} M \gammatil_1(\frac{\delta_m}{a(m) \sqrt{m}})}{a(m) \sqrt{m}} \le \frac{8 \sqrt{2} M}{\delta_m} \to 0,
	\end{align*}
	where the last inequality is a consequence of the last statement in Lemma~\ref{lemjump:property of l}(a).
	Next, it follows from Lemma~\ref{lemjump:Lipschitz of G} that
	\begin{align*} 
		\left\| \Emc_4^{m,\varphi^m}\right\|_{*,T} & \le \int_0^T \Big\| \int_\Xmb \Big( G(\mubar^{m,\varphi^m}(s),y) - G(p(s),y) \Big) \psi^m(s,y) \one_{\{|\psi^m(s,y)| \le \delta_m \}} \, \lambda_\Xmb(dy) \Big\| \, ds \\
		& \le \gammatil_5 \delta_m \int_0^T \| \mubar^{m,\varphi^m}(s) - p(s)\| \, ds \le \gammatil_5 T \frac{\delta_m}{a(m) \sqrt{m}} \left\| \Zbar^{m,\varphi^m} \right\|_{*,T}.
	\end{align*}
	Since $\frac{\delta_m}{a(m) \sqrt{m}} = (a(m) \sqrt{m})^{-\frac{1}{2}} \to 0$, it follows from Lemma~\ref{lemjump:moment bounds on Z bar} that $\Emb \left\| \Emc_4^{m,\varphi^m}\right\|_{*,T} \to 0$ as $m \to \infty$.
	Putting above estimates together we have from \eqref{eqjump:decomposition of Zbar}
	\begin{align} \label{eqjump:Zbar m varphi m}
		\begin{aligned}
			\Zbar^{m,\varphi^m}(t) & = \Emc^{m,\varphi^m}(t) + \Btil^{m,\varphi^m}(t) + \Ctil^{m,\varphi^m}(t) \\
			& = \Emc^{m,\varphi^m}(t) + \int_0^t Db(p(s))[\Zbar^{m,\varphi^m}(s)] \, ds \\
			& \qquad + \intXt G(p(s),y) \psi^m(s,y) \one_{\{|\psi^m| \le \beta a(m) \sqrt{m} \}} \, \lambda(ds \, dy),
		\end{aligned}
	\end{align}
	where $\Emc^{m,\varphi^m} \doteq M^{m,\varphi^m} + A^m + \Emc_1^{m,\varphi^m} + \Emc_2^{m,\varphi^m} + \Emc_3^{m,\varphi^m} + \Emc_4^{m,\varphi^m} \Rightarrow 0$ in $\Dmb([0,T]:l_2)$.
	Thus we have
	\begin{equation*}
		\Gmc^m\left(\frac{1}{m} N^{m\varphi^m}\right) = \Zbar^{m,\varphi^m} = h(\Emc^{m,\varphi^m}, \psi^m \one_{\{|\psi^m| \le \beta a(m) \sqrt{m} \}}),
	\end{equation*}
	where $h$ is as introduced in Lemma~\ref{lemjump:h}.
	It follows from Lemma~\ref{lemjump:moment bounds of varphi and psi}(c) that $\psi^m \one_{\{|\psi^m| \le \beta a(m) \sqrt{m} \}}$ takes values in $B_2(C_M)$ for all $m \in \Nmb$.
	Finally note that $\Gmc_0(\psi) = h(0,\psi)$ and
	\begin{equation*}
		(\Emc^{m,\varphi^m}, \psi^m \one_{\{|\psi^m| \le \beta a(m) \sqrt{m} \}}) \Rightarrow (0,\psi).
	\end{equation*}
	The result now follows by combining above observations and applying continuous mapping theorem together with Lemma~\ref{lemjump:h}.
\end{proof}

Now we can complete the proof of Theorem~\ref{thmjump:MDP}.

\noindent \textbf{Proof of Theorem~\ref{thmjump:MDP}:}
As noted earlier, it suffices to show that Condition~\ref{condjump:cond3 to be checked} holds with $\Gmc^m$ and $\Gmc_0$ above and in \eqref{eqjump:G 0}, respectively.
Part (a) of the condition was verified in Proposition~\ref{propjump:verifying condition part a}, while part (b) was verified in Proposition~\ref{propjump:verifying condition part b}.
\qed

%----------------------------------------------------

\subsection{Proof of Theorem~\ref{thmjump:alternative expression}} \label{secjump:pfalter}

Fix $\eta \in \Dmb([0,T]:l_2)$.
We first argue that $\Ibar(\eta) \le I(\eta)$. 
Let $\delta>0$ be arbitrary.
Let $u \doteq \{u_{ij}\}_{i,j=1}^\infty$ be such that $u_{ij} \in L^2([0,T]:\Rmb)$,
\begin{equation*}
	\half \int_0^T \sum_{i=1}^\infty \sum_{j=1,j \ne i}^\infty u_{ij}^2(s) \, ds \le I(\eta) + \delta,
\end{equation*}
and $(\eta,u)$ satisfies \eqref{eqjump:eta and u}.
Define $\psi \colon \XT \to \Rmb$ by
\begin{equation*} 
%\label{eqjump:psi}
	\psi(s,y) \doteq \sum_{i=1}^\infty \sum_{j=1,j \ne i}^\infty \one_{A_{ij}(p(s))}(y) \frac{u_{ij}(s)}{\sqrt{p_i(s)\Gamma_{ij}(p(s))}} \one_{\{p_i(s)\Gamma_{ij}(p(s)) \ne 0\}}, \quad (s,y) \in [0,T] \times \Xmb.
\end{equation*}
Then we have
\begin{equation*}
	\int_{\XT} \psi^2(s,y) \, \lambda(ds\,dy) = \int_0^T \sum_{i=1}^\infty \sum_{j=1,j \ne i}^\infty u_{ij}^2(s) \one_{\{p_i(s) \Gamma_{ij}(p(s)) \ne 0\}} \, ds < \infty
\end{equation*}
and hence $\psi \in L^2(\lambda)$.
Also note that for $q \in \Smchat$ and $i \ne j$,
\begin{equation*}
	\int_\Xmb G(q,y) \one_{A_{ij}(q)}(y) \, \lambda(dy) = (\ebd_j - \ebd_i) q_i \Gamma_{ij}(q).
\end{equation*}
From this it follows that $(\eta,\psi)$ satisfies \eqref{eqjump:eta and psi}.
Thus
\begin{equation*}
	\Ibar(\eta) \le \half \int_{\XT} \psi^2(s,y) \, \lambda(ds\,dy) = \half \int_0^T \sum_{i=1}^\infty \sum_{j=1,j \ne i}^\infty u_{ij}^2(s) \one_{\{p_i(s) \Gamma_{ij}(p(s)) \ne 0\}} \, ds \le I(\eta) + \delta.
\end{equation*}
Since $\delta > 0$ is arbitrary, we have that $\Ibar(\eta) \le I(\eta)$.

Conversely, suppose $\psi \in L^2(\lambda)$ is such that
\begin{equation*}
	\frac{1}{2} \intXT \psi^2(s,y) \, \lambda(ds\,dy) \le \Ibar(\eta) + \delta
\end{equation*}
and \eqref{eqjump:eta and psi} holds.
%We can further suppose without loss of generality that $\psi(s,y) = 0$ if $p_i(s)\Gamma_{ij}(p(s)) = 0$ and $y \in (i-1,i]\times((j-1)\Gammainf,j\Gammainf]$.
For $i,j \in \Nmb$ with $i \ne j$ and $s \in [0,T]$, define $u_{ij} \colon [0,T] \to \Rmb$ by
\begin{equation*}
	u_{ij}(s) \doteq \frac{\int_\Xmb \one_{A_{ij}(p(s))}(y) \psi(s,y) \, \lambda_\Xmb(dy)}{\sqrt{p_i(s)\Gamma_{ij}(p(s))}} \one_{\{p_i(s) \Gamma_{ij}(p(s)) \ne 0\}}.
\end{equation*}
An application of Cauchy-Schwarz inequality shows that 
\begin{equation*}
	u_{ij}^2(s) \le \int_\Xmb \one_{A_{ij}(p(s))}(y) \psi^2(s,y) \, \lambda_\Xmb(dy),
\end{equation*}
and hence $u_{ij} \in L^2([0,T]:\Rmb)$ for all $i \ne j$.
We set $u_{ii} \doteq 0$ for $i \in \Nmb$ and let $u \doteq \{u_{ij}\}_{i,j=1}^\infty$.
It is easy to check that $(\eta,u)$ satisfies \eqref{eqjump:eta and u}, and hence
\begin{align*}
	I(\eta) & \le \half \int_0^T \sum_{i=1}^\infty \sum_{j=1,j \ne i}^\infty u_{ij}^2(s) \, ds \le \half \int_{\XT} \sum_{i=1}^\infty \sum_{j=1,j \ne i}^\infty \one_{A_{ij}(p(s))}(y) \psi^2(s,y) \, \lambda(ds\,dy) \\
	& \le \half \intXT \psi^2(s,y) \, \lambda(ds\,dy) \le \Ibar(\eta) + \delta.
\end{align*}
Since $\delta > 0$ is arbitrary we have $I(\eta) \le \Ibar(\eta)$.
The result follows.
\qed

\appendix

\section{Proof of Theorem~\ref{thmjump:LLN}} \label{Appendix:proof of theorem of LLN}

Part (a) can be established using a recursive construction of the solution from one jump to the next.
Note that although $\ti\Emb \nbd_m(\Xt) = \infty$ for all $t>0$, the property that
\begin{equation*}
	\lambda_\Xmb\left(\bigcup_{i=1}^\infty \bigcup_{j=1,j \ne i}^\infty A_{ij}(q)\right) \le \sum_{i=1}^\infty \sum_{j=1,j \ne i}^\infty q_i \Gamma_{ij}(q) \le \|\Gamma\|_\infty < \infty,\quad \forall q \in \Smchat,
\end{equation*} 
allows one to enumerate the jump instants $t$ at which the state of $\mutil^m(t)$ changes.
At any such jump instant we define $\mutil^m(t) \doteq \mutil^m(t-) + \frac{1}{m} G(\mutil^m(t-),y)$ if the jump corresponds to the point $(t,y)$ of the point process $\nbd_m$.
We omit the details.

For part (b), uniqueness of solution of \eqref{eqjump:p t} follows from an application of Gronwall's lemma along with Condition~\ref{condjump:cond1}(b).
For existence of solution we follow a standard iteration scheme.
Define $p^0(t) \doteq p(0)$ and for $n \in \Nmb$,
\begin{equation*}
	p^{n+1}(t) \doteq p(0) + \int_0^t b(p^n(s)) \, ds, \quad t \in [0,T].
\end{equation*}
From Condition~\ref{condjump:cond1}(b) we see that 
\begin{equation*}
	\|p^{n+1}-p^n\|_{*,t} = \left\|\int_0^\cdot (b(p^n(s)) - b(p^{n-1}(s))) \, ds \right\|_{*,t} \le L_b \int_0^t \|p^n-p^{n-1}\|_{*,s} \, ds,
\end{equation*}
which implies $\{p^n\}_{n=0}^\infty$ is a Cauchy sequence in $\Cmb([0,T]:l_2)$.
Hence there exists some $\ptil \in \Cmb([0,T]:l_2)$ such that $p^n \to \ptil$ and it is easy to see that $\ptil$ is a solution to \eqref{eqjump:p t}.

We now argue that $\mu^m \Rightarrow p$ as $m \to \infty$. 
For $t \in [0,T]$,
\begin{align*}
	\Emb \sup_{s \le t} \| \mu^m(s) - p(s) \|^2 & \le 3 \| \mu^m(0) - p(0) \|^2 + 3 \Emb \sup_{s \le t} \Big\| \int_0^s \big( b(\mu^m(u)) - b(p(u)) \big) \, du \Big\|^2 \\
	& \qquad + 3 \Emb \sup_{s \le t} \Big\| \frac{1}{m} \int_{\Xmb_s} G(\mu^m(u-),y) \, \Ntil^m(du\,dy) \Big\|^2 \\
	& \le 3 \| \mu^m(0) - p(0) \|^2 + 3T \int_0^t \Emb \| b(\mums) - b(p(s)) \|^2 \, ds \\
	& \qquad + \frac{12}{m} \intXt \Emb \| G(\mums,y) \|^2 \, \lambda(ds \, dy) \\
%	& \le \kappa_2 \| \mu^m(0) - p(0) \|^2 + \kappa_2 \int_0^t \Emb \| \mums - p(s) \|^2 \, ds + \frac{\kappa_2}{m} \\
	& \le \kappa \| \mu^m(0) - p(0) \|^2 + \frac{\kappa}{m} + \kappa \int_0^t \Emb \sup_{u \le s} \| \mu^m(u) - p(u) \|^2 \, ds,
\end{align*}
where the second inequality follows from Doob's inequality and the third inequality follows from Lemma~\ref{lemjump:moment bounds on G}.
The result now follows from Gronwall's inequality and Condition~\ref{condjump:cond1}(c). \qed
%We get from Gronwall's inequality and Condition~\ref{condjump:cond1}(c) that
%\begin{equation*}
%	\pushQED{\qed} \Emb \sup_{s \le T} \| \mu^m(s) - p(s) \|^2 \le \kappa_3 \| \mu^m(0) - p(0) \|^2 + \frac{\kappa_3}{m} \to 0. \qedhere \popQED
%\end{equation*}

\section{Proof of Remark~\ref{rmkjump:rmk for cond2}(iv)} \label{Appendix:proof of rmk}

Suppose that for some $n \in \Nmb$, $\sum_{m=1}^\infty [a(m)]^{2n} < \infty$.
We need to show that $a(m) \sqrt{m} \| \mu^m(0)-p(0) \| \to 0$ almost surely.
To simplify the notation, we will abbreviate $\mu^m(0), p(0), \mu^m_i(0), p_i(0)$ as $\mu^m, p, \mu^m_i, p_i$.
It follows from Markov's inequality that for $\eps > 0$,
\begin{equation*}
	\Pmb(a(m) \sqrt{m} \| \mu^m-p \| > \eps) \le \left(\frac{a(m) \sqrt{m}}{\eps}\right)^{2n} \Emb \| \mu^m-p \|^{2n} = \frac{[a(m)]^{2n}}{\eps^{2n}} m^n \Emb \left[\sum_{i=1}^\infty (\mu_i^m-p_i)^2\right]^n.
\end{equation*}
Since $\sum_{m=1}^\infty [a(m)]^{2n} < \infty$, by Borel-Cantelli lemma, it suffices to show that for every $n \in \Nmb$ there exists some $\hat\gamma_n \in (0,\infty)$ such that
\begin{equation} \label{eqApp:key estimate}
	\Emb \left[\sum_{i=1}^\infty (\mu_i^m-p_i)^2\right]^n \le \frac{\hat\gamma_n}{m^n}.
\end{equation}
We will prove \eqref{eqApp:key estimate} when $n=2$ in detail and then sketch the argument for $n > 2$.
Write
\begin{equation*}
	\mu_i^m-p_i = \frac{1}{m} \sum_{j=1}^m \one_{\{\xi_j=i\}} - p_i \equiv \frac{1}{m} \sum_{j=1}^m Y_{ij},
\end{equation*}
where $Y_{ij} \doteq \one_{\{\xi_j=i\}} - p_i$ is such that for $\alpha, \beta \in \Nmb$,
\begin{equation*}
	|Y_{ij}| \le 1, \quad \Emb Y_{ij} = 0, \quad \Emb Y_{ij}^{2\alpha} \le \Emb Y_{ij}^2 \le p_i, \quad \Emb Y_{ij}^{2\alpha} Y_{kl}^{2\beta} \le \Emb Y_{ij}^2 Y_{kl}^2 \le p_i p_k \text{ for } i \ne k.
\end{equation*}
Now write 
\begin{align}
	\Emb \left[\sum_{i=1}^\infty (\mu_i^m-p_i)^2\right]^2 & = \Emb \left[\sum_{i=1}^\infty \left(\frac{1}{m} \sum_{j=1}^m Y_{ij}\right)^2\right]^2 = \frac{1}{m^4} \Emb \left[ \sum_{i=1}^\infty \sum_{j,j'=1}^m Y_{ij} Y_{ij'} \right]^2 \notag \\
	& = \frac{1}{m^4} \Emb \sum_{i,k=1}^\infty \sum_{j,j',l,l'=1}^m Y_{ij} Y_{ij'} Y_{kl} Y_{kl'}. \label{eqApp:key decomp}
\end{align}
From independence of $\{\xi_j\}$ it follows that $Y_{ij}$ and $Y_{kl}$ are independent for $j \ne l$.
Hence $\Emb Y_{ij} Y_{ij'} Y_{kl} Y_{kl'} \ne 0$ only if $j,j',l,l'$ are matched in pairs (e.g.\ $j=j'$ and $l=l'$).
Using this observation, \eqref{eqApp:key decomp} can be written as
\begin{equation*}
	\frac{1}{m^4} \Emb \sum_{r=1}^5 \sum_{(i,k,j,j',l,l') \in \Hmc_r} Y_{ij} Y_{ij'} Y_{kl} Y_{kl'} \equiv \sum_{r=1}^5 \Tmc_r^m,
\end{equation*}
where $\Hmc_1$, $\Hmc_2$, $\Hmc_3$, $\Hmc_4$ and $\Hmc_5$ are collections of $(i,k) \in \Nmb^2$ and $(j,j',l,l') \in \{1,\dotsc,m\}^4$ such that $\{ j = j' \ne l = l' \}$, $\{ j = l \ne j' = l' \}$, $\{ j = l' \ne j' = l \}$, $\{ j = j' = l = l', i = k \}$ and $\{ j = j' = l = l', i \ne k \}$, respectively.
For $\Tmc_1^m$, it follows from independence of $\{\xi_j\}$ that
\begin{equation*}
	\Tmc_1^m = \frac{1}{m^4} \Emb \sum_{i,k=1}^\infty \sum_{\substack{j,l=1 \\ j \ne l}}^m Y_{ij}^2 Y_{kl}^2 = \frac{1}{m^4} \sum_{i,k=1}^\infty \sum_{\substack{j,l=1 \\ j \ne l}}^m \Emb Y_{ij}^2 \Emb Y_{kl}^2 \le \frac{1}{m^2} \sum_{i,k=1}^\infty p_i p_k = \frac{1}{m^2}.
\end{equation*}
For $\Tmc_2^m$, using independence of $\{\xi_j\}$ and Cauchy-Schwarz inequality we have
\begin{align*}
	\Tmc_2^m & = \frac{1}{m^4} \Emb \sum_{i,k=1}^\infty \sum_{\substack{j,l'=1 \\ j \ne l'}}^m Y_{ij} Y_{il'} Y_{kj} Y_{kl'} = \frac{1}{m^4} \sum_{i,k=1}^\infty \sum_{\substack{j,l'=1 \\ j \ne l'}}^m \Emb \left( Y_{ij} Y_{kj} \right) \Emb \left( Y_{il'} Y_{kl'} \right) \\
	& \le \frac{1}{m^4} \sum_{i,k=1}^\infty \sum_{\substack{j,l'=1 \\ j \ne l'}}^m \sqrt{\Emb Y_{ij}^2 \Emb Y_{kj}^2 \Emb Y_{il'}^2 \Emb Y_{kl'}^2} \le \frac{1}{m^2} \sum_{i,k=1}^\infty p_i p_k = \frac{1}{m^2}.
\end{align*}
Similarly, $\Tmc_3^m \le \frac{1}{m^2}$.
For $\Tmc_4^m$,
\begin{equation*}
	\Tmc_4^m = \frac{1}{m^4} \Emb \sum_{i=1}^\infty \sum_{j=1}^m Y_{ij}^4 \le \frac{1}{m^3} \sum_{i=1}^\infty p_i = \frac{1}{m^3}.	
\end{equation*}
Finally for $\Tmc_5^m$, we have
\begin{equation*}
	\Tmc_5^m = \frac{1}{m^4} \Emb \sum_{\substack{i,k=1 \\ i \ne k}}^\infty \sum_{j=1}^m Y_{ij}^2 Y_{kj}^2 \le \frac{1}{m^3} \sum_{\substack{i,k=1 \\ i \ne k}}^\infty p_i p_k \le \frac{1}{m^3}.
\end{equation*}
Combining above estimates we see that \eqref{eqApp:key decomp} is bounded by $\frac{3}{m^2}+\frac{2}{m^3}$.
This proves \eqref{eqApp:key estimate} when $n=2$.

For the case $n>2$, write 
\begin{align}
	\Emb \left[\sum_{i=1}^\infty (\mu_i^m-p_i)^2\right]^n & = \Emb \left[\sum_{i=1}^\infty \left(\frac{1}{m} \sum_{j=1}^m Y_{ij}\right)^2 \right]^n = \frac{1}{m^{2n}} \Emb \left[ \sum_{i=1}^\infty \sum_{j=1}^m \sum_{k=1}^m Y_{ij} Y_{ik} \right]^n \notag \\
	& = \frac{1}{m^{2n}} \Emb \sum_{i_1,\dotsc,i_n=1}^\infty \sum_{j_1,k_1,\dotsc,j_n,k_n=1}^m Y_{i_1j_1} Y_{i_1k_1} \dotsm Y_{i_nj_n} Y_{i_nk_n}. \label{eqApp:key decomp general}
\end{align}
Once again $\Emb \left( Y_{i_1j_1} Y_{i_1k_1} \dotsm Y_{i_nj_n} Y_{i_nk_n} \right) \ne 0$ only if $j_1,k_1,\dotsc,j_n,k_n$ are matched in pairs.
Hence the $2n$-fold summation over $j_1,k_1,\dotsc,j_n,k_n$ in \eqref{eqApp:key decomp general} can be reduced to no more than an $n$-fold sum.
We can break up the outer sum into $n$ terms where the $M$-th term, $M=1,\dotsc,n$, corresponds to indices $(i_1,\dotsc,i_n)$ of which exactly $M$ indices are distinct.
Similarly the inner sum can be split into $n$ terms where the $N$-th term, $N=1,\dotsc,n$, corresponds to indices $(j_1,k_1,\dotsc,j_n,k_n)$ matched in pairs with exactly $N$ distinct pairs.
Furthermore each such $(M,N)$-term can be split into a finite number of terms, each of which corresponds to a collection $\{c_{\alpha\beta}, \alpha=1,\dotsc,M,\beta=1,\dotsc,N\}$ of non-negative integers representing how $\{Y_{ij}\}$ is paired up, with $\sum_{\alpha=1}^M c_{\alpha\beta} \ge 1$, $\sum_{\beta=1}^N c_{\alpha\beta} \ge 1$ and $\sum_{\alpha=1}^M \sum_{\beta=1}^N c_{\alpha\beta} = 2n$.
By independence of $\{\xi_j\}$, the contribution of each such $(M,N,\{c_{\alpha\beta}\})$-term to \eqref{eqApp:key decomp general} is at most
\begin{equation} \label{eqApp:key decomp general example}
	\frac{\kappa_n}{m^{2n}} \sum_{\substack{i_1,\dotsc,i_M=1\\i_1,\dotsc,i_M \text{ distinct}}}^\infty \sum_{\substack{j_1,\dotsc,j_N=1\\j_1,\dotsc,j_N \text{ distinct}}}^m \Emb \left( Y_{i_1j_1}^{c_{11}} Y_{i_2j_1}^{c_{21}} \dotsm Y_{i_Mj_1}^{c_{M1}} \right) \dotsm \Emb \left( Y_{i_1j_N}^{c_{1N}} Y_{i_2j_N}^{c_{2N}} \dotsm Y_{i_Mj_N}^{c_{MN}} \right),
\end{equation}
where $\kappa_n \in (0,\infty)$ only depends on $n$.
A simple calculation gives that for all $\beta = 1,\dotsc,N$,
\begin{equation*}
	\left|\Emb \left( Y_{i_1j_\beta}^{c_{1\beta}} Y_{i_2j_\beta}^{c_{2\beta}} \dotsm Y_{i_Mj_\beta}^{c_{M\beta}} \right)\right| \le \ti \kappa_n p_{i_1}^{c_{1\beta} \wedge 1} \dotsm p_{i_M}^{c_{M\beta} \wedge 1},
\end{equation*}
where $\ti \kappa_n \in (0,\infty)$ only depends on $n$.
Hence \eqref{eqApp:key decomp general example} is bounded by $\frac{\kappa_n \ti \kappa_n^N}{m^{2n-N}} \le \frac{\kappa_n \ti \kappa_n^n}{m^n}$.
So \eqref{eqApp:key decomp general} is bounded by $\frac{\gammatil_n}{m^n}$ for some $\gammatil_n \in (0,\infty)$, which gives \eqref{eqApp:key estimate} and completes the proof. \qed

\def\bibfont{\footnotesize}

% %\bigskip
% 
% %\setcounter{equation}{0}
% %\appendix
% %\numberwithin{equation}{section}
% %\section{Auxiliary Results}
% %\subsection{Proof of Lemma~\ref{lemdif:lem916}.}
% %\AB{Put In.}
% 
% %%%%%%%%%%%%%%%%%%%%%%%%%%%%%%%%%%%%%%%%%%%%%%%%%
% 
% 
% \bibliographystyle{plain}
% \bibliography{reference}
% %\bibliography{F:/Study/UNC/Research/latex_template/reference}

\footnotesize{{\sc
\bigskip
\noi
A. Budhiraja (email: budhiraj@email.unc.edu)\\ 
R. Wu (email: wuruoyu@live.unc.edu)\\
Department of Statistics and Operations Research\\
University of North Carolina\\
Chapel Hill, NC 27599, USA\\
%email: budhiraj@email.unc.edu
% \skp
% 
% \noi
% R. Wu\\
% Department of Statistics and Operations Research\\
% University of North Carolina\\
% Chapel Hill, NC 27599, USA\\
% email: wuruoyu@live.unc.edu

}}

\end{document}